\documentclass[10pt]{amsart}
\usepackage{amssymb}
\usepackage{graphicx}
\usepackage{xcolor} 
\usepackage{tensor}
\usepackage{fullpage} 
\usepackage{amsmath}
\usepackage{amsthm}
\usepackage{verbatim}
\usepackage{yhmath}
\DeclareFontFamily{U}{mathx}{\hyphenchar\font45}
\DeclareFontShape{U}{mathx}{m}{n}{
	<5> <6> <7> <8> <9> <10>
	<10.95> <12> <14.4> <17.28> <20.74> <24.88>
	mathx10
}{}
\DeclareSymbolFont{mathx}{U}{mathx}{m}{n}
\DeclareFontSubstitution{U}{mathx}{m}{n}
\DeclareMathAccent{\widecheck}{0}{mathx}{"71}
\DeclareMathAccent{\wideparen}{0}{mathx}{"75}

\usepackage{enumitem}
\setlist[enumerate]{leftmargin=1.5em}
\setlist[itemize]{leftmargin=1.5em}

\usepackage{seqsplit}

\definecolor{green}{rgb}{0,0.8,0} 

\newtheorem{maintheorem}{Theorem}

\newtheorem{theorem}{Theorem}[section]

\newtheorem{lemma}[theorem]{Lemma}
\newtheorem{proposition}[theorem]{Proposition}

\theoremstyle{definition}

\theoremstyle{remark}
\newtheorem{remark}[theorem]{Remark}

\numberwithin{equation}{section}

\makeatletter
\newcommand{\nrm}{\@ifstar{\nrmb}{\nrmi}}
\newcommand{\nrmi}[1]{\Vert{#1}\Vert}
\newcommand{\nrmb}[1]{\left\Vert{#1}\right\Vert}
\newcommand{\abs}{\@ifstar{\absb}{\absi}}
\newcommand{\absi}[1]{\vert{#1}\vert}
\newcommand{\absb}[1]{\left\vert{#1}\right\vert}
\newcommand{\brk}{\@ifstar{\brkb}{\brki}}
\newcommand{\brki}[1]{\langle{#1}\rangle}
\newcommand{\brkb}[1]{\left\langle{#1}\right\rangle}
\newcommand{\set}{\@ifstar{\setb}{\seti}}
\newcommand{\seti}[1]{\{#1\}}
\newcommand{\setb}[1]{\left\{ #1\right\}}
\makeatother

\newcommand{\nnrm}[1]{{\vert\kern-0.25ex\vert\kern-0.25ex\vert #1 
    \vert\kern-0.25ex\vert\kern-0.25ex\vert}}

\newcommand{\VERT}[1]{{\left\vert\kern-0.25ex\left\vert\kern-0.25ex\left\vert #1 
    \right\vert\kern-0.25ex\right\vert\kern-0.25ex\right\vert}}

\newcommand{\rd}{\partial}
\newcommand{\nb}{\nabla}

\newcommand{\alp}{\alpha}
\newcommand{\bt}{\beta}

\newcommand{\dlt}{\delta}
\newcommand{\Dlt}{\Delta}
\newcommand{\eps}{\epsilon}

\newcommand{\Lmb}{\Lambda}

\newcommand{\omg}{\omega}

\newcommand{\zt}{\zeta}

\newcommand{\bbR}{\mathbb R}
\newcommand{\bbS}{\mathbb S}

\newcommand{\calA}{\mathcal A}

\newcommand{\calJ}{\mathcal J}

\newcommand{\calL}{\mathcal L}

\newcommand{\calN}{\mathcal N}

\newcommand{\calQ}{\mathcal Q}
\newcommand{\calR}{\mathcal R}

\newcommand{\calT}{\mathcal T}

\newcommand{\mrI}{\mathrm{I}}
\newcommand{\mrII}{\mathrm{II}}
\newcommand{\mrIII}{\mathrm{III}}

\newcommand{\Abs}[1]{\left\vert#1\right\vert}		
\newcommand{\normb}[1]{{\left\Vert #1 \right\Vert}_{L^2}}

\begin{document}

\title{Global self-similar solutions for the 3D Muskat equation}
\author{Jungkyoung Na{*}} \thanks{*Department of Mathematics, Brown University\newline  E-mail address: jungkyoung\_na@brown.edu}

\date{\today}



\maketitle


\begin{abstract}
    In this paper, we establish the existence of global self-similar solutions to the 3D Muskat equation when the two fluids have the same viscosity but different densities. These self-similar solutions are globally defined in both space and time, with exact cones as their initial data. Furthermore we estimate the difference between our self-similar solutions and solutions of the linearized equation around the flat interface in terms of critical spaces and some weighted $\dot{W}^{k,\infty}(\mathbb{R}^2)$ spaces for $k=1,2$. The main ingredients of the proof are new estimates in the sense of $\dot{H}^{s_1}(\mathbb{R}^2) \cap \dot{H}^{s_2}(\mathbb{R}^2)$ with $3/2<s_1<2<s_2<3$, which is continuously embedded in critical spaces for the 3D Muskat problem: $\dot{H}^2(\mathbb{R}^2)$ and $\dot{W}^{1,\infty}(\mathbb{R}^2)$.
\end{abstract}

\tableofcontents

\section{Introduction}
In the field of fluid dynamics, free boundary problems have posed significant challenges in modeling the evolution of boundaries between fluids. Among these, the Muskat problem stands out, describing the dynamics of the interface between two distinct immiscible and incompressible fluids (such as water and oil, or salt water and fresh water) as they propagate through porous media (such as sand or sandstone aquifer). Introduced initially by Morris Muskat in the 1930s \cite{Muskat-1,Muskat}, this problem has attracted considerable attention in mathematical analysis and various applications in physical and engineering studies (see e.g. \cite{Naras-98,KT-21}).

In this paper, we study the existence of global self-similar solutions to the three-dimensional Muskat equation when the two fluids have the same viscosity but different constant densities. This specific case has been one of the central focuses in much of the existing literature. To the best of the author's knowledge, our result is the first to address self-similar solutions in the 3D setting. Our self-similar solutions, which model 2D interfaces between two fluids in $\bbR^3$, have exact cones as their initial data and exist globally in both space and time. Moreover, we estimate the difference between our self-similar solutions and solutions of the linearized equation around the flat interface in terms of critical spaces such as $\dot{H}^2(\bbR^2)$ and $\dot{W}^{1,\infty}(\bbR^2)$, along with certain weighted $\dot{W}^{k,\infty}(\bbR^2)$ spaces for $k=1,2$. While the estimates in critical spaces provide a foundational information, the additional estimates in some weighted $\dot{W}^{k,\infty}$ spaces for $k=1,2$ offer more detailed insights into the behavior of our self-similar solutions.  The proof of our result is based on a new reformulation inspired by \cite{AN22}, along with new estimates with respect to  homogeneous Sobolev spaces. This approach differs from the proof of existence of self-similar solutions for the 2D Muskat equation in \cite{2Dself}. These aspects will be discussed in Section \ref{Main results} in detail.

To begin with, we present the 3D Muskat problem. We denote by $\Omega_i(t)$ $(i=1,2)$ two different time dependent fluid regions in $\bbR^3$, separated by a time dependent surface $\Sigma(t)$. Under the assumption that $\Sigma(t)$ is the graph of a function, we denote
\begin{equation}\label{Original Muskat-regions}
\begin{split}
    \Omega_1(t)&=\left\{X=(x,z)\in \bbR^2\times \bbR : z>f(t,x) \right\},\\
    \Omega_2(t)&=\left\{X=(x,z)\in \bbR^2\times \bbR : z<f(t,x) \right\},\\
    \Sigma(t)&=\left\{X=(x,z)\in \bbR^2\times \bbR : z=f(t,x) \right\}
    \end{split}
\end{equation}
for a function $f:\bbR_{\ge0}\times\bbR^2\rightarrow \bbR$.
\begin{figure}[h]
\centering
\includegraphics[scale=0.6]{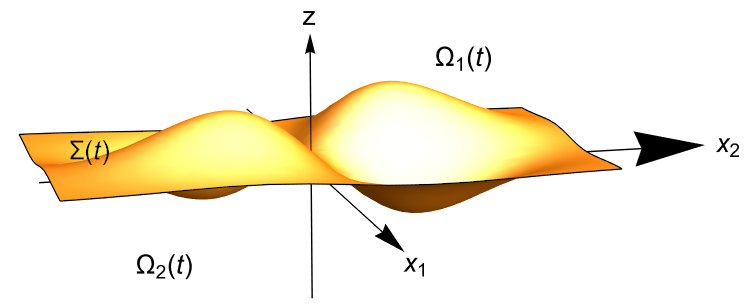}
\end{figure}

\noindent We assume that the medium has the constant permeability $\kappa$ and that the two fluids have the
same contant viscosity $\mu$ but different constant densities $\rho_i$ in $\Omega_i(t)$ $(i=1,2)$. 
Then the motions of the fluids can be written as 
\begin{equation}\label{Original Muskat}
\left\{
\begin{aligned}
 &\partial_t \rho_i(t,X) + u_i(t,X) \cdot \nb_{X} \rho_i(t,X) = 0,\\
 &\nb_X \cdot u_i(t,X)=0,\\
 &\frac{\mu}{\kappa}u_i(t,X)= (0,0,g\rho_i(t,X))-\nb_X P_i(t,X)
\end{aligned}
\right.
\end{equation}
in $\Omega_i(t)$ $(i=1,2)$ for $X=(x,z)\in \bbR^2\times\bbR$, where $u_i$ and $P_i$ are the fluid velocity and pressure in each $\Omega_i(t)$, respectively, and $g$ denotes the gravitational constant. The first equation in \eqref{Original Muskat} is the transport equation which represents the transport of the density by the flow. The second equation means the incompressibility of the fluid, and the third equation is the experimental Darcy's law \cite{Darcy}. We further assume that $\rho_1<\rho_2$, which corresponds to a stable regime. In other words, the stable regime means the heavier fluid lies below the lighter fluid. Then taking $\mu=\kappa=g=1$ and normalizing $\rho_2-\rho_1=2$ for simplicity, we can derive a contour equation of graphical interface $\Sigma(t)$ from \eqref{Original Muskat-regions} and \eqref{Original Muskat}:
\begin{equation}\label{eq:Muskat}
	\rd_{t} f(t,x) = \frac{1}{2\pi}\int_{\bbR^2} \frac{\alp \cdot \nabla_x \Delta_\alp f(t,x)}{\left(1+ (\Delta_\alp f(t,x))^2\right)^{\frac32}} \frac{d\alp}{|\alp|^2}, \qquad
	\Delta_\alp f(t,x)=\frac{f(t,x)-f(t,x-\alp)}{|\alp|}
\end{equation}
(see \cite{CG-07-contour} for a detailed derivation). Basic notable properties of this equation are twofold. Firstly, \eqref{eq:Muskat} is invariant by the transformation
\begin{equation*}
    f(t,x) \mapsto f_{\lambda}(t,x):=\lambda^{-1} f(\lambda t,\lambda x),\qquad \lambda>0.
\end{equation*}
Therefore, the Sobolev spaces $\dot{W}^{1,\infty}(\bbR^2)$, $\dot{H}^2(\bbR^2)$ and Wiener algebra $\calL^{1,1}(\bbR^2)$ are examples of critical spaces for the 3D Muskat problem. In general, $\dot{W}^{1,\infty}(\bbR^d)$, $\dot{H}^{1+\frac{d}{2}}(\bbR^d)$ and $\calL^{1,1}(\bbR^d)$ are critical spaces for the $(d+1)$-dimensional Muskat problem (see e.g. \cite{CNQX-22}).  In addition, \eqref{eq:Muskat} can be linearized around the flat solution as follows (see e.g. \cite{CG-07-contour}):
\begin{equation*}
    \rd_{t} f(t,x) + \Lmb f(t,x)=0,
\end{equation*}
where $\Lambda=(-\Delta)^\frac12$ denotes the Zygmund operator defined by
\begin{equation}\label{def: Zygmund}
    \Lmb f (x):=  \frac{1}{2\pi} \int_{\bbR^2} 
   \frac{f(x)-f(x-\alp)}{|\alp|}\frac{d\alp}{|\alp|^2}
\end{equation}
in $\bbR^2$. This linearization highlights the parabolic nature of the Muskat problem.

\subsection{Previous works}\label{Previous works}
As demonstrated by numerous previous studies, the Muskat problem is well-posed locally in time for sufficiently smooth initial interfaces, and globally in time if such initial data satisfies specific smallness conditions. However, for certain large initial data, the Muskat problem can experience finite-time singularity formation.

\medskip

\noindent \textbf{Well-posedness.} There have been numerous well-posedness results concerning the Cauchy problem for the 2D or 3D Muskat equation in subcritical regime. Local well-posedness results for initial data with sufficiently high regularity and global well-posedness results under further smallness assumption date back to Yi \cite{Yi-96,Yi-03}, Caflisch, Howison, and Siegel \cite{SCH-04} and Ambrose \cite{Ambrose-04,Ambrose-07}. In \cite{CG-07-contour}, D. C\'ordoba and Gancedo derived the contour equation \eqref{eq:Muskat} and showed local well-posedness in $H^s$ with $s\ge3$ in the 2D case and $s\ge4$ in the 3D case. Moreover, the authors of \cite{SCH-04} and \cite{CG-07-contour} also established ill-posedness results in the unstable regime, with the heavier fluid lying above the lighter one. Extended results to the viscosity jump case can be found in \cite{CCG-11,CCG-13}, building upon the work of \cite{CG-07-contour}.  Cheng, Granero-Belinch\'on and Shkoller \cite{Cheng-16} showed local well-posedness of the 2D problem in $H^2$. Moreover, they established global well-posedness and decay to equilibrium for small $H^2$ perturbations of the rest state. In \cite{CGSV-17}, Constantin, Gancedo, Shvydkoy and Vicol proved local well-posedness of the 2D problem for initial data in $W^{2,p}$ for $p\in[1,\infty)$, and furthermore, a global regularity result when the initial slope of the interface is sufficiently small. Later, Abels and Matioc \cite{AM-22} extended this 2D result to the subcritical space $W^{s,p}$ with $s\in(1+1/p,2)$ and $p\in(1,\infty)$. In \cite{Matioc-18,Matioc-19}, Matioc showed local well-posedness of the 2D problem for initial data in $H^2$ and $H^{s}$ with $s\in(3/2,2)$, respectively. Alazard and Lazar \cite{AL-20} paralinearized the 2D Muskat equation and applied it to show local well-posedness for initial data in $\dot{H}^1\cap\dot{H}^{s}$ with $s\in(3/2,2)$. In \cite{NP-20}, H. Q. Nguyen and Pausader employed a paradifferential approach to establish local well-posedness of arbitrary $d$-dimensional problems in any subcritical Sobolev spaces $H^s$ with $s>d/2+1$, accommodating various configurations such as viscosity jumps and presence of rigid boundaries. Very recently, Zlato{\v{s}} \cite{Zlatos-24-LWP} showed a local regularity result of the 2D problem on the half-plane and strips.

Now let us delve into well-posedness results in critical spaces. In \cite{CCGRS16}, Constantin, C\'ordoba, Gancedo, Piazza and Strain showed global existence of unique strong solutions of both 2D and 3D problems when the initial data $f_0\in L^2$ and its Wiener norm $\nrmb{f_0}_{\calL^{1,1}}:=\nrmb{|\xi|\hat{f_0}(\xi)}_{L^1_\xi}$ is less than $1/3$ for 2D and $1/5$ for 3D (see also \cite{CCGS-13}). Later, Gancedo, Garc\' ia-Ju\'arez, Patel and Strain \cite{GGPS-19} proved a similar result in the viscosity jump case. In \cite{DLL-17}, Deng, Lei and Lin established global existence of weak solutions for the 2D problem when initial data is monotone and belongs to $W^{1,\infty}$. Their solutions allow initial data with arbitrarily large slopes. Cameron \cite{C-19} established the existence of global classical solutions to the 2D problem under the condition that  initial data $f_0\in W^{1,\infty}$ and the product of the maximal and minimal slope of $f_0$ is less than $1$. Then he \cite{C-20} extended this 2D result to 3D setting for $f_0\in\dot{W}^{1,\infty}$ with $\nrmb{\nb_x f_0}<1/\sqrt{5}$ and $f_0$ satisfying sublinear growth. (See also \cite{C-20-ev} for an eventual regularization result in the 3D case.)
In \cite{CL-21}, C\'ordoba and Lazar proved a global existence result of the unique strong solution for initial data in 
$\dot{H}^{\frac{3}{2}}\cap \dot{H}^{\frac{5}{2}}$ with small
$\dot{H}^{\frac{3}{2}}$ semi-norm for the 2D problem. This allows the interface to have arbitrarily large finite slopes. Later, Gancedo and Lazar \cite{GL-22} extended this result to the 3D case, showing that the 3D problem is globally well-posed in the critical space $\dot{H}^2\cap\dot{W}^{1,\infty}$ with small $\dot{H}^{2}$, thereby permitting the interface to have arbitrary large finite slopes. In a sequence of three papers \cite{AN-21-series1,AN-21-series2,AN-23-series3}, Alazard and Q. H. Nguyen established local well-posedness of the 2D problem when initial data belonging to logarithmic subcritical space $\log(4+\Lmb)^{-1}H^{\frac32}:=\left\{f:\log(4+\Lmb)f\in H^{\frac32}\right\}$ in \cite{AN-21-series1},  $H^{\frac32}\cap \dot{W}^{1,\infty}$ in \cite{AN-21-series2}, and $H^{\frac32}$ in \cite{AN-23-series3}. Moreover they proved global well-posedness under the smallness of corresponding norms or semi-norms of initial data. In particular, the result in \cite{AN-21-series1} allows for initial data with infinite slopes. Later, they \cite{AN22} extended their 2D critical regularity results to 3D case in $\dot{H}^2\cap W^{1,\infty}$ by using the quailinearization of the 3D equation. In \cite{CNQX-22}, Chen, Q. H. Nguyen and Xu established local well-posedness of general $d$-dimensional problems in $L^2\cap\dot{W}^{1,\infty}$. Moreover, H. Q. Nguyen \cite{N-22} constructed unique global solutions for general $d$-dimensional problems in the Besov space $\dot{B}^{1}_{\infty,1}$, a space embedded in the critical space $\dot{W}^{1,\infty}$.

Finally, it is important to highlight studies on the existence and potential non-uniqueness of weak solutions in \cite{CFG-11,SL-12,CCF-21,FSL-18,NSL-21}. Moreover, a series of recent papers by Dong, Gancedo, and H. Q. Nguyen \cite{DGN-23, DGN-23-3d} are notable for their global regularity results concerning the 2D and 3D one-phase Muskat problem. For regularity results on the Hele-Shaw problem, which is mathematically analogous to the Muskat problem, we refer the reader to \cite{CJK-07, CJK-09}.

\medskip

\noindent \textbf{Finite time singularity formation.}
For the 2D Muskat problem, Castro et al. \cite{CCFGL-12,CCFG-13} proved the existence of a smooth graphical initial interface which becomes a non-graph (turning singularity) and later loses their $C^4$ regularity. G\'omez-Serrano and Granero-Belinch\'on \cite{GG-14} conducted a detailed study on the influences of depth and permeability of the medium in the formation of turning singularities. One can also refer to a series of papers \cite{Shi-23,Shi-23-1} for the analyticity of solutions which have turning singularities. Splash singularites, in other words, self-intersections at a single point are ruled out in our physical setting \cite{GS-14} while the existence of splash singularities are proved in the one-phase setting \cite{CCFG-16}. In \cite{CGZ-15,CGZ-17}, C\'ordoba, G\'omez-Serrano and Zlato{\v{s}} showed the existence of solutions undergoing the stability shifting. Very recently, Zlato{\v{s}} \cite{Zlatos-24-FB} proved that the finite time blow-up can arise in the half-plane setting even from arbitrarily small smooth initial data, unlike in the whole plane.

\medskip

\noindent \textbf{Self-similar solutions.}
To the best of the author's knowledge, the only known result concerning the existence of self-similar solutions to the Muskat problem in our physical setting is 2D self-similar solutions detailed in \cite{2Dself}. Under the assumption $0<s\ll 1$, the authors of the paper found self-similar solutions to the 2D Muskat problem that form exact corners with a slope $s$ at $t=0$ and become smooth in $x$ for $t>0$. Later Garc\'ia-Ju\'arez et al. \cite{GGHP-23} studied the behavior of an interface whose  initial data consists of a superposition of a finite number of small corners. For other physical settings, one can see \cite{EMM-12,LM-17,GGS-20} for the thin film Muskat problem and \cite{EM-11} for traveling wave solutions for the Muskat problem with surface tension.

\subsection{Main results}\label{Main results}
We assume that $f(t,\cdot)$ is radially symmetric, i.e., $f(t,x)=\Tilde{f}(t,|x|)$ for a function $\Tilde{f}(t,r)$ whose domain is $\bbR_{\ge0}\times\bbR_{\ge0} $. We set the self-similar ansatz $\tilde{f}(t,|x|)=tk\left(\frac{|x|}{t}\right)$ for a function $k: \bbR_{\ge0}\rightarrow \bbR$. Denoting $y=x/t$ and plugging our ansatz in \eqref{eq:Muskat}, we have
\begin{equation}\label{eq:Muskat-2}
     -y\cdot\nb_y(k(|y|)) +k(|y|) = \frac{1}{2\pi}\int_{\bbR^2} \frac{\alp \cdot \nabla_y \Delta_\alp k(|y|)}{\left(1+ \left(\Delta_\alp k(|y|)\right)^2\right)^{\frac32}}  \frac{d\alp}{|\alp|^2}.
\end{equation}
On the other hand, we note
\begin{equation}\label{eq:Muskat-2-1-lin}
    \Lambda (k(|y|)) = -\frac{1}{2\pi} \int_{\bbR^2} \alp \cdot \nabla_y \Delta_\alp k(|y|) \frac{d\alp}{|\alp|^2}.
\end{equation}
To obtain \eqref{eq:Muskat-2-1-lin}, we need two identities:
\begin{equation*}
    \int_{\bbR^2} \alp \cdot \nabla_y  k(|y|) \frac{d\alp}{|\alp|^3}= 0 \qquad \qquad \text{(by symmetry)}
\end{equation*}
and
\begin{equation*}
    \nb_y  (k(|y-\alp|))= -\nb_\alp (k(|y-\alp|))=  \nb_\alp \left( k(|y|)- k(|y-\alp|)\right),
\end{equation*}
which enable us to check that the right-hand side of \eqref{eq:Muskat-2-1-lin} is equal to
\begin{equation}\label{eq: before integration by parts-intro}
    \frac{1}{2\pi} \int_{\bbR^2} \alp \cdot \nb_\alp \left( k(|y|)- k(|y-\alp|)\right) \frac{d\alp}{|\alp|^3}.
\end{equation}
Using
\begin{equation*}
    \nb_\alp \cdot \left(\frac{\alp}{|\alp|^3}\right)=-\frac{1}{|\alp|^3},
\end{equation*}
we integrate \eqref{eq: before integration by parts-intro} by parts in $\alp$ and recall \eqref{def: Zygmund} to obtain \eqref{eq:Muskat-2-1-lin}.
Combining \eqref{eq:Muskat-2} with \eqref{eq:Muskat-2-1-lin}, we arrive at
\begin{equation}\label{eq: reformulation}
        (\Lambda-y\cdot\nb_y+1)  k (|y|) 
        = \calT[k(|y|)],
   \end{equation}
where $\calT$ is a nonlinear operator given by
\begin{equation}\label{def: nonlinear operator}
    \calT[f]:=\calT[f,f],\quad \calT[f_1,f_2]:=\frac{1}{2\pi}\int_{\bbR^2} \alp \cdot \nabla_y \Delta_\alp f_1\left(\left(1+ \left(\Delta_\alp f_2\right)^2\right)^{-\frac32} -1\right) \frac{d\alp}{|\alp|^2}.
\end{equation}
We construct a solution $k(|y|)$ for \eqref{eq: reformulation}, which yields a radially symmetric self-similar solution for \eqref{eq:Muskat}:
\begin{maintheorem}\label{thm A}
	Given $t_1\in(3/2,2)$, there exist $t^*=t^*(t_1)\in(t_1,2)$ and $s_*=s_*(t_1)>0$ such that for all $s\in(0,s_*)$, there exists a global self-similar solution of \eqref{eq:Muskat} given by 
 $f_{s}(t,x)=tk_{s}\left(\frac{|x|}{t}\right)$ for $(t,x)\in \bbR_{>0}\times \bbR^2$. Here, the function $k_s: \bbR_{\ge0}\rightarrow \bbR$ satisfies 
\begin{equation}\label{eq: thmA}
        \sum_{\substack{|\beta|=1\\1\le|\gamma|\le 2}} 
        \nrmb{\nb_y^{\beta}\left(k_{s}(|y|)-k^{Lin}_{s}(|y|)\right)}_{\dot{H}_y^{t^*-1} \cap \dot{H}_y^{t_1}(\bbR^2)}
        +\nrmb{\frac{\nb_y^{\gamma}\left(k_{s}(|y|)-k^{Lin}_{s}(|y|)\right)}{|y|^{t_1-|\gamma|}}}_{L_y^\infty(\bbR^2)}\lesssim s^3,
    \end{equation}
    where the function $k^{Lin}_s(|y|)$ is defined by
    \begin{equation}\label{eq: solution of lin eq for intro}
        k^{Lin}_{s} (|y|) := s\left(\sqrt{|y|^2+1}-\log\left(\sqrt{|y|^2+1}+1\right)\right).
    \end{equation}
\end{maintheorem} 
\begin{remark}
   Referring to the proofs in Sections \ref{sec: linear term} and \ref{sec: nonlinear term}, we can explicitly express $t^*$ as $\frac{t_1}{2}+1$. Although this form is not optimal, it serves as an illustrative example of $t^*$. The crucial point is that  we can find a $t^*$ which lies within the interval $(t_1,2)$.
\end{remark}

\medskip

\noindent \textbf{Interpretation of the result.}
We first observe that $k^{Lin}_{s} (|y|)$ defined by \eqref{eq: solution of lin eq for intro} is the radially symmetric  function satisfying
\begin{equation*}
    k^{Lin}_{s}(|y|)\in C^{\infty}(\bbR^2)\quad \text{and}\quad \nrmb{\nb_y k^{Lin}_{s}(|y|)}_{L^\infty(\bbR^2)}\le s. 
\end{equation*}
\begin{figure}[h]
\centering
\includegraphics[scale=0.6]{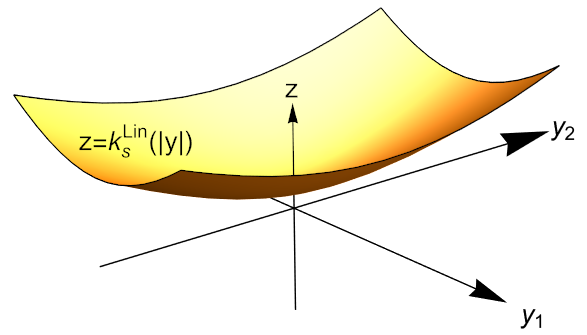}
\end{figure}

\noindent Furthermore, with the aid of Lemma \ref{lem: klin-1}, we can see that there exists a constant $C$ such that the function $k^{Lin}_{s} (|y|)+C$ is a solution to the linearized equation of \eqref{eq: reformulation}:
    \begin{equation}\label{lin eq for intro-interpretation}
        (\Lambda-y\cdot\nb_y+1)  k (|y|) =0.
    \end{equation}

Secondly, we note that the space $\dot{H}^{t^*-1}(\bbR^2)$, which appears in \eqref{eq: thmA}, is continuously embedded in $L^{\frac{2}{2-t^*}}(\bbR^2)$ (see \eqref{sobolev embedding-1}). Thus, we can deduce from the scaling $y=x/t$ and the $\dot{H}^{t^*-1}$ estimate in \eqref{eq: thmA} that 
\begin{equation*}\label{eq: thmA-1}
        \nrmb{\nb_x \left(f_s(t,x)-tk^{Lin}_s(|x|/t)\right)}_{L_x^{\frac{2}{2-t^*}}(\bbR^2) }\lesssim s^3t^{2-t^*}.
    \end{equation*}
By taking $t\rightarrow 0$, this implies that the initial data of the self-similar solution $f_s$ obtained in Theorem \ref{thm A} has the following form: 
\begin{equation*}
    f_s(0,x)=s|x|+C
\end{equation*}
for some constant $C$.
This initial data $f_s(0,x)$ forms an exact cone with linear growth and does not belong to $\dot{H}^2(\bbR^2)$, despite the restriction on the size of the  slope  $s$. These properties can be contrasted with the aformentioned global well-posedness results in 3D setting for initial data with medium-sized slope but sublinear growth \cite{C-20} and arbitrarily large slope but small enough in $\dot{H}^{2}(\bbR^2)$ \cite{GL-22}.

Next, we note that the space $\dot{H}^{t^*-1}(\mathbb{R}^2) \cap \dot{H}^{t_1}(\mathbb{R}^2)$ in \eqref{eq: thmA} is continuously embedded in $L^{\infty}(\mathbb{R}^2) \cap \dot{H}^1 (\mathbb{R}^2)$ (see \eqref{sobolev interpolation} and \eqref{sobolev embedding-3(infty)}). 
Hence, utilizing the scaling $y=x/t$, we can derive from $\dot{H}^{t^*-1} \cap \dot{H}^{t_1}$ estimate in \eqref{eq: thmA} that 
\begin{equation}\label{est: Linfty Hdot1 for intro}
    \nrmb{\nb_x \left(f_s(t,x)-tk^{Lin}_s(|x|/t)\right)}_{L_x^{\infty}\cap \dot{H}_x^1(\bbR^2) }\lesssim s^3.
\end{equation}
We recall that both $\dot{H}^2(\mathbb{R}^2)$ and $\dot{W}^{1,\infty}(\mathbb{R}^2)$ are critical spaces for the 3D Muskat problem. Thus, by considering $\nabla_x$ in the $L^{\infty}\cap \dot{H}^{1}$ norm from \eqref{est: Linfty Hdot1 for intro}, we observe that the inequality \eqref{est: Linfty Hdot1 for intro} shows that the difference between our self-similar solution and a solution of the linearized equation, with respect to critical Sobolev spaces, is $O(s^3)$. Our 3D result can be compared with the aformentioned 2D result in \cite{2Dself}. In that work, authors constructed a self-similar solution $f_s(t,x)=tk_s(x/t)$ $(x\in \bbR)$ of the 2D Muskat equation, satisfying
\begin{equation}\label{est: 2D self intro}
    \nrmb{\partial_y k_s(y)-\frac{2s}{\pi}\arctan(y)}_{H^1_y(\bbR)} \lesssim s^3, \qquad y:=x/t,
\end{equation}
under the assumption that $0<s\ll1$, so that its initial data is an exact corner of a small slope $s$. In \eqref{est: 2D self intro}, the integral of  $\frac{2s}{\pi}\arctan(y)$ is a solution of the linearized equation in the 2D case, corresponding to our $k^{Lin}_{s}$.

Finally, noticing the scaling $y=x/t$, we can obtain from the $L^\infty$ estimate in \eqref{eq: thmA} that
\begin{equation}\label{est: Linfty interms of x with time decay}
        \sum_{1\le|\gamma|\le 2} \Abs{\nb_x^{\gamma}\left(f_s(t,x)-tk^{Lin}_{s}(|x|/t)\right)}\lesssim \frac{s^3|x|^{t_1-|\gamma|}}{t^{t_1-1}} \quad \text{almost everywhere.}
    \end{equation}
Since $t_1-1>0$, this describes the long-time behavior of the first and second derivatives of our solution $f_s$. On the other hand, using $L^\infty$ estimate from \eqref{est: Linfty Hdot1 for intro}, we can derive that
\begin{equation}\label{est: Linfty interms of x without time decay}
        \sum_{|\beta|=1} \Abs{\nb_x^{\beta}\left(f_s(t,x)-tk^{Lin}_{s}(|x|/t)\right)}\lesssim s^3 \quad \quad \text{almost everywhere}.
    \end{equation}
Examining both \eqref{est: Linfty interms of x with time decay} and \eqref{est: Linfty interms of x without time decay}, we observe that for a fixed time $t>0$, the function $\nb_x\left(f_s(t,x)-tk^{Lin}_{s}(|x|/t)\right)$ exhibits behavior as $O(s^3|x|^{t_1-1})$ when $|x|\ll 1$ and as $O(s^3)$ when $|x|\gg 1$. Furthermore, considering the case $|\gamma|=2$ in  \eqref{est: Linfty interms of x with time decay} along with the $\dot{H}^1$ estimate from \eqref{est: Linfty Hdot1 for intro}, we find that for a fixed time $t>0$, $\nb_x^2\left(f_s(t,x)-tk^{Lin}_{s}(|x|/t)\right)$ decays with a bound of $s^3|x|^{t_1-2}$ and sufficiently fast as $|x|\rightarrow \infty$ to belong to $L^2(\bbR^2)$.

\medskip
	
\noindent \textbf{Challenges in 3D and ideas of the proof.}
Our goal is to find a solution $k_s(|y|)$ of \eqref{eq: reformulation} in the form
\begin{equation}\label{form of k -intro}
k_s(|y|)=k^{Lin}_s(|y|) + \calJ[\tilde{g}_s(|y|)]
\end{equation}
for sufficiently small $s$. Here, $\calJ$ is the operator from \eqref{eq: inverse laplacian-1}, which corresponds to the inverse of the Laplacian with certain regularity properties (see Proposition \ref{eq: poisson-radial}.) To achieve this, we substitute the ansatz \eqref{form of k -intro} into \eqref{eq: reformulation} and transform it into a fixed point equation \eqref{eq: g-fixed} in terms of $\tilde{g}_s(|y|)$. By using the smallness of the parameter $s$, we apply the Banach fixed point theorem to show the existence of $\tilde{g}_s(|y|)$. We then  take the operator $\calJ$ to $\tilde{g}_s(|y|)$ to obtain our desired solution $k_s(|y|)$. Consequently, our primary focus is on selecting an appropriate function space and estimating all the terms involved in the fixed point equation \eqref{eq: g-fixed} with respect to the chosen space to successfully apply the Banach fixed point theorem, as well as analyzing the operator $\calJ$ to determine regularity properties of our solution $k_s(|y|)$.

The main challenges in this process essentially arise from two factors: the limited regularity of $k^{Lin}_{s}(|y|)$ and the structure of nonlinearity in \eqref{eq:Muskat}. Regarding the regularity of $k^{Lin}_{s}(|y|)$, the fact that $\nrmb{k^{Lin}_{s}(|y|)}_{\dot{H}^t(\bbR^2)}<\infty$ only when $t>2$ (see Lemma \ref{lem: klin-1}) complicates the selection of a suitable function space. In other words, while we need to control all the terms in the fixed point equation \eqref{eq: g-fixed} involving $k^{Lin}_{s}(|y|)$, the $\dot{H}^2(\bbR^2)$ semi-norm of $k^{Lin}_{s}(|y|)$ diverges. Furthermore, we require a function space that is contained within critical spaces such as $\dot{H}^2(\bbR^2)$ and $\dot{W}^{1,\infty}(\bbR^2)$ to describe the behavior of our solutions in the context of critical spaces. This leads us to consider fractional Sobolev spaces of the type appearing in Theorem \ref{thm A}.
In contrast, in the 2D setting studied in \cite{2Dself}, the solution to the linearized equation corresponding to our $k^{Lin}_{s}$ is 
\begin{equation*}
    \frac{2s}{\pi}\left(y \arctan\left(y\right) - \frac12\log\left(y^{2} + 1\right)\right),\quad y\in \bbR.
\end{equation*}
This function, which is the integral of $\frac{2s}{\pi}\arctan(y)$, is such that its $\dot{H}^t(\bbR)$ semi-norm is bounded for any $t>\frac{3}{2}$. This regularity allowed the authors of \cite{2Dself} to derive the estimate \eqref{est: 2D self intro} without the need for fractional Sobolev spaces. Moreover, although we can exploit the smallness of $s$, the specific regularity of $k^{Lin}_{s}(|y|)$ prevents us from using the power series argument as in \cite{CCGRS16}. Specifically, we need to control $\calT\left[k^{Lin}_{s}(|\cdot|) \right]$ with respect to the fractional Sobolev space (Lemma \ref{lem: l^2 estimate for T(klin)}) to use the Banach fixed point theorem. However, for the power series argument in \cite{CCGRS16} to apply, $k^{Lin}_{s}(|y|)$ would need to belong to $\dot{H}^1(\bbR^2)$, which it does not. See the proof of Theorem 5.1 in \cite{CCGRS16} for further details.

To elucidate the structural challenge of  \eqref{eq:Muskat}, we recall the 2D Muskat equation (see e.g. \cite{CG-07-contour}):
\begin{equation*}\label{2D Muksat}
    \rd_{t} f(t,x) = \frac{1}{\pi}\int_{\bbR} \frac{ \rd_x \Delta_\alp f(t,x)}{1+ (\Delta_\alp f(t,x))^2} d\alp, \qquad
	\Delta_\alp f(t,x)=\frac{f(t,x)-f(t,x-\alp)}{\alp}, \qquad x\in\bbR.
\end{equation*}
After rewriting this in terms of the slope $\rd_x f$:
\begin{equation}\label{2D Muksat-slope}
    \rd_{t} \rd_x f = \frac{1}{\pi}\int_{\bbR} \frac{ \rd_{x} \Delta_\alp  \rd_x f}{1+ (\Delta_\alp f)^2} d\alp - \frac{2}{\pi} \int_{\bbR} \frac{  \left(\Delta_\alp  \rd_x f\right)^2 \Delta_\alp f }{(1+ (\Delta_\alp f)^2)^2} d\alp,
\end{equation}
the authors of \cite{2Dself} essentially extracted the quasilinear  structure of \eqref{2D Muksat-slope} as follows:
\begin{equation*}
    \frac{1}{\pi}\int_{\bbR} \frac{ \rd_x \Delta_\alp \rd_x f}{1+ (\Delta_\alp f)^2} d\alp= -(1+\left(\rd_x f\right)^2)^{-1}\Lmb \rd_x f + \frac{1}{\pi} \int_{\bbR} \rd_x \Delta_\alp \rd_x f \left((1+(\Delta_\alp f)^2)^{-1}-(1+\left(\rd_x f\right)^2)^{-1}\right)  d\alp,
\end{equation*}
which transforms \eqref{2D Muksat-slope} into
\begin{equation*}
    (1+\left(\rd_x f\right)^2)\rd_{t}\rd_x f+\Lmb \rd_x f= \text{Nonlinearity}.
\end{equation*}
They then used a key observation:
\begin{equation}\label{key ob for 2D-1}
    (1+\left(\rd_x f\right)^2)\rd_{t}\rd_x f=\rd_t(\rd_x f+\left(\rd_x f\right)^3/3),
\end{equation}
which played an important role in formulating a fixed point equation via a normal form. In other words, in the 2D case, the fact that the integral of the function $F_1(a):=1+a^2$ is $a+a^3/3$ led to the key observation \eqref{key ob for 2D-1} and consequently enabled the clear and efficient reformulation, despite the extracted quasilinear structure.
However, in the 3D case, the denominator of the integrand in \eqref{eq:Muskat} is
$F_2(\Dlt_\alp f)$ with $F_2(a):=(1+a^2)^{\frac32}$, whose integral has a much more complicated form:
\begin{equation*}
    \int_0^a F_2(b)\,db=\frac{3 \log\left(\left|\sqrt{a^{2} + 1} + a\right|\right) + a \sqrt{a^{2} + 1} \left(2a^{2} + 5\right)}{8}.
\end{equation*}
This presents significant difficulty in applying the 2D argument from \cite{2Dself}. 

To address these challenges, the strategy involves analyzing the operators $\widehat{\calL}$ in \eqref{linear operator L hat}, $\calT$ in  \eqref{def: nonlinear operator}, $\calQ$ in \eqref{def: calQ}, $\calR$ in \eqref{def: calR}, and $\calJ$ in \eqref{eq: inverse laplacian-1} in terms of the space $\dot{H}^{t^*}(\bbR^2) \cap \dot{H}^{t_1+1}(\bbR^2)$ introduced in Theorem \ref{thm A}.
First, the operator $\widehat{\calL}$ corresponds to the inverse operator of $(\Lambda-y\cdot\nb_y+1)$ that appears in \eqref{eq: reformulation}. Given $t_1\in[1,2)$, we utilize the fact that $\widehat{\calL}$ is defined via an integral to show that
\begin{equation}\label{operator L-estimate for intro}
    \nrmb{\widehat{\calL}}_{\dot{H}^{t_1}_{rad}(\bbR^2) \rightarrow \dot{H}^{t}_{rad}\cap\dot{H}^{t_1+1}_{rad}(\bbR^2)}\lesssim1
\end{equation}
for any $t\in(t_1,t_1+1]$ as detailed in Lemma \ref{lem: linear operator L estimate} and Remark \ref{remark for inverse operator of L}. Here, $X_{rad}$ denotes the space of radially symmetric functions in a space $X$. It is crucial that we can choose $t$ within the interval $(t_1,t_1+1]$, as this flexibility plays a significant role in obtaining estimates in terms of $\dot{H}^{t^*}(\bbR^2) \cap \dot{H}^{t_1+1}(\bbR^2)$ later on. 

All terms in the fixed point equation \eqref{eq: g-fixed} are of the form $\Dlt\circ\widehat{\calL}\circ \calT\circ\calJ$, $\Dlt\circ\widehat{\calL}\circ \calQ \circ\calJ$, and $\Dlt\circ\widehat{\calL}\circ \calR \circ\calJ$, where $\Dlt$ is the Laplacian. To apply the estimate for $\widehat{\calL}$ given in \eqref{operator L-estimate for intro}, we estimate these operators $\calT$, $\calQ$, and $\calR$ with respect to $\dot{H}^{t_1}$ for $t_1\in(3/2,2)$ as provided in Proposition \ref{proposition: key}, \ref{proposition: key for linear}, and \ref{proposition: key for nonlinear}. Due to technical considerations related to the regularity of $k^{Lin}_{s}$, we restrict our analysis to $t_1>3/2$, rather than $t_1\ge1$. The key aspects of these estimates are: 1) controlling operators using each component's semi-norms of spaces contained within $\dot{H}^{t^*} \cap \dot{H}^{t_1+1}$ and  2) assigning 
$\dot{W}^{1,\infty}$ or $\dot{H}^t$ ($t>2$) semi-norms to the components where $k^{Lin}_{s}$ will be inserted later. For example, Proposition \ref{proposition: key} and Remark \ref{rmk: proposition-1} illustrate how we control the operator $\calT$ in terms of $\dot{H}^{t^*}$ as shown in \eqref{key estimate-1}, while assigning $\dot{H}^{t^{**}}\cap \dot{H}^{t_1+1}$ $(t^{**}>2)$ to $f$ into which $k^{Lin}_{s}$ will be inserted later. The key strategy for accomplishing these estimates involves decomposing $\Lmb^{t_1}\calT$, $\Lmb^{t_1}\calQ$, and $\Lmb^{t_1}\calR$ into $T^{t_1,j}$, $Q^{t_1,j}$, and $R^{t_1,j}$ $(1\le j \le 7)$, respectively, through the symmetrization of these operators, as detailed in Lemma \ref{lem: reformulation}, \ref{lem: reformulation for linear A}, and \ref{lem: reformulation for nonlinear N}. In particular, the components $T^{t_1,1}$, $Q^{t_1,1}$, and $R^{t_1,1}$ correspond to the quasilinear structures of  $\Lmb^{t_1}\calT$, $\Lmb^{t_1}\calQ$, and $\Lmb^{t_1}\calR$, respectively, reflecting the idea inspired by \cite{AN22}. Moreover, the components $T^{t_1,2}$, $Q^{t_1,2}$, and $R^{t_1,2}$ contain factors that compare the difference between $\Delta_\alp f$ and $\frac{\alp}{|\alp|}\cdot \nb f$, so that the Morrey type estimate \eqref{App of Morrey} enables us to control the integrand near the origin. To control the components $T^{t_1,7}$, $Q^{t_1,7}$, and $R^{t_1,7}$, we use an explicit commutator identity for the fractional Laplacian given in \eqref{commutator identity for fractional laplacian}.

The operator $\calJ$ represents the inverse of the Laplacian acting on radially symmetric functions that belong to certain homogeneous Sobolev spaces (see Proposition \ref{prop: poisson radial}). Consequently, the operators $\Dlt$ and $\calJ$, which appear in compositions $\Dlt\circ\widehat{\calL}\circ \calT\circ\calJ$, $\Dlt\circ\widehat{\calL}\circ \calQ \circ\calJ$, and $\Dlt\circ\widehat{\calL}\circ \calR \circ\calJ$, play a crucial role in changing our space $\dot{H}^{t^*}(\bbR^2) \cap \dot{H}^{t_1+1}(\bbR^2)$ into $\dot{H}^{t^*-2}(\bbR^2) \cap \dot{H}^{t_1-1}(\bbR^2)$ which is a Banach space since $t^*-2<0<t_1-1<1$. After obtaining $\tilde{g}_s(|y|)\in\dot{H}^{t^*-2}(\bbR^2) \cap \dot{H}^{t_1-1}(\bbR^2)$ via the Banach fixed point theorem, we apply the operator $\calJ$ to $\tilde{g}_s(|y|)$ and utilize the regularity properties given in \eqref{eq: L infty est for inverse laplacain-1} (Proposition \ref{prop: poisson radial}) to yield the desired solution $k_s(|y|)$ satisfying Theorem \ref{thm A}. In this process, the Sobolev inequality \eqref{eq: radial sobolev for frac} for radial functions, which is an extension of the well-known Strauss' inequality \cite{Strauss-77}, plays an important role.

\subsection{Outline of the paper}
The rest of the paper is structured as follows. In Section \ref{sec: preliminaries}, we provide basic definitions and elementary lemmas that will be used frequently throughout the paper. Section \ref{sec: proof of the main result} is dedicated to the proof of Theorem \ref{thm A}. As the first step, we explore the Poisson equation of a radial function and obtain some properties of $\calJ$ in Subsection \ref{subsec: inverse laplacian}. Then we analyze the linearized equation \eqref{lin eq for intro-interpretation} in Subsection \ref{subsec: analysis of linearized eq}. Specifically, in Subsection \ref{subsubsec: homo}, we show $k^{Lin}_{s}(|y|)$ defined by \eqref{eq: solution of lin eq for intro} is a solution of \eqref{lin eq for intro-interpretation} up to constant, and then we derive two quantitative estimates of $k^{Lin}_{s}(|y|)$. In Subsection \ref{subsubsec: inhomo}, we obtain the estimate \eqref{operator L-estimate for intro} of $\widehat{\calL}$. Then in Subsection \ref{subsec: fixed point formula}, we introduce a fixed point equation and prove Theorem \ref{thm A} via Proposition \ref{prop 1}. Subsections \ref{sec: forcing term} - \ref{sec: nonlinear term} are devoted to the analysis of all terms in the fixed point equation through the estimations of the operators $\calT$, $\calQ$, and $\calR$.

\subsection*{Acknowledgments}{}
The author thanks Javier G\'{o}mez-Serrano and Benoit Pausader for educational discussions and comments. The author's research was partially supported by NSF under Grants DMS-2245017 and DMS-2247537.

\section{Preliminaries}\label{sec: preliminaries}
Throughout this note, we shall employ the letter $C=C(a,b,\cdots)$ to denote any constant depending on $a,b,\cdots$, which may change from line to line in a given computation. We frequently use $A\lesssim B$ and $A\approx B$, which means $A\le CB$ and $C^{-1}B\le A \le CB$, respectively, for some constant $C$. In addition, we shall use the following notations to represent the finite differences of functions: Given $x,\,\alp \in \bbR^2$ and a function $f:\bbR^2\rightarrow \bbR$, we define
\begin{equation*}
         \Delta_\alp f(x):=\frac{f(x)-f(x-\alp)}{|\alp|},\qquad
          \dlt_\alp f(x):=f(x)-f(x-\alp).
\end{equation*}

Based on the above notations, we first consider an elementary estimate followed by the mean value theorem:
\begin{lemma}
For any $p\in[1,\infty]$ and a function $f:\bbR^2\rightarrow \bbR$, there holds
    \begin{equation}\label{fundamental thm of cal}
        \nrmb{\dlt_\alp f}_{L^p(\bbR^2
        )}\le |\alp| \nrmb{\nb f}_{L^p(\bbR^2)}
    \end{equation}
\end{lemma}
\begin{proof}
    By the mean value theorem, we have
    \begin{equation*}
        \Abs{\dlt_\alp f(x)}=\Abs{\int_{0}^1 \nb f(\tau x +(1-\tau)(x-\alp)) \cdot \alp d\tau}\le |\alp| \int_{0}^1 \Abs{\nb f(x -(1-\tau)\alp)} d\tau
    \end{equation*}
    Since the case when $p=\infty$ is trivial, we consider $p\in[1,\infty)$. Applying Minkowski's inequality, we have
    \begin{equation*}
         \nrmb{\dlt_\alp f}_{L^p}\le |\alp| \int_{0}^1 \left(\int_{\bbR^2}\left|\nb f(x -(1-\tau)\alp)\right|^p dx\right)^{\frac1p} d\tau 
         \le |\alp| \int_{0}^1 \left(\int_{\bbR^2}\left|\nb f(y)\right|^p dy\right)^{\frac1p} d\tau \le |\alp| \nrmb{\nb f}_{L^p}.
    \end{equation*}
\end{proof}
Next, we introduce estimates followed by Morrey's estimate:
\begin{lemma}
For any $p\in(2,\infty)$ and a function $f:\bbR^2\rightarrow \bbR$, there holds
    \begin{equation}\label{Morrey}
        \Abs{\dlt_\alp f}\lesssim |\alp|^{1-\frac2p} \nrmb{\nb f}_{L^p(\bbR^2)},
    \end{equation}
    \begin{equation}\label{App of Morrey}
        \Abs{\dlt_\alp f-\alp\cdot\nb f}\lesssim |\alp|^{2-\frac2p} \nrmb{\nb^2 f}_{L^p(\bbR^2)}.
    \end{equation}
\end{lemma}
\begin{proof}
       The estimate \eqref{Morrey} follows from the 2D version of Morrey's estimate (see e.g. \cite{Eva}): Given $x\in \bbR^2$,
\begin{equation}\label{Morrey for bounded domain}
    |u(x)-u(w)| \lesssim r^{1-\frac2p}\left(\int_{B_{2r}(x)}|\nb u(z)|^p dz \right)^{1/p}, \quad w\in B_{r}(x),
\end{equation}
valid for any $u\in W^{1,p}(B_{2r}(x))$ with $p\in(2,\infty)$. 

To obtain \eqref{App of Morrey}, we fix any $x\in\bbR^2$ and set
\begin{equation*}
    u(w):=f(x)-f(w)-(x-w)\cdot \nb f(x)
\end{equation*}
in \eqref{Morrey for bounded domain}, where $r=|x-w|$. Then
we find
\begin{equation*}\label{morrey-2}
\begin{split}
    \left|f(x)-f(w)-(x-w)\cdot \nb f(x)\right|&
    =\left|u(x)-u(w)\right|\\
    &\lesssim r^{1-\frac2p} \left( \int_{B_{2r}(x)} \left|\nb u(z)\right|^p dz\right)^{\frac1p} \\
        &= r^{1-\frac2p} \left( \int_{B_{2r}(x)} \left|\nb f(x)-\nb f(z)\right|^p dz\right)^{\frac1p} \\
        &\lesssim r^{2-\frac2p}\left(\int_{B_{2r}(x)}\Abs{\int_0^1 \nb^2 f (\tau x +(1-\tau) z) d\tau}^p dz\right)^{\frac1p} \\
        &\lesssim r^{2-\frac2p} \int_0^1 \left(\int_{B_{2r}(x)}\Abs{ \nb^2 f (\tau x +(1-\tau) z)  }^p dz \right)^{\frac1p} d\tau \\
        &\lesssim r^{2-\frac2p} \nrmb{\nb^2 f}_{L^p} \int_0^1 \frac{1}{(1-\tau)^\frac2p} d\tau \lesssim r^{2-\frac2p} \nrmb{\nb^2 f}_{L^p},
\end{split}
\end{equation*}
where we used the mean value theorem and Minkowski's inequality in the second and the third  inequalities, respectively. Now setting $w=x-\alp$, we are done.
\end{proof}

Next, we recall the definition of the fractional Laplacian $\Lambda^s:=(-\Delta)^\frac{s}{2}$ with $s\in(0,2)$ in $\bbR^2$ (see e.g. \cite{Hitchhiker}):
\begin{equation}\label{def: fractional laplacian-1}
    \Lambda^s f(x)= C(s) P.V. \int_{\bbR^2} \frac{\dlt_\alp f(x)}{|\alp|^{2+s}} d\alp,
\end{equation}
where $C(s)$ is a constant given by
\begin{equation*}
    C(s)=\left(\int_{\bbR^2} \frac{1-\cos{\zt_1}}{|\zt|^{2+s}} dh\right)^{-1}.
\end{equation*}
Using \eqref{def: fractional laplacian-1} together with the identity
\begin{equation*}
    \dlt_\alp(fg)-f(\dlt_\alp g) - g(\dlt_\alp f)=-(\dlt_\alp f) (\dlt_\alp g),
\end{equation*}
one can obtain
\begin{equation}\label{commutator identity for fractional laplacian}
    \left(\Lambda^s(fg)-f\Lambda^s g -g \Lambda^s f\right)(x) 
    =-C(s) P.V.\int_{\bbR^2} \frac{\dlt_\alp f(x) \dlt_\alp g(x)}{|\alp|^{2+s}} d\alp. 
\end{equation}
Moreover, we recall the interpolation of homogeneous Sobolev spaces (see e.g. \cite{BCD}):
\begin{equation}\label{sobolev interpolation}
    \dot{H}^{s_0}(\bbR^2) \cap \dot{H}^{s_1}(\bbR^2) \hookrightarrow \dot{H}^{s}(\bbR^2) \quad \text{for}\;\,s\in[s_0,s_1].
\end{equation}
Furthermore, we recall the following Sobolev embedding (see e.g. \cite{BCD}):
\begin{equation}\label{sobolev embedding-1}
\dot{H}^{1-\frac2p}(\bbR^2)\hookrightarrow L^p(\bbR^2) \quad \text{for}\;\,p\in(2,\infty).
\end{equation}
We can use H\"older's inequality and \eqref{sobolev embedding-1} successively to obtain
\begin{equation}\label{sobolev embedding-2}
    \nrmb{\prod_{i=1}^{n}f_i}_{L^2(\bbR^2)}\lesssim 
    \prod_{i=1}^{n}\nrmb{f_i}_{\dot{H}^{s_i}(\bbR^2)}  \quad \text{for}\;\,n\ge2,\;\,s_i\in(0,1)\;\,\text{satisfying}\;\, \sum_{i=1}^n s_i=n-1.
\end{equation}
In addition, we note another Sobolev embedding:
\begin{equation}\label{sobolev embedding-3(infty)}
    \dot{H}^{s_1}(\bbR^2) \cap \dot{H}^{s_2}(\bbR^2) \hookrightarrow L^\infty(\bbR^2) \quad \text{for} \;\, (s_1,s_2) \in [0,1)\times(1,\infty).
\end{equation}
Indeed, noticing that the double Fourier transform of $f(x)$ is just $f(-x)$, we use Hausdorff-Young's inequality to obtain
\begin{equation*}
    \nrmb{f}_{L^\infty}\le \nrmb{\widehat{f}}_{L^1}
    \le \left(\int_{|\xi|\le1} |\xi|^{-2s_1} d\xi\right)^{\frac12} \nrmb{f}_{\dot{H}^{s_1}} + \left(\int_{|\xi|>1} |\xi|^{-2s_2} d\xi\right)^{\frac12}\nrmb{f}_{\dot{H}^{s_2}}\lesssim \nrmb{f}_{\dot{H}^{s_1}}+\nrmb{f}_{\dot{H}^{s_2}}.
\end{equation*}
Next, we introduce a lemma which we will use frequently:
  \begin{lemma}[Lemma 2.9 in \cite{AN22}]
$\quad$
  
  $\bullet\;$ For all $a\in[0,\infty)$ and $b\in(0,1)$, there holds
  \begin{equation}\label{eq: prelemma-1}
      \int_{\bbR^2} \nrmb{\dlt_\alp f}^2_{\dot{H}^a(\bbR^2)}\frac{d\alp}{|\alp|^{2+2b}} \approx \nrmb{f}^2_{\dot{H}^{a+b}(\bbR^2)}.
  \end{equation}

  $\bullet\;$ For all $a\in[0,\infty)$, $\gamma \in [1,\infty)$, $\gamma<b<2\gamma$, there holds
  \begin{equation}\label{eq: prelemma-2}
       \int_{\bbR^2} \nrmb{\dlt_\alp f - \alp \cdot \nb_x f}^{2\gamma}_{\dot{H}^a(\bbR^2)}\frac{d\alp}{|\alp|^{2+2b}}\lesssim \nrmb{f}^{2\gamma}_{\dot{H}^{a+\frac{b}{\gamma}}(\bbR^2)}.
  \end{equation}
  \end{lemma}
Finally, we introduce a Sobolev inequality for radially symmetric functions:
\begin{lemma}[Proposition 1 in \cite{CO-09}]\label{lem: radial sobolev}
    Let $s\in(1/2,1)$ and let $f\in \dot{H}^{s}(\bbR^2)$ be a radially symmetric function. Then $f$ is almost everywhere equal to a function $U(x)$, continuous for $x\neq0$ such that
    \begin{equation}\label{eq: radial sobolev for frac}
        \sup_{x\in\bbR^2\backslash\{0\}} |x|^{1-s}|U(x)|\lesssim \nrmb{f}_{\dot{H}^s(\bbR^2)}.
    \end{equation}
\end{lemma}

\section{Proof of the main result}\label{sec: proof of the main result}
\subsection{The inverse of the Laplacian acting on radial functions}\label{subsec: inverse laplacian}
In this subsection, we consider the Poisson equation
\begin{equation}\label{eq: poisson-radial}
    \Dlt u(|x|)=\phi(|x|),\qquad x\in\bbR^2,
\end{equation}
where $\phi(|x|)$ is a radial function which belongs to the homogeneous Sobolev space $\dot{H}^s(\bbR^2)$ with $s\in(1/2,1)$. Our goal is to prove the following proposition.
\begin{proposition}\label{prop: poisson radial}
    Given a function $\phi(|x|)\in \dot{H}^s(\bbR^2)$ with $s\in(1/2,1)$, let $\calJ$ be the operator defined by
    \begin{equation}\label{eq: inverse laplacian-1}
        \calJ[\phi](|x|):= \int_{0}^{|x|}\frac1r \int_0^r \tau \phi(\tau)d\tau dr.
    \end{equation}
    Then $\calJ[\phi](|x|)$ is a solution of \eqref{eq: poisson-radial} and satisfies the following properties:
    \begin{equation}\label{eq: L infty est for inverse laplacain-1}
        \sum_{0\le|\beta|\le 2} \nrmb{\frac{\nb^{\beta}_x\calJ[\phi](|x|)}{|x|^{s+1-|\beta|}}}_{L^\infty(\bbR^2)}\lesssim \nrmb{\phi(|\cdot|)}_{\dot{H}^s(\bbR^2)},
    \end{equation}
    \begin{equation}\label{eq: fourier transform of JC1C2}
        -|\xi|^2\widehat{\calJ[\phi](|\cdot|)}(\xi)=\widehat{\phi(|\cdot|)}(\xi).
    \end{equation}
    
    Moreover, if $u(|x|)$ is a radially symmetric solution of \eqref{eq: poisson-radial} and satisfies $\nb_xu(|x|)\in L_{loc}^{p}(\bbR^2)$ for some $p\in(2,\infty]$, then there exists a constant $C$ such that
    \begin{equation}\label{uniqueness of u}
        u(|x|)=\calJ[\phi](|x|)+C
    \end{equation}
    (up to redefinition of $u$ on a set of measure zero).
\end{proposition}
\begin{remark}\label{rmk: radial poisson}
For future use, we note following observations.
    Under the same assumption with this proposition, we see that
    \begin{equation*}
        (\Lmb-x\cdot\nb_x+1)\calJ[\phi](|x|)\in \dot{W}_{loc}^{1,\frac{2}{1-s}}(\bbR^2).
    \end{equation*}
    Indeed, we can establish that $\nb_x(-x\cdot\nb_x+1)\calJ[\phi](|x|)\in L_{loc}^{\frac{2}{1-s}}(\bbR^2)$ using the estimate in \eqref{eq: L infty est for inverse laplacain-1}, while we can deduce that $\nb_x\Lmb \calJ[\phi](|x|)\in L^{\frac{2}{1-s}}(\bbR^2)$ by the Sobolev embedding \eqref{sobolev embedding-1}:
    \begin{equation}\label{ineq: for remark 3.2}
        \nrmb{\nb\Lmb \calJ[\phi](|\cdot|)}_{L^{\frac{2}{1-s}}(\bbR^2)}\lesssim \nrmb{\nb\Lmb \calJ[\phi](|\cdot|)}_{\dot{H}^s(\bbR^2)}\approx \nrmb{\phi(|\cdot|)}_{\dot{H}^s(\bbR^2)}.
    \end{equation}
\end{remark}
\begin{proof}
To begin with, we show that $\calJ[\phi](|x|)$ is well-defined and satisfies \eqref{eq: L infty est for inverse laplacain-1}. 
    Since $\phi(|x|)\in \dot{H}^s(\bbR^2)$ with $s\in(1/2,1)$,  Lemma \ref{lem: radial sobolev} provides us with a function $\tilde{\phi}$ almost everywhere equal to $\phi$ such that
    \begin{equation}\label{est for tilde phi}
        |\tau\tilde{\phi}(\tau)|\lesssim \tau^s\nrmb{\phi(|\cdot|)}_{\dot{H}^s(\bbR^2)} \quad \text{for any}\;  \tau>0.
    \end{equation}
    Hereafter, we identify $\phi$ with $\tilde{\phi}$.
    Using \eqref{est for tilde phi}, we have
    \begin{equation}\label{est: for J-L infty}
        \Abs{\calJ[\phi](|x|)}\lesssim \Abs{\int_{0}^{|x|}\frac1r \int_0^r \tau^s d\tau dr}\nrmb{\phi(|\cdot|)}_{\dot{H}^s(\bbR^2)} \approx  \Abs{\int_{0}^{|x|} r^s d\tau dr}\nrmb{\phi(|\cdot|)}_{\dot{H}^s(\bbR^2)} \approx |x|^{s+1}\nrmb{\phi(|\cdot|)}_{\dot{H}^s(\bbR^2)},
    \end{equation}
    which implies $\calJ[\phi](|x|)$ is well-defined.
    Moreover, we calculate
    \begin{equation}\label{calculation of derivatives of calj}
    \begin{split}
        \nb\calJ[\phi](|x|)&=\frac{x}{|x|^2} \int_0^{|x|} \tau \phi(\tau)d\tau, \\
        \nb_{x_i,x_j}^2\calJ[\phi](|x|)&=\left(\frac{\dlt_{ij}}{|x|^2}-\frac{2x_ix_j}{|x|^4}\right) \int_0^{|x|} \tau \phi(\tau)d\tau +\frac{x_ix_j}{|x|^2}\phi(|x|).
        \end{split}
    \end{equation}
    Using \eqref{est for tilde phi} again, we estimate
    \begin{equation}\label{est: for J-L infty-higher}
    \begin{split}
        \Abs{\nb\calJ[\phi](|x|)}&\lesssim \frac{1}{|x|} \Abs{\int_0^{|x|} \tau^s d\tau}\nrmb{\phi(|\cdot|)}_{\dot{H}^s(\bbR^2)} \approx |x|^{s}\nrmb{\phi(|\cdot|)}_{\dot{H}^s(\bbR^2)},\\
        \Abs{\nb_{x_i,x_j}^2\calJ[\phi](|x|)}&\lesssim \left(\frac{1}{|x|^2} \Abs{\int_0^{|x|} \tau^s d\tau} + |x|^{s-1}\right)\nrmb{\phi(|\cdot|)}_{\dot{H}^s(\bbR^2)}\approx |x|^{s-1}\nrmb{\phi(|\cdot|)}_{\dot{H}^s(\bbR^2)}.
        \end{split}
    \end{equation}
    Combining all, we obtain \eqref{eq: L infty est for inverse laplacain-1}. Moreover, the second derivatives of $\calJ[\phi](|x|)$ in \eqref{calculation of derivatives of calj} implies that $\calJ[\phi](|x|)$ is a solution of \eqref{eq: poisson-radial}, so that its Fourier transform satisfies \eqref{eq: fourier transform of JC1C2}.

    Next, suppose that $\nb u(|x|)\in L_{loc}^{p}(\bbR^2)$ is a radially symmetric solution of \eqref{eq: poisson-radial} for some $p\in(2,\infty]$. Then since
\begin{equation*}
    \Dlt(u-\calJ[\phi])(|x|)=0,
\end{equation*}
we utilize the formula
\begin{equation*}
        \Delta (u-\calJ[\phi])(|x|)=\left.\frac{1}{r}\frac{d}{dr}\left(r\frac{d}{dr}(u-\calJ[\phi])(r)\right)\right|_{r=|x|},\qquad |x|>0
\end{equation*}
to obtain
\begin{equation*}
    \nb(u-\calJ[\phi])(|x|)=\frac{Cx}{|x|^2},\qquad |x|>0
\end{equation*}
for some constant $C$. Note that  the left-hand side of this equality belongs to $L_{loc}^{p}(\bbR^2)$  by the assumption and \eqref{est: for J-L infty-higher}. However, the right-hand side does not belong to $L_{loc}^{p}(\bbR^2)$ unless $C\neq 0$, and consequently
\begin{equation*}
    \nb(u-\calJ[\phi])(|x|)=0,\quad |x|>0,
\end{equation*}
which leads to \eqref{uniqueness of u}.
\end{proof}

\subsection{Analysis of the linearized equation}\label{subsec: analysis of linearized eq}
\subsubsection{Homogeneous linearized equation}\label{subsubsec: homo}
Here, we consider the homogeneous linearized equation of \eqref{eq: reformulation}:
\begin{equation}\label{linearized eq}
    (\Lambda-y\cdot\nb_y+1) k (|y|) =0.
\end{equation}

\begin{lemma}\label{lem: klin-1}
Given any $s\in \bbR$, the function $k^{Lin}_{s}(|y|)$ defined by
\begin{equation}\label{klin}
        k^{Lin}_{s}(|y|)= s\left(\sqrt{|y|^2+1}-\log\left(\sqrt{|y|^2+1}+1\right)\right),
    \end{equation}
    there exists a constant $C$ such that $k^{Lin}_{s}(|y|)+C$ is a solution of \eqref{linearized eq}.
Furthermore, it satisfies
    \begin{equation}\label{klin laplacian}
        \Delta k^{Lin}_{s}(|y|)= \frac{s}{\sqrt{|y|^2+1}}
    \end{equation}
    whose Fourier transform is 
    \begin{equation}\label{klin fourier}
        \widehat{\Delta k^{Lin}_{s} (|\cdot|)}(\xi)=\frac{se^{-|\xi|}}{|\xi|}.
    \end{equation}
\end{lemma}
\begin{remark}
    For future use, we also note 
    \begin{equation}\label{klin nabla}
        \nb_y k^{Lin}_{s}(|y|) = \frac{sy}{\sqrt{|y|^2+1}+1},
    \end{equation}
    \begin{equation}\label{klin nabla^2}
        \nb^2_{y_iy_j} k^{Lin}_{s}(|y|) = s\left(\frac{\dlt_{ij}}{\sqrt{|y|^2+1}+1}-\frac{y_iy_j}{\sqrt{|y|^2+1}\left(\sqrt{|y|^2+1}+1\right)^2}\right).
    \end{equation}
\end{remark}
\begin{proof}
To begin with, we note that the Fourier transform of $\frac{s}{\sqrt{|y|^2+1}}$ is
\begin{equation}\label{fourier transform of 1/sqrt y^2+1}
    \left(\frac{s}{\sqrt{|\cdot|^2+1}}\right)^{\wedge}(\xi)=\frac{se^{-|\xi|}}{|\xi|},
\end{equation}
which can be calculated via the Hankel transform. Moreover \eqref{fourier transform of 1/sqrt y^2+1} implies that $\frac{s}{\sqrt{|y|^2+1}}\in\dot{H}^t(\bbR^2)$ for any $t>0$.

Next, 
to prove that $k^{Lin}_s(|y|)+C$ is a solution of \eqref{linearized eq} for a constant $C$, it suffices to show the following two things: 
\begin{equation}\label{conti est for JCC1 for klin-2}
    (\Lambda-y\cdot\nb_y+1) \calJ\left[\frac{s}{\sqrt{|y|^2+1}}\right]\in \dot{W}^{1,p}_{loc}(\bbR^2) \quad\text{for some}\;\,p\in(2,\infty],
\end{equation}
and
\begin{equation}\label{another linearized equation acted laplacain}
    \Dlt \left((\Lambda-y\cdot\nb_y+1) \calJ\left[\frac{s}{\sqrt{|y|^2+1}}\right] \right)=0,
\end{equation}
where $\calJ$ is the operator defined in \eqref{eq: inverse laplacian-1}.
Indeed, \eqref{conti est for JCC1 for klin-2} and \eqref{another linearized equation acted laplacain} imply that 
$(\Lambda-y\cdot\nb_y+1) \calJ\left[\frac{s}{\sqrt{|y|^2+1}}\right]\in \dot{W}^{1,p}_{loc}(\bbR^2)$ is a radially symmetric solution of the Laplace equation, so that \eqref{uniqueness of u} in Proposition \ref{prop: poisson radial} implies that
\begin{equation*}
    (\Lambda-y\cdot\nb_y+1) \calJ\left[\frac{s}{\sqrt{|y|^2+1}}\right]=C_1 \quad\text{for a constant}\;\,C_1.
\end{equation*}
But noticing
\begin{equation*}
    (\Lambda-y\cdot\nb_y+1) \left(\calJ\left[\frac{s}{\sqrt{|y|^2+1}}\right]-C_1\right)=-C_1+(\Lambda-y\cdot\nb_y+1) \calJ\left[\frac{s}{\sqrt{|y|^2+1}}\right]=0
\end{equation*}
and
\begin{equation*}
    \calJ\left[\frac{s}{\sqrt{|y|^2+1}}\right]= k^{Lin}_{s}(|y|)-1+\log 2,
\end{equation*}
we can conclude that $k^{Lin}_{s}(|y|)-1+\log 2-C_1$
is a solution of \eqref{linearized eq}.
The reason for this approach is to address the singularity issue at the origin in the frequency space.
Note that \eqref{conti est for JCC1 for klin-2} follows from Remark \ref{rmk: radial poisson}. For \eqref{another linearized equation acted laplacain}, denoting the Fourier transform of $\calJ\left[\frac{s}{\sqrt{|y|^2+1}}\right]$ by $\varphi(|\xi|)$ for simplicity, we have
\begin{equation*}
    \left(\Dlt\left((\Lambda-y\cdot\nb_y+1) \calJ\left[\frac{s}{\sqrt{|y|^2+1}}\right]\right)\right)^{\wedge}(\xi)=-|\xi|^2(|\xi|+1)\varphi(|\xi|) - |\xi|^2\sum_{j=1}^2 \partial_{\xi_j}\left( \xi_j\varphi(|\xi|)\right).
\end{equation*}
Since $\varphi(|\xi|)$ is radially symmetric, we compute
\begin{equation*}
    \sum_{j=1}^2 \partial_{\xi_j}\left( \xi_j\varphi(|\xi|)\right)=2\varphi(|\xi|)+\sum_{j=1}^2 \xi_j \partial_{\xi_j}\left( \varphi(|\xi|)\right)
    =2\varphi(|\xi|)+\sum_{j=1}^2 \frac{\xi_j^2}{|\xi|}\partial_{|\xi|}\left( \varphi(|\xi|)\right)
    =2\varphi(|\xi|)+|\xi|\partial_{|\xi|}\left( \varphi(|\xi|)\right),
\end{equation*}
which yields
\begin{equation}\label{fourier of linearized eq}
\begin{split}
    \left(\Dlt\left((\Lambda-y\cdot\nb_y+1) \calJ\left[\frac{s}{\sqrt{|y|^2+1}}\right]\right)\right)^{\wedge}(\xi)&=-|\xi|^3\partial_{|\xi|}\left( \varphi(|\xi|)\right)-(|\xi|^3+3|\xi|^2)\varphi(|\xi|)\\
    &=-e^{-|\xi|}\partial_{|\xi|}\left(|\xi|^3 e^{|\xi|}\varphi(|\xi|)\right).
    \end{split}
\end{equation}
On the other hand, \eqref{eq: fourier transform of JC1C2}, together with \eqref{fourier transform of 1/sqrt y^2+1}, gives us
\begin{equation}\label{fourier of variphi}
        -|\xi|^2\varphi(|\xi|)=\frac{se^{-|\xi|}}{|\xi|}.
    \end{equation}
    Inserting this into \eqref{fourier of linearized eq}, we can check \eqref{another linearized equation acted laplacain}.
\end{proof}

Next we derive two estimates of $k^{Lin}_{s}$ in \eqref{klin}.
\begin{lemma}
    For $p\in(2,\infty)$ and $\alp \in \bbR^2$, there holds
    \begin{equation}\label{est: klin-1'}
\left|\dlt_\alp k^{Lin}_{s}(|y|)-\alp \cdot\nb_y k^{Lin}_{s}(|y|)\right|
        \lesssim s|\alp|^{2-\frac2p}.
    \end{equation}
\end{lemma}
\begin{proof}
 Recalling \eqref{klin nabla^2}, we have
\begin{equation*}
    \nrmb{\nb^2 k^{Lin}_{s}(|\cdot|)}_{L^p} \lesssim s,
\end{equation*}
since $p\in(2,\infty)$.
Combining this with \eqref{App of Morrey}, we are done.
\end{proof}

\begin{lemma}
    For $\alp\neq0,\,y\in\bbR^2$ , there holds
    \begin{equation}\label{est: klin-1}
    \begin{split}
        &\Abs{\left(\Delta_{\alp} k^{Lin}_{s}(|y|)\right)^2-\left(\Delta_{-\alp} k^{Lin}_{s}(|y|)\right)^2}\\
        &\quad\lesssim s^2\left(1_{\left\{|\alp|\le \frac{|y|}{2}\right\}}\left(\frac{|y||\alp|}{|y|^2+|\alp|^2+1}+\frac{|y|}{|\alp|\sqrt{|y|^2+|\alp|^2+1}}\right)+1_{\left\{|\alp|\ge \frac{|y|}{2}\right\}}\frac{|y|}{\sqrt{|y|^2+|\alp|^2+1}}\right).
        \end{split}
    \end{equation}
\end{lemma}
\begin{proof}
Our aim is to show that
\begin{equation}\label{est: klin-1--1}
        \Abs{\left(\dlt_{\alp}-\dlt_{-\alp}\right)  k^{Lin}_{s}(|y|)}\lesssim \frac{s|y||\alp|}{\sqrt{|y|^2+|\alp|^2+1}}
    \end{equation}
    and
    \begin{equation}\label{est: klin-1--2}
        \Abs{\left(\dlt_{\alp}+\dlt_{-\alp}\right)  k^{Lin}_{s}(|y|)}\lesssim s \left(1_{\left\{|\alp|\le \frac{|y|}{2}\right\}}\left(\frac{|\alp|^2}{\sqrt{|y|^2+|\alp|^2+1}}+1\right)+1_{\left\{|\alp|\ge \frac{|y|}{2}\right\}}|\alp|\right).
    \end{equation}
Then \eqref{est: klin-1} follows from
    \begin{equation*}
        \Abs{\left(\Delta_{\alp} k^{Lin}_{s}(|y|)\right)^2-\left(\Delta_{-\alp} k^{Lin}_{s}(|y|)\right)^2}\le \frac{ \Abs{\left(\dlt_{\alp}-\dlt_{-\alp}\right)  k^{Lin}_{s}(|y|)}\Abs{\left(\dlt_{\alp}+\dlt_{-\alp}\right)  k^{Lin}_{s}(|y|)}}{|\alp|^2}.
    \end{equation*}
    To begin with, for \eqref{est: klin-1--1}, it suffices to show that
    \begin{equation}\label{est: klin-1-2}
    \Abs{\sqrt{|y+\alp|^2+1}-\sqrt{|y-\alp|^2+1}}\lesssim \frac{|y||\alp|}{\sqrt{|y|^2+|\alp|^2+1}},
    \end{equation}
    \begin{equation}\label{est: klin-1-1}
    \Abs{\log\left(\sqrt{|y+\alp|^2+1}+1\right)-\log\left(\sqrt{|y-\alp|^2+1}+1\right)}\lesssim \frac{|y||\alp|}{\sqrt{|y|^2+|\alp|^2+1}}.
    \end{equation}
    For \eqref{est: klin-1-2},
    we compute
    \begin{equation*}
        \Abs{\sqrt{|y+\alp|^2+1}-\sqrt{|y-\alp|^2+1}}
        =\frac{\Abs{|y+\alp|^2-|y-\alp|^2}}{\sqrt{|y+\alp|^2+1}+\sqrt{|y-\alp|^2+1}}
        =\frac{4\Abs{y\cdot \alp}}{\sqrt{|y+\alp|^2+1}+\sqrt{|y-\alp|^2+1}}.
    \end{equation*}
    If $y\cdot \alp\ge 0$, then we estimate
    \begin{equation*}
        \frac{\Abs{y\cdot \alp}}{\sqrt{|y+\alp|^2+1}+\sqrt{|y-\alp|^2+1}}\le\frac{ |y||\alp|}{\sqrt{|y|^2+|\alp|^2+2 (y\cdot \alp) +1}}\le\frac{ |y||\alp|}{\sqrt{|y|^2+|\alp|^2 +1}}
    \end{equation*}
    while we do
    \begin{equation*}
        \frac{\Abs{y\cdot \alp}}{\sqrt{|y+\alp|^2+1}+\sqrt{|y-\alp|^2+1}}\le\frac{ |y||\alp|}{\sqrt{|y|^2+|\alp|^2-2 (y\cdot \alp) +1}}\le\frac{ |y||\alp|}{\sqrt{|y|^2+|\alp|^2 +1}}
    \end{equation*}
    whenever $y\cdot \alp < 0$.
    Next for \eqref{est: klin-1-1},
    we apply the mean value theorem to the function $F(x)=\log(x+1)$ to obtain $\tau\in(0,1)$ such that
    \begin{equation*}
            \Abs{\log\left(\sqrt{|y+\alp|^2+1}+1\right)-\log\left(\sqrt{|y-\alp|^2+1}+1\right)}=\frac{\Abs{\sqrt{|y+\alp|^2+1}-\sqrt{|y-\alp|^2+1}}}{\tau \sqrt{|y+\alp|^2+1} +(1-\tau)\sqrt{|y-\alp|^2+1}+1},
    \end{equation*}
    which satisfies \eqref{est: klin-1-1} by \eqref{est: klin-1-2}.

    Next we prove \eqref{est: klin-1--2}. For $|\alp|\ge \frac{|y|}{2}$, we recall \eqref{klin nabla} to obtain
    \begin{equation*}
    \begin{split}
        \Abs{\left(\dlt_{\alp}+\dlt_{-\alp}\right)  k^{Lin}_{s}(|y|)}&\lesssim  \Abs{k^{Lin}_{s}(|y|)-k^{Lin}_{s}(|y+\alp|)}+ \Abs{k^{Lin}_{s}(|y|)-k^{Lin}_{s}(|y-\alp|)} \\
        &\lesssim \Abs{\int_0^1 \nb_y k^{Lin}_{s} (|\tau y +(1-\tau) (y+\alp)|) \cdot \alp \, d\tau} \\
        &\quad + \Abs{\int_0^1 \nb_y k^{Lin}_{s} (|\tau y +(1-\tau) (y-\alp)|) \cdot \alp \, d\tau}\\
        &\lesssim s|\alp|.
        \end{split}
    \end{equation*}
We now consider the case when $|\alp|\le \frac{|y|}{2}$. We decompose
\begin{equation*}
    \begin{split}
        \Abs{\left(\dlt_{\alp}+\dlt_{-\alp}\right)  k^{Lin}_{s}(|y|)}
        &\le s\Abs{2\sqrt{|y|^2+1}-\sqrt{|y+\alp|^2+1}-\sqrt{|y-\alp|^2+1}} \\
        &\quad + s\Abs{2\log\left({\sqrt{|y|^2+1}}+1\right)-\log\left({\sqrt{|y+\alp|^2+1}+1}\right)-\log\left({\sqrt{|y-\alp|^2+1}+1}\right)}\\
        &= \mrI + \mrII.
        \end{split}
    \end{equation*}
    For $\mrI$, we compute
    \begin{equation*}
        \begin{split}
            \mrI
            &=\frac{2s\Abs{|y|^2+1-|\alp|^2-\sqrt{|y+\alp|^2+1}\sqrt{|y-\alp|^2+1}}}{2\sqrt{|y|^2+1}+\sqrt{|y+\alp|^2+1}+\sqrt{|y-\alp|^2+1}}\\
            & =\frac{8s\Abs{|y|^2|\alp|^2-(y\cdot\alp)^2+|\alp|^2}}{\left(2\sqrt{|y|^2+1}+\sqrt{|y+\alp|^2+1}+\sqrt{|y-\alp|^2+1}\right)\Abs{|y|^2+1-|\alp|^2+\sqrt{|y+\alp|^2+1}\sqrt{|y-\alp|^2+1}}}.
        \end{split}
    \end{equation*}
   Note that
\begin{equation*}
    \Abs{|y|^2|\alp|^2-(y\cdot\alp)^2+|\alp|^2}\le 2|\alp|^2\left(|y|^2+1\right),
\end{equation*}
    \begin{equation*}
        2\sqrt{|y|^2+1}+\sqrt{|y+\alp|^2+1}+\sqrt{|y-\alp|^2+1}\ge \sqrt{|y|^2+|\alp|^2 +1},
    \end{equation*}
    and for $|\alp|\le \frac{|y|}{2}$,
    \begin{equation*}
        \Abs{|y|^2+1-|\alp|^2+\sqrt{|y+\alp|^2+1}\sqrt{|y-\alp|^2+1}}\ge \Abs{|y|^2+1-|\alp|^2} \ge \frac{3}{4}\left(|y|^2+1\right).
    \end{equation*}
    Combining all, we arrive at
    \begin{equation*}
            \mrI \lesssim \frac{s|\alp|^2}{\sqrt{|y|^2+|\alp|^2 +1}}.
    \end{equation*}
    To estimate $\mrII$, we observe
    \begin{equation}\label{est: for log3-1}
        \frac{1}{10} \le \frac{\left(\sqrt{|y|^2+1}+1\right)^2}{\left(\sqrt{|y+\alp|^2+1}+1\right)\left(\sqrt{|y-\alp|^2+1}+1\right)} \le 10
    \end{equation}
    for $|\alp|\le \frac{|y|}{2}$. Indeed, $|\alp|\le \frac{|y|}{2}$ implies
    \begin{equation*}
        \left(\sqrt{|y+\alp|^2+1}+1\right)\left(\sqrt{|y-\alp|^2+1}+1\right)\le \left(\sqrt{(|y|+|\alp|)^2+1}+1\right)^2 \le 10\left(\sqrt{|y|^2+1}+1\right)^2
    \end{equation*}
    and
     \begin{equation*}
        10\left(\sqrt{|y+\alp|^2+1}+1\right)\left(\sqrt{|y-\alp|^2+1}+1\right)\ge 10\left(\sqrt{(|y|-|\alp|)^2+1}+1\right)^2 \ge \left(\sqrt{|y|^2+1}+1\right)^2.
    \end{equation*}
    Hence \eqref{est: for log3-1} provide us with
    \begin{equation*}
        \mrII=s\Abs{\log \left( \frac{\left(\sqrt{|y|^2+1}+1\right)^2}{\left(\sqrt{|y+\alp|^2+1}+1\right)\left(\sqrt{|y-\alp|^2+1}+1\right)}\right)}\lesssim s.
    \end{equation*}
\end{proof}

\subsubsection{Inhomogeneous linearized equation}\label{subsubsec: inhomo}
To begin with, we define the linear operator $\calL$ by
\begin{equation}\label{linear operateor L}
    \calL: f(r)\mapsto \int_{\bbR_\ge0} 1_{\left\{s\le r\right\}}\frac{s^2 e^{s-r}}{r^3}  f(s)\,ds
\end{equation}
for a locally integrable function $f:\bbR_{\ge0}\rightarrow\bbR$.
Then $\calL$ has the following property.
\begin{lemma}\label{lem: linear operator L estimate}
    Given $t_1\in[1,2)$, there holds
    \begin{equation}\label{est:2 for lemmm-1}
        \nrmb{\calL[f(r)]}_{L^2(\bbR_{\ge0};\;r^{1+2t}dr)}\lesssim \nrmb{f(r)}_{L^2(\bbR_{\ge0};\;r^{1+2t_1}dr)}
    \end{equation}
    for every $t\in(t_1, t_1+1]$.
\end{lemma}
\begin{remark}\label{remark for inverse operator of L}
    For a radially symmetric function $h(|y|)$ and the operator $\calL$ defined by \eqref{linear operateor L}, we consider the linear operator $\widehat{\calL}$ given by
    \begin{equation}\label{linear operator L hat}
    \widehat{\calL}[h(|\cdot|)]:=\left(\calL\left[\widehat{h(|\cdot|)}\right]\right)^\lor.
\end{equation}
Then for any $t_1\in[1,2)$, \eqref{est:2 for lemmm-1} implies that
\begin{equation}\label{est:2 for lemmm}
        \nrmb{\widehat{\calL}(h(|\cdot|))}_{\dot{H}^{t}\cap\dot{H}^{t_1+1}(\bbR^2)}\lesssim \nrmb{h(|\cdot|)}_{\dot{H}^{t_1}(\bbR^2)}
    \end{equation}
    for every $t\in(t_1, t_1+1]$.
\end{remark}
\begin{proof}
Let $t_0\in[0,1)$ be given. We aim to show that
\begin{equation*}
    \nrmb{r^{\frac32+t}\calL[f(r)]}_{L^2(\bbR_{\ge0})} \lesssim \nrmb{r^{\frac32+t_0}f(r)}_{L^2(\bbR_{\ge0})}
\end{equation*}
for every $t\in(t_0, t_0+1]$.
    Let us write $t=t_0+\tilde{t}$ for $\tilde{t}\in(0,1]$. 
    First of all, we consider the case when $0\le r<2$.
    Employing \eqref{linear operateor L} and the H\"older's inequality, we have
    \begin{equation*}
    \begin{split}
        \Abs{1_{\left\{0\le r <2\right\}}r^{\frac32+t}\calL[f(r)]}^2 
        &=1_{\left\{0\le r <2\right\}}\left(\int_{s=0}^{r} r^{-\frac32 +t_0+\tilde{t}}s^2 e^{s-r}f(s)\,ds\right)^2\\
        &\le 1_{\left\{0\le r <2\right\}}\left(\int_{s=0}^{r} r^{-\frac32 +t_0+\tilde{t}}s^2 f(s)\,ds\right)^2\\
        &\le 1_{\left\{0\le r <2\right\}}\left(\int_{s=0}^{r} s^{-2t_0+1} \,ds\right)\left(\int_{s=0}^{r} r^{-3+2t_0+2\tilde{t}} s^{2t_0+3} |f(s)|^2\,ds\right) \\
         &\approx  1_{\left\{0\le r <2\right\}}r^{-1+2\tilde{t}}\int_{\bbR} 1_{\left\{0\le s \le r\right\}}s^{2t_0+3}|f(s)|^2\,ds.
    \end{split}
    \end{equation*}
    Hence using Fubini's theorem, we obtain
    \begin{equation*}
        \nrmb{1_{\left\{0\le r <2\right\}}r^{\frac32+t}\calL[f(r)]}_{L^2(\bbR_{\ge0})}^2 \lesssim \int_{ \bbR_{\ge0}}  \left(\int_{\bbR} 1_{\left\{0 \le r<2\right\}}r^{-1+2\tilde{t}} dr\right)s^{2t_0+3}|f(s)|^2\,ds \lesssim \nrmb{s^{\frac32+t_0}f(s)}_{L^2(\bbR_{\ge0})},
    \end{equation*}
    where we used $\tilde{t}>0$ on the last line.
    
    Next, we consider the case when $r\ge 2$. We  decompose
    \begin{equation*}
    \begin{split}
        \Abs{1_{\left\{r \ge 2\right\}}r^{\frac32+t}\calL[f(r)]}^2 
        &\lesssim 1_{\left\{r\ge 2\right\}}\left(\int_{s=0}^{1} r^{-\frac32 +t_0+\tilde{t}}s^2 e^{s-r}f(s)\,ds\right)^2+1_{\left\{r\ge 2\right\}}\left(\int_{s=1}^{r} r^{-\frac32 +t_0+\tilde{t}}s^2 e^{s-r}f(s)\,ds\right)^2\\
        &=\mrI+\mrII,
    \end{split}
    \end{equation*}
    which implies
    \begin{equation*}
        \nrmb{1_{\left\{r \ge 2\right\}}r^{\frac32+t}\calL[f(r)]}_{L^2(\bbR_{\ge0})}^2\lesssim \nrmb{\mrI}_{L^1(\bbR_{\ge0})}+\nrmb{\mrII}_{L^1(\bbR_{\ge0})}.
    \end{equation*}
For $\nrmb{\mrI}_{L^1(\bbR_{\ge0})}$, we use $t_0\in[0,1)$, $\tilde{t}\in(0,1]$, and the H\"older's inequality to obtain
\begin{equation*}
    \begin{split}
        \mrI&\lesssim 1_{\left\{r\ge 2\right\}} re^{-2r} \left(\int_{s=0}^{1} s^{-2t_0+1} ds\right) \left(\int_{s=0}^{1} s^{2t_0+3} |f(s)|^2 ds\right)\\
        &\lesssim 1_{\left\{r\ge 2\right\}} re^{-2r} \int_{\bbR_{\ge0}} s^{2t_0+3} |f(s)|^2 ds,
    \end{split}
\end{equation*}
which yields
\begin{equation*}
    \nrmb{\mrI}_{L^1(\bbR_{\ge0})}\lesssim \nrmb{s^{\frac32+t_0}f(s)}_{L^2(\bbR_{\ge0})}.
\end{equation*}
For $\nrmb{\mrII}_{L^1(\bbR_{\ge0})}$, we first observe that 
    \begin{equation}\label{replacement of gamma function}
        \int_{s=1}^{r} s^{-t_0+1}e^{s-r}\,ds
        \le r^{-t_0+1}
    \end{equation}
    for $0\le t_0<1$ and $r\ge2$. Indeed, integrating the both side of the following inequality:
    \begin{equation*}
        s^{-t_0+1}e^{s}\le \left((-t_0+1)s^{-t_0}+s^{-t_0+1}\right)e^{s}=\frac{d}{ds}\left(s^{-t_0+1}e^{s}\right)
    \end{equation*}
    from $1$ to $r$ with respect to $s$,
    we have
     \begin{equation*}
        \int_{s=1}^{r} s^{-t_0+1}e^{s}\,ds
        \le r^{-t_0+1}e^{r}-e\le r^{-t_0+1}e^{r}.
    \end{equation*}
    Employing the H\"older's inequality and \eqref{replacement of gamma function}, we have
    \begin{equation*}
    \begin{split}
        \mrII 
        &\le 1_{\left\{r\ge 2\right\}}\left(\int_{s=1}^{r} s^{-t_0+1} e^{s-r}\,ds\right)\left(\int_{s=1}^{r} r^{-3+2t_0+2\tilde{t}}  e^{s-r}s^{t_0+3}|f(s)|^2\,ds\right) \\
        &\le 1_{\left\{r\ge 2\right\}}r^{-2+t_0+2\tilde{t}}e^{-r} \int_{\bbR} 1_{\left\{1\le s\le r\right\}} e^{s}s^{t_0+3}|f(s)|^2\,ds.
    \end{split}
    \end{equation*}
    Hence using Fubini's theorem, we obtain
    \begin{equation*}
    \begin{split}
        \nrmb{\mrII}_{L^1(\bbR_{\ge0})} &\lesssim \int_{s=1}^{\infty}  \left(\int_{ r=s}^{\infty} 1_{\left\{r\ge 2\right\}}r^{-2+t_0+2\tilde{t}}e^{-r} dr\right)e^{s}s^{t_0+3}|f(s)|^2\,ds \\
         &\le \int_{s=1}^{\infty}\left(\int_{ r=s}^{\infty}r^{t_0}e^{-r} dr\right) e^{s}s^{t_0+3}|f(s)|^2\,ds,
        \end{split}
    \end{equation*}
    where we used $\tilde{t}\in(0,1]$ and $r\ge2$ on the last line.
We observe that
\begin{equation}\label{replacement of gamma function-1}
    \int_{ r=s}^{\infty}r^{t_0}e^{-r} dr\lesssim s^{t_0}e^{-s} 
\end{equation}
for $0\le t_0<1$ and $s\ge1$.
Indeed, noticing the inequality
\begin{equation*}
    r^{t_0}e^{-r}\le \frac{1}{1-t_0}\left(-t_0r^{t_0-1}+r^{t_0}\right)e^{-r}=-\frac{1}{1-t_0}\frac{d}{dr}\left(r^{t_0}e^{-r}\right),
\end{equation*}
for $0\le t_0<1$ and  $r\ge1$,
 we integrate the both side of this 
from $s$ to $\infty$ with respect to $r$,
we have
\begin{equation*}
    \int_{ r=s}^{\infty}r^{t_0}e^{-r} dr\le -\frac{1}{1-t_0}\int_{ r=s}^{\infty}\frac{d}{dr}\left(r^{t_0}e^{-r}\right)\,ds\\
    = \frac{1}{1-t_0}s^{t_0}e^{-s}.
\end{equation*}
Thus using \eqref{replacement of gamma function-1}, we estimate
\begin{equation*}
        \nrmb{\mrII}_{L^1(\bbR_{\ge0})} \lesssim  \nrmb{s^{\frac32+t_0}f(s)}_{L^2(\bbR_{\ge0})}.
    \end{equation*}
\end{proof}

We are now ready to explore the inhomogeneous counterpart of \eqref{linearized eq}:
\begin{equation}\label{linearized eq-1}
    (\Lambda-y\cdot\nb_y+1) k(|y|) =h(|y|).
\end{equation}
Our goal is to solve the above equation, using the function $k^{Lin}_{s}(|y|)$ from \eqref{klin}.
\begin{lemma}\label{lem: k=klin+L(h)}
Let $t_1\in(3/2,2)$ and let $h(|y|)$ be a function satisfying $\nrmb{h(|\cdot|)}_{\dot{H}^{t_1}(\bbR^2)}<\infty$. Suppose that there holds 
\begin{equation*}
    \calJ[\phi](|y|)=\widehat{\calL}[h(|y|)] +C_1
\end{equation*}
for a constant $C_1$ and a function $\phi\in \dot{H}^{t_1-1}(\bbR^2)$, where $\calJ$ and $\widehat{L}$ are the operators given in \eqref{eq: inverse laplacian-1} and \eqref{linear operator L hat}, respectively. Then there exists a constant $C_2$ such that the function $k(|y|)$ defined by 
\begin{equation}\label{k=klin+L(h)}
    k(|y|):=k^{Lin}_{s}(|y|)+ \calJ[\phi](|y|)+C_2
\end{equation}
is a solution of \eqref{linearized eq-1}.
\end{lemma}
\begin{proof}
Note that  $\widehat{\calL}[h(|y|)]$ is well-defined by Remark \ref{remark for inverse operator of L}. To prove that $k(|y|)$ defined by \eqref{k=klin+L(h)} is a solution of \eqref{linearized eq-1}, we shall proceed similarly to the proof of Lemma \ref{lem: klin-1}. Specifically we show that \begin{equation}\label{WLp est for Lh}
    -h(|y|)+(\Lambda-y\cdot\nb_y+1) \calJ[\phi](|y|)\in \dot{W}^{1,\frac{2}{2-t_1}}(\bbR^2),
\end{equation}
and
\begin{equation}\label{another inhomo linearized equation acted laplacain}
    \Dlt \left(-h(|y|)+(\Lambda-y\cdot\nb_y+1) \calJ[\phi](|y|)\right) =0.
\end{equation}
Then, both \eqref{WLp est for Lh} and \eqref{another inhomo linearized equation acted laplacain} yield a constant $C_3$ such that
\begin{equation}\label{eq: for calL(h)}
        (\Lambda-y\cdot\nb_y+1) \calJ[\phi](|y|) =h(|y|)+C_3
    \end{equation}
by \eqref{uniqueness of u} in  Proposition \ref{prop: poisson radial}. But since Lemma \ref{lem: klin-1} gives a constant $C$ satisfying
\begin{equation*}
        (\Lambda-y\cdot\nb_y+1) \left(k^{Lin}_{s}(|y|)+C\right) =0,
    \end{equation*}
    we can compute
    \begin{equation*}
    \begin{split}
        &(\Lambda-y\cdot\nb_y+1) \left(k^{Lin}_{s}(|y|)+\calJ[\phi](|y|)+C-C_3\right) \\
        &\quad=-C_3+(\Lambda-y\cdot\nb_y+1) \left(k^{Lin}_{s}(|y|)+C\right)+(\Lambda-y\cdot\nb_y+1) \calJ[\phi](|y|)=h(|y|),
        \end{split}
    \end{equation*}
in other words, $k^{Lin}_{s}(|y|)+\calJ[\phi](|y|)+C-C_3$ is a solution of \eqref{linearized eq-1}.

    The claim \eqref{WLp est for Lh} follows from Remark \ref{rmk: radial poisson} and the Sobolev embedding \eqref{sobolev embedding-1}:
    \begin{equation*}
        \nrmb{\nb h(|\cdot|)}_{L^{\frac{2}{2-t_1}}(\bbR^2)}\lesssim \nrmb{\nb h(|\cdot|)}_{\dot{H}^{t_1-1}(\bbR^2)}\approx \nrmb{h(|\cdot|)}_{\dot{H}^{t_1}(\bbR^2)}.
    \end{equation*}
    For \eqref{another inhomo linearized equation acted laplacain}, we note that
    \begin{equation*}
    \Dlt \left(-h(|y|)+(\Lambda-y\cdot\nb_y+1) \calJ[\phi](|y|)\right) = \Dlt \left(-h(|y|)+(\Lambda-y\cdot\nb_y+1) \widehat{\calL}[h(|y|)]\right).
\end{equation*}
    Denoting $\widehat{h(|\cdot|)}(\xi)=\widehat{h}(|\xi|)$ and computing similarly to \eqref{fourier of linearized eq}, we can observe that 
    \begin{equation*}
      \left(\Dlt\left(-h(|y|)+(\Lambda-y\cdot\nb_y+1) \widehat{\calL}[h(|y|)]\right)\right)^{\wedge}(\xi)=|\xi|^2\widehat{h}(|\xi|)-e^{-|\xi|}\partial_{|\xi|}\left( |\xi|^3e^{|\xi|}\calL\left[\widehat{h}(|\xi|)\right]\right).
    \end{equation*}
    Recalling \eqref{linear operateor L}, we obtain
    \begin{equation*}
    \begin{split}
        &|\xi|^2\widehat{h}(|\xi|)-e^{-|\xi|}\partial_{|\xi|}\left( |\xi|^3e^{|\xi|}\calL\left[\widehat{h}(|\xi|)\right]\right)\\
        &\qquad=|\xi|^2\widehat{h}(|\xi|)-e^{-|\xi|}\partial_{|\xi|}\left(|\xi|^3e^{|\xi|} \int_{|\eta|=0}^{|\xi|} \frac{|\eta|^2 e^{|\eta|-|\xi|}}{|\xi|^3} \widehat{h}(|\eta|)\,d|\eta|\right)=0,
        \end{split}
    \end{equation*}
    which gives \eqref{another inhomo linearized equation acted laplacain}.
\end{proof}

\subsection{Fixed point formulation}\label{subsec: fixed point formula}
In this subsection, we introduce a fixed point formulation and explain how we will make use of it to obtain Theorem \ref{thm A}. Let $\calT$,$\calJ$, and $\widehat{\calL}$ be operators defined in \eqref{def: nonlinear operator}, \eqref{eq: inverse laplacian-1}, and \eqref{linear operator L hat}, respectively, and let $k^{Lin}_{s}$ be the function given in \eqref{klin}.
Hereafter, we denote
\begin{equation}\label{def: subTs}
\begin{split}
    \calT_1[g(|y|)] &:=\left.\frac{d}{d\tau}\left[\calT\left(k^{Lin}_{s}(|y|) + \tau g(|y|)\right)\right]\right|_{\tau=0}=d\calT\left[k^{Lin}_{s}(|y|)\right]g(|y|), \\
    \calT_{\ge 2}[g(|y|)] &:= \calT\left[k^{Lin}_{s}(|y|) + g(|y|)\right]-\calT\left[k^{Lin}_{s}(|y|) \right]-\calT_1\left[g(|y|)\right].
\end{split}
\end{equation}
Our fixed point formulation is
\begin{equation}\label{eq: g-fixed}
   \tilde{g}(|y|)=\Phi(|y|)+\calA[\tilde{g}(|y|)]+\calN[\tilde{g}(|y|)]
\end{equation}
with forcing term
\begin{equation}\label{eq: g-fixed-forcing}
    \Phi(|y|):=\left(\Dlt \circ \widehat{\calL}\circ \calT\right)\left[k^{Lin}_{s}(|y|) \right],
\end{equation}
the linear operator
\begin{equation}\label{eq: g-fixed-linear}
    \calA[\tilde{g}(|y|)]:=\left(\Dlt\circ\widehat{\calL}\circ  \calT_1\circ\calJ\right)\left[\tilde{g}(|y|)\right],
\end{equation}
and the nonlinear operator
\begin{equation}\label{eq: g-fixed-nonlinear}
    \calN[\tilde{g}(|y|)]:=\left(\Dlt\circ\widehat{\calL}\circ  \calT_{\ge2}\circ \calJ\right)\left[\tilde{g}(|y|)\right],
\end{equation}
where $\Dlt$ denotes the Laplacian.
Here are steps detailing how we reach from \eqref{eq: g-fixed} to Theorem \ref{thm A}.

\medskip

\textbf{Step 1.} We first estimate the forcing term $\Phi(|\cdot|)$.
\begin{lemma}\label{lem: l^2 estimate for T(klin)}
    Let $s\in(0,1)$, $t_1\in (3/2,2)$ be given. Then
    \begin{equation*}
         \nrmb{\calT\left[k^{Lin}_{s}(|\cdot|) \right]}_{\dot{H}^{t_1}} \lesssim s^3.
    \end{equation*}
\end{lemma}
The proof of Lemma \ref{lem: l^2 estimate for T(klin)} can be found in Section \ref{sec: forcing term}.
Furthermore, Lemma \ref{lem: l^2 estimate for T(klin)}, together with \eqref{est:2 for lemmm} and \eqref{eq: g-fixed-forcing}, gives us
\begin{lemma}\label{lem: estimate for phi}
    Let $s\in(0,1)$, $t_1\in (3/2,2)$ be given. Then for any  $t\in\left(t_1,t_1+1\right]$, there exists a constant $C_{\Phi}>0$ such that
    \begin{equation*}
        \nrmb{\Phi(|\cdot|)}_{\dot{H}^{t-2} \cap \dot{H}^{t_1-1}} \le C_{\Phi} s^3.
    \end{equation*}
\end{lemma}

\textbf{Step 2.} Next we estimate the linear term $\calA[g(|\cdot|)]$.
\begin{lemma}\label{lem: l^2 estimate for T_1(g)}
   Let $s\in(0,1)$, $t_1\in (3/2,2)$ be given. Then there exists $t^{*}=t^*(t_1)\in (t_1,2)$ such that 
    \begin{equation*}
        \nrmb{\calT_1\left[g(|\cdot|) \right]}_{\dot{H}^{t_1}} \lesssim s^2\nrmb{g(|\cdot|)}_{\dot{H}^{t^*} \cap \dot{H}^{t_1+1}}.
    \end{equation*}
\end{lemma}
The proof of Lemma \ref{lem: l^2 estimate for T_1(g)} is in Section \ref{sec: linear term}.
Moreover, Proposition \ref{prop: poisson radial}, \eqref{est:2 for lemmm}, \eqref{eq: g-fixed-linear}, Lemma \ref{lem: l^2 estimate for T_1(g)},  and properties of Neumann series give us
\begin{lemma}\label{lem: estimate for A(g)}
    Let $s\in(0,1)$, $t_1\in (3/2,2)$ be given. Then there exists $t^{*}=t^*(t_1)\in (t_1,2)$ such that for any $t\in\left(t_1,t_1+1\right]$, there exists a constant $C_{\calA}>0$ such that
    \begin{equation}\label{final est: for A}
        \nrmb{\calA[\tilde{g}(|\cdot|)]
        }_{\dot{H}^{t-2} \cap \dot{H}^{t_1-1}} \le 
C_{\calA} s^2\nrmb{\tilde{g}(|\cdot|)}_{\dot{H}^{t^*-2} \cap \dot{H}^{t_1-1}}.
    \end{equation}
    Moreover, $Id-A$ is invertible whenever $s\in (0,s_*)$, where $Id$ is the identity operator and $s_*$ is a constant satisfying
    \begin{equation}\label{condition for s_* for just I-A}
        C_{\calA}s_*^2<1.
    \end{equation}
\end{lemma}

\textbf{Step 3.} In this step, we estimate the nonlinear term $\calN(g(|\cdot|))$.
\begin{lemma}\label{lem: l^2 estimate for T_>2(g)}
    Let $s\in(0,1)$, $t_1\in (3/2,2)$ be given. Then $\calT_{\ge 2}\left[0\right]=0$, and there exists $t^{*}=t^*(t_1)\in (t_1,2)$ such that if
    \begin{equation}\label{ineq: g_1+g_2 < 1}
        \nrmb{g_1(|\cdot|)}_{\dot{H}^{t^*} \cap \dot{H}^{t_1+1}}+\nrmb{g_2(|\cdot|)}_{\dot{H}^{t^*} \cap \dot{H}^{t_1+1}}\le 1,
    \end{equation}
    then
    \begin{equation*}
    \begin{split}
         &\nrmb{\calT_{\ge 2}\left[g_1(|\cdot|) \right]-\calT_{\ge 2}\left[g_2(|\cdot|) \right]}_{\dot{H}^{t_1}}\\
         &\qquad\lesssim \left(s^2+\nrmb{g_1(|\cdot|)}_{\dot{H}^{t^*} \cap \dot{H}^{t_1+1}}^2+\nrmb{g_2(|\cdot|)}_{\dot{H}^{t^*} \cap \dot{H}^{t_1+1}}^2\right)\nrmb{g_1(|\cdot|)-g_2(|\cdot|)}_{\dot{H}^{t^*} \cap \dot{H}^{t_1+1}}.
         \end{split}
    \end{equation*}
\end{lemma}
The proof of Lemma \ref{lem: l^2 estimate for T_>2(g)} is in Section \ref{sec: nonlinear term}.
Furthermore, Proposition \ref{prop: poisson radial}, \eqref{est:2 for lemmm}, \eqref{eq: g-fixed-nonlinear}, Lemma \ref{lem: l^2 estimate for T_>2(g)} give us
\begin{lemma}\label{lem: estimate for N(g)}
Under the assumptions: $s\in(0,1)$ and 
\begin{equation*}
        \nrmb{\tilde{g}_1(|\cdot|)}_{\dot{H}^{t^*-2} \cap \dot{H}^{t_1-1}}+\nrmb{\tilde{g}_2(|\cdot|)}_{\dot{H}^{t^*-2} \cap \dot{H}^{t_1-1}}\le 1,
    \end{equation*}
there exists a constant $C_{\calN}>0$ such that
    \begin{equation}\label{final est: for N}
    \begin{split}
    &\nrmb{\calN\left[\tilde{g}_1(|\cdot|) \right]}_{\dot{H}^{t-2} \cap \dot{H}^{t_1-1}} \le C_{\calN} \left(s^2+\nrmb{\tilde{g}_1(|\cdot|)}_{\dot{H}^{t^*-2} \cap \dot{H}^{t_1-1}}^2\right)\nrmb{\tilde{g}_1(|\cdot|)}_{\dot{H}^{t^*-2} \cap \dot{H}^{t_1-1}},\\
        &\nrmb{\calN\left[\tilde{g}_1(|\cdot|) \right]-\calN\left[\tilde{g}_2(|\cdot|) \right]}_{\dot{H}^{t-2} \cap \dot{H}^{t_1-1}}\\
        &\qquad \le C_{\calN} \left(s^2+\nrmb{\tilde{g}_1(|\cdot|)}_{\dot{H}^{t^*-2} \cap \dot{H}^{t_1-1}}^2+\nrmb{\tilde{g}_2(|\cdot|)}_{\dot{H}^{t^*-2} \cap \dot{H}^{t_1-1}}^2\right)\nrmb{\tilde{g}_1(|\cdot|)-\tilde{g}_2(|\cdot|)}_{\dot{H}^{t^*-2} \cap \dot{H}^{t_1-1}}
        \end{split}
    \end{equation}
    for any $t\in\left(t_1,t_1+1\right]$.
\end{lemma}

\textbf{Step 4.} With the aid of Step 1. - Step 3., we obtain the following proposition.
\begin{proposition}\label{prop 1}
    Given $t_1\in\left(3/2,2\right)$, there exists $s_*$ and $t^*\in(t_1,2)$ such that if $s\in(0,s_*)$, then there exists a unique solution $\tilde{g}_s(|y|)$ to \eqref{eq: g-fixed} satisfying
    \begin{equation*}
        \nrmb{\tilde{g}_s(|\cdot|)}_{\dot{H}^{t^*-2} \cap \dot{H}^{t_1-1}}\lesssim s^3.
    \end{equation*}
\end{proposition}
\begin{proof}
Abusing the notation, we define $t^*$ as the smaller of the $t^*$'s appearing in Lemma \ref{lem: estimate for A(g)} and Lemma \ref{lem: estimate for N(g)}. For this newly defined $t^*$, both \eqref{final est: for A} and \eqref{final est: for N} still hold by \eqref{sobolev interpolation}. Moreover we can put $t^*$ into $t$'s in both \eqref{final est: for A} and \eqref{final est: for N} because $t^*\in(t_1,2)$. 
Now we set a small number $s_*>0$ satisfying
\begin{equation}\label{condition: s_*}
    10\left(C_{\Phi}+C_{\calA}+C_{\calN}\right)s_*^2\le 1.
\end{equation}
Then for each $s\in(0,s_*)$, we define 
\begin{equation}\label{condition: eps}
   \eps=\eps(s)= 10\left(C_{\Phi}+C_{\calA}+C_{\calN}\right)s^3.
\end{equation}
For such $s\in(0,s_*)$, $\eps>0$, we consider
\begin{equation*}
    X_\eps:=\left\{\tilde{g} \in\dot{H}^{t^*-2} \cap \dot{H}^{t_1-1}  :\; \nrmb{\tilde{g}}_{\dot{H}^{t^*-2} \cap \dot{H}^{t_1-1}}\le \eps  \right\}
\end{equation*}
and a map $\Psi$ given by
    \begin{equation*}
        \Psi: \tilde{g}(|\cdot|) \mapsto (Id-\calA)^{-1}\left[\Phi(|\cdot|)+\calN[\tilde{g}(|\cdot|)]\right].
    \end{equation*}
    This map is well defined since $s_*$ satisfies \eqref{condition for s_* for just I-A}. Moreover using \eqref{condition: s_*}, we can check
    \begin{equation}\label{est: for inverse (id-A) less than 2}
        \nrmb{\left(Id-\calA\right)^{-1}}_{\dot{H}^{t^*-2} \cap \dot{H}^{t_1-1} \rightarrow \dot{H}^{t^*-2} \cap \dot{H}^{t_1-1}}\le \frac{1}{1-\nrmb{A}_{\dot{H}^{t^*-2} \cap \dot{H}^{t_1-1} \rightarrow \dot{H}^{t^{*}-1} \cap \dot{H}^{t_1}}}\le \frac{1}{1-C_{\calA}s^2}\le 2.
    \end{equation}
    Hence for $\tilde{g}(|\cdot|)\in X_\eps$, we use Lemma \ref{lem: estimate for phi}, Lemma \ref{lem: estimate for N(g)}, \eqref{condition: s_*} - \eqref{est: for inverse (id-A) less than 2} to estimate
    \begin{equation*}
    \begin{split}
\nrmb{\Psi\left[\tilde{g}(|\cdot|)\right]}_{\dot{H}^{t^*-2} \cap \dot{H}^{t_1-1}}&=\nrmb{(Id-\calA)^{-1}\left[\Phi(|\cdot|)+\calN[\tilde{g}(|\cdot|)]\right]}_{\dot{H}^{t^*-2} \cap \dot{H}^{t_1-1}}\\
        &\le 2\left(C_\Phi + C_\calN\right)\left(s^3+(s^2+\eps^2)\eps\right) \\
        &= 2\left(C_\Phi + C_\calN\right)s^3\left(1+10\left(C_{\Phi}+C_{\calA}+C_{\calN}\right)s^2+\left(10\left(C_{\Phi}+C_{\calA}+C_{\calN}\right)s^2\right)^3\right)\\
        &\le 6\left(C_\Phi + C_\calN\right)s^3\le \eps.
        \end{split}
    \end{equation*}
    This implies $\Psi: X_\eps \rightarrow X_\eps$. Furthermore, for $\tilde{g}_1(|\cdot|),\,\tilde{g}_2(|\cdot|)\in X_\eps$, we again use Lemma \ref{lem: estimate for N(g)}, \eqref{condition: s_*} - \eqref{est: for inverse (id-A) less than 2} to estimate
    \begin{equation*}
    \begin{split}
        &\nrmb{\Psi\left[\tilde{g}_1(|\cdot|)\right]-\Psi\left[\tilde{g}_2(|\cdot|)\right]}_{\dot{H}^{t^*-2} \cap \dot{H}^{t_1-1}}\\
        &\qquad=\nrmb{(Id-\calA)^{-1}\left[\calN\left[\tilde{g}_1(|\cdot|)\right]-\calN\left[\tilde{g}_2(|\cdot|)\right]\right]}_{\dot{H}^{t^*-2} \cap \dot{H}^{t_1-1}}\\
        &\qquad\le 2C_{\calN}\left(s^2+2\eps^2\right)\nrmb{\tilde{g}_1(|\cdot|)-\tilde{g}_2(|\cdot|)}_{\dot{H}^{t^*-2} \cap \dot{H}^{t_1-1}}\\
        &\qquad= 2C_{\calN}s^2\left(1+2\left(10\left(C_{\Phi}+C_{\calA}+C_{\calN}\right)s^2\right)^2\right)\nrmb{\tilde{g}_1(|\cdot|)-\tilde{g}_2(|\cdot|)}_{\dot{H}^{t^*-2} \cap \dot{H}^{t_1-1}}\\
        &\qquad \le 6C_{\calN}s_*^2\nrmb{\tilde{g}_1(|\cdot|)-\tilde{g}_2(|\cdot|)}_{\dot{H}^{t^*-2} \cap \dot{H}^{t_1-1}}\\
        &\qquad \le \frac{3}{5}\nrmb{\tilde{g}_1(|\cdot|)-\tilde{g}_2(|\cdot|)}_{\dot{H}^{t^*-2} \cap \dot{H}^{t_1-1}},
        \end{split}
    \end{equation*}
    which implies $\Psi$ is a contraction on $X_\eps$. Since $X_\eps$ is a Banach space, the Banach fixed point theorem ensures the existence of $\tilde{g}_s\in X_\eps$ satisfying $\tilde{g}_s=\Psi(\tilde{g}_s)$.
\end{proof}

\textbf{Step 5.} Finally, we show the existence of a solution to our original equation \eqref{eq: reformulation}, using $\tilde{g}_s(|y|)$ obtained in Proposition \ref{prop 1}.
To do so, we recall \eqref{def: subTs} and the fact that $\widehat{\calL}$ is the linear operator, which yield
\begin{equation*}
\Phi(|y|)+\calA[\tilde{g}_s(|y|)]+\calN[\tilde{g}_s(|y|)]=\Dlt\widehat{\calL}\left[ \calT\left[k^{Lin}_{s}(|y|)+\calJ[\tilde{g}_s](|y|) \right]\right].
\end{equation*}
Hence, we see
\begin{equation}\label{eq: for step 5}
    \tilde{g}_s(|y|)=\Dlt\widehat{\calL}\left[ \calT\left[k^{Lin}_{s}(|y|)+ \calJ[\tilde{g}_s](|y|) \right]\right],
\end{equation}
since $\tilde{g}_s(|y|)$ is a solution to \eqref{eq: g-fixed}. 
On the other hand, the fact that $\tilde{g}_s\in X_\eps$, together with \eqref{condition: eps}, yields
\begin{equation*}\label{est: final g H t_1-1}
    \nrmb{\tilde{g}_{s}(|y|)}_{\dot{H}^{t_1-1}}\lesssim s^3.
\end{equation*}
Therefore, Proposition \ref{prop: poisson radial} implies that $\calJ[\tilde{g}_s](|y|)$ is a solution of the Poisson equation 
$\tilde{g}_s(|y|)=\Dlt u(|x|)$ satisfying
\begin{equation*}
\sum_{0\le|\beta|\le 2} \nrmb{\frac{\nb^{\beta}_y\calJ[\tilde{g}_s(|y|)]}{|y|^{t_1-|\beta|}}}_{L^\infty}\lesssim s^3.
    \end{equation*}
Now, we compare $\widehat{\calL}\left[ \calT\left[k^{Lin}_{s}(|y|)+ \calJ[\tilde{g}_s](|y|) \right]\right]$ with $\calJ[\tilde{g}_s](|y|)$, utilizing Proposition \ref{prop: poisson radial}. 
To do so, we claim that
\begin{equation*}
    \widehat{\calL}\left[ \calT\left[k^{Lin}_{s}(|y|)+ \calJ[\tilde{g}_s](|y|) \right]\right]\in \dot{W}^{1,\infty}(\bbR^2).
\end{equation*}
Indeed, the Sobolev embedding \eqref{sobolev embedding-3(infty)}, together with \eqref{eq: for step 5} and the fact that $\tilde{g}_s \in X_\eps$, gives
\begin{equation*}
    \begin{split}
        \nrmb{\nb \widehat{\calL}\left[ \calT\left[k^{Lin}_{s}(|y|)+ \calJ[\tilde{g}_s](|y|) \right]\right]}_{L^{\infty}}&\lesssim \nrmb{\widehat{\calL}\left[ \calT\left[k^{Lin}_{s}(|y|)+ \calJ[\tilde{g}_s](|y|) \right]\right]}_{\dot{H}^{t^*}\cap\dot{H}^{t_1+1}}\\
        &\approx \nrmb{\tilde{g}_{s}(|y|)}_{\dot{H}^{t^*-2}\cap\dot{H}^{t_1-1}}\lesssim \eps.
    \end{split}
\end{equation*}
Hence, \eqref{uniqueness of u} in Proposition \ref{prop: poisson radial} ensures the existence of a constant $C_1$ such that
\begin{equation*}
    \widehat{\calL}\left[ \calT\left[k^{Lin}_{s}(|y|)+ \calJ[\tilde{g}_s](|y|) \right]\right]=\calJ[\tilde{g}_s](|y|)+C_1.
\end{equation*}
This, together with Lemma \ref{lem: k=klin+L(h)}, guarantees the existence of a constant $C_2$ such that 
\begin{equation*}
    (\Lambda-y\cdot\nb_y+1) \left(k^{Lin}_{s}(|y|)+\calJ[\tilde{g}_s](|y|)+C_2\right) = \calT\left[k^{Lin}_{s}(|y|)+ \calJ[\tilde{g}_s(|y|)]\right]
\end{equation*}
for every $s\in(0,s_*)$. But recalling the definition of the operator $\calT$ in \eqref{def: nonlinear operator}, we have
\begin{equation*}
    \calT\left[k^{Lin}_{s}(|y|)+ \calJ[\tilde{g}_s(|y|)]\right]=\calT\left[k^{Lin}_{s}(|y|)+ \calJ[\tilde{g}_s(|y|)]+C_2\right],
\end{equation*}
so that the function $k(|y|)$ defined by
\begin{equation*}\label{k=klin+g_s}
    k_s(|y|):=k^{Lin}_{s}(|y|)+\calJ[\tilde{g}_s](|y|)+C_2 \qquad (0<s<s_*)
\end{equation*}
is a solution to our original equation \eqref{eq: reformulation}.

\subsection{The forcing term $\Phi$}\label{sec: forcing term}
Here, we aim to prove the following Proposition \ref{proposition: key}, whcih leads to Lemma \ref{lem: l^2 estimate for T(klin)}.
\begin{proposition}\label{proposition: key}
    Let $t_1\in (3/2,2)$ be given, and let $\calT$ and $k^{Lin}_{s}$ be the operator and function defined in \eqref{def: nonlinear operator} and \eqref{klin}, respectively. Then for any  $t_1\in (3/2,2)$, there exists $t^{**}=t^{**}(t_1)\in(2,t_1+1)$ such that
    \begin{equation}\label{key estimate-1}
\nrmb{\calT\left[f(\cdot),k^{Lin}_{s}(|\cdot|)\right]}_{\dot{H}^{t_1}}\lesssim s^2\nrmb{f}_{\dot{H}^{t^{**}}\cap \dot{H}^{t_1+1}}.
    \end{equation}
\end{proposition}
\begin{remark}\label{rmk: proposition-1}
    Inserting $f(\cdot)=k^{Lin}_{s}(|\cdot|)$ in \eqref{key estimate-1}, recalling \eqref{klin fourier} and using $t^{**}>2$, we can prove Lemma \ref{lem: l^2 estimate for T(klin)}.
\end{remark}

As the first step toward Proposition \ref{proposition: key}, we reformulate $\Lambda^{t_1}\calT[f_1,f_2]$:
\begin{lemma}\label{lem: reformulation}
    For any  $t_1\in (3/2,2)$, there holds
    \begin{equation*}
        \Lambda^{t_1}\calT [f_1,f_2] 
        = \sum_{j=1}^7 T^{t_1,j}[f_1,f_2],
    \end{equation*}
    where $T^{t_1,j}[f_1,f_2]$ $(1\le j \le 7)$ are defined by 
    \begin{equation}\label{eq: reformulation- definitions}
    \begin{split}
    &T^{t_1,1}[f_1,f_2]=-\frac{1}{4\pi}\int_{\bbR^2}\left( \Delta_\alp \Lambda^{t_1} f_1+\Delta_{-\alp} \Lambda^{t_1} f_1\right)\left(\left(1+ \left(\frac{\alp}{|\alp|}\cdot \nb_y f_2\right)^2\right)^{-\frac32}-1\right)\frac{d\alp}{|\alp|^2}\\
        &T^{t_1,2}[f_1,f_2]= \frac{1}{2\pi}\int_{B_1(0)}\alp \cdot \nabla_y \Lambda^{t_1} f_1 \left(\left(1+ (\Delta_\alp f_2)^2\right)^{-\frac32}-\left(1+ \left(\frac{\alp}{|\alp|}\cdot \nb_y  f_2\right)^2\right)^{-\frac32} \right)\frac{d\alp}{|\alp|^3},\\ 
        &T^{t_1,3}[f_1,f_2]= \frac{1}{4\pi}\int_{\bbR^2\backslash B_1(0)}\alp \cdot \nabla_y \Lambda^{t_1} f_1 \left(\left(1+ (\Delta_\alp f_2)^2\right)^{-\frac32}- \left(1+ (\Delta_{-\alp} f_2)^2\right)^{-\frac32} \right)\frac{d\alp}{|\alp|^3}, \\
        &T^{t_1,4}[f_1,f_2]= -\frac{1}{2\pi}\int_{\bbR^2} \Delta_\alp \Lambda^{t_1} f_1\left(\left(1+ (\Delta_\alp f_2)^2\right)^{-\frac32}-\left(1+ \left(\frac{\alp}{|\alp|}\cdot \nb_y  f_2\right)^2\right)^{-\frac32}\right)\frac{d\alp}{|\alp|^2},\\ 
        &T^{t_1,5}[f_1,f_2]= \frac{3}{2\pi}\int_{\bbR^2}\frac{\Delta_\alp \Lambda^{t_1} f_1\;\Delta_\alp f_2}{\left(1+ (\Delta_\alp f_2)^2\right)^{\frac52}}\left(\Delta_\alp f_2-\frac{\alp}{|\alp|}\cdot \nb_y  f_2 + \alp \cdot \nb_y \Delta_\alp f_2\right)\frac{d\alp}{|\alp|^2}, \\
        &T^{t_1,6}[f_1,f_2]= \frac{1}{2\pi}\int_{\bbR^2} \alp \cdot \nabla_y \Delta_\alp f_1\;\Lambda^{t_1}\left(\left(1+ \left(\Delta_\alp  f_2\right)^2\right)^{-\frac32} \right)\frac{d\alp}{|\alp|^2},\\
        &T^{t_1,7}[f_1,f_2] =\frac{1}{2\pi}\int_{\bbR^2} \Lambda^{t_1}\left( \alp \cdot \nabla_y \Delta_\alp  f_1\left(\left(1+ \left(\Delta_\alp  f_2\right)^2\right)^{-\frac32}-1 \right)\right) \\
        &\qquad\qquad\qquad\qquad -\alp \cdot \nabla_y \Delta_\alp \Lambda^{t_1} f_1\left(\left(1+ \left(\Delta_\alp  f_2\right)^2\right)^{-\frac32}-1 \right) - \alp \cdot \nabla_y \Delta_\alp f_1\,\Lambda^{t_1}\left(\left(1+ \left(\Delta_\alp  f_2\right)^2\right)^{-\frac32} \right)\frac{d\alp}{|\alp|^2}.
    \end{split}
    \end{equation}
\end{lemma}
\begin{proof}
Noticing $\Lambda^{t_1}\nabla_y=\nabla_y \Lambda^{t_1}$, we have
\begin{equation*}
    \Lambda^{t_1}\calT [f_1,f_2]=\frac{1}{2\pi}\int_{\bbR^2} \alp \cdot \nabla_y \Delta_\alp \Lambda^{t_1} f_1\left(\left(1+ \left(\Delta_\alp  f_2\right)^2\right)^{-\frac32}-1 \right)\frac{d\alp}{|\alp|^2} +\sum_{j=6}^7 T^{t_1,j}[f_1,f_2].
\end{equation*}
We compute
\begin{equation*}
    \begin{split}
        &\frac{1}{2\pi}\int_{\bbR^2} \alp \cdot \nabla_y \Delta_\alp \Lambda^{t_1} f_1\left(\left(1+ \left(\Delta_\alp  f_2\right)^2\right)^{-\frac32}-1 \right)\frac{d\alp}{|\alp|^2} \\
        &\quad= \frac{1}{2\pi}\int_{\bbR^2}\alp \cdot \nabla_y \Delta_\alp \Lambda^{t_1} f_1 \left(\left(1+ \left(\frac{\alp}{|\alp|}\cdot \nb_y  f_2\right)^2\right)^{-\frac32}-1\right)\frac{d\alp}{|\alp|^2}\\
        &\qquad+\frac{1}{2\pi}\int_{\bbR^2}\alp \cdot \nabla_y \Delta_\alp \Lambda^{t_1} f_1 \left(\left(1+ (\Delta_\alp f_2)^2\right)^{-\frac32}-\left(1+ \left(\frac{\alp}{|\alp|}\cdot \nb_y  f_2\right)^2\right)^{-\frac32}\right)\frac{d\alp}{|\alp|^2}\\
        &\quad= \mrI + \mrII.
    \end{split}
\end{equation*}
Before reformulating $\mrI$ and $\mrII$, we note
\begin{equation}\label{change y to alpha derivative}
    \nb_y \left(\Lambda^{t_1} f_1(y-\alp)\right)= -\nb_\alp \left(\Lambda^{t_1}f_1(y-\alp)\right)=  \nb_\alp \left(\Lambda^{t_1} f_1(y)-\Lambda^{t_1} f_1(y-\alp)\right),
\end{equation}
\begin{equation}\label{derivatives of alpha functions-1}
        \nb_\alp \cdot \left( \frac{\alp}{|\alp|^3}\right)= -\frac{1}{|\alp|^3},
\end{equation}
\begin{equation}\label{derivatives of alpha functions-2}
    \alp \cdot \nb_\alp  \left( \left(1+ \left(\frac{\alp}{|\alp|}\cdot \nb_y f_2\right)^2\right)^{-\frac32}\right) =0,
\end{equation}
\begin{equation}\label{derivatives of alpha functions-3}
    \begin{split}
    \alp \cdot \nb_\alp \left(\left(1+ (\Delta_\alp f_2)^2\right)^{-\frac32} \right)&=
    \frac{3\Delta_\alp f_2}{\left(1+ (\Delta_\alp f_2)^2\right)^{\frac52}}\left(\Delta_\alp f_2-\frac{\alp\cdot \nb_y f_2(|y-\alp|)}{|\alp|} \right) \\
    &=\frac{3\Delta_\alp f_2}{\left(1+ (\Delta_\alp f_2)^2\right)^{\frac52}}\left(\Delta_\alp f_2-\frac{\alp \cdot \nb_y f_2}{|\alp|} + \alp \cdot \nb_y \Delta_\alp f_2\right).
    \end{split}
\end{equation}
For $\mrI$, noticing
\begin{equation*}
        \int_{\bbR^2}\alp \cdot \nabla_y \Lambda^{t_1} f_1\frac{d\alp}{|\alp|^3}=\int_{\bbR^2}\frac{\alp \cdot \nabla_y  \Lambda^{t_1} f_1}{\left(1+ \left(\frac{\alp}{|\alp|}\cdot \nb_y f_2\right)^2\right)^{\frac32}} \frac{d\alp}{|\alp|^3}=0
\end{equation*}
by symmetry,  we have
\begin{equation*}
\begin{split}
        \mrI &= -\frac{1}{2\pi}\int_{\bbR^2}\alp \cdot \nabla_y \Lambda^{t_1} f_1(y-\alp)\left(\left(1+ \left(\frac{\alp}{|\alp|}\cdot \nb_y f_2\right)^2\right)^{-\frac32}-1\right)\frac{d\alp}{|\alp|^3} \\
        &=-\frac{1}{2\pi}\int_{\bbR^2}\alp \cdot \nabla_\alp \delta_\alp \Lambda^{t_1} f_1\left(\left(1+ \left(\frac{\alp}{|\alp|}\cdot \nb_y f_2\right)^2\right)^{-\frac32}-1\right)\frac{d\alp}{|\alp|^3},
        \end{split}
\end{equation*}
where in the second equality we used \eqref{change y to alpha derivative}. We integrate by parts and employ \eqref{derivatives of alpha functions-1}, \eqref{derivatives of alpha functions-2} to obtain
\begin{equation*}
    \begin{split}
        \mrI &=  \frac{1}{2\pi}\int_{\bbR^2}\dlt_\alp\Lambda^{t_1} f_1 \,\nb_\alp \cdot \left(\frac{\alp}{|\alp|^3} \left(\left(1+ \left(\frac{\alp}{|\alp|}\cdot \nb_y f_2\right)^2\right)^{-\frac32} -1 \right)\right)d\alp \\
        &=- \frac{1}{2\pi}\int_{\bbR^2}\dlt_\alp\Lambda^{t_1} f_1 \left(\left(1+ \left(\frac{\alp}{|\alp|}\cdot \nb_y f_2\right)^2\right)^{-\frac32} -1 \right)\frac{d\alp}{|\alp|^3}.
    \end{split}
\end{equation*}
Making a change of variables $\alp \mapsto -\alp$, we observe
\begin{equation*}
\begin{split}
    \int_{\bbR^2}&\dlt_\alp\Lambda^{t_1} f_1  \left(\left(1+ \left(\frac{\alp}{|\alp|}\cdot \nb_y f_2\right)^2\right)^{-\frac32} -1 \right)\frac{d\alp}{|\alp|^3}=\int_{\bbR^2}\dlt_{-\alp}\Lambda^{t_1} f_1  \left(\left(1+ \left(\frac{\alp}{|\alp|}\cdot \nb_y f_2\right)^2\right)^{-\frac32} -1 \right)\frac{d\alp}{|\alp|^3},
\end{split}
\end{equation*}
which implies $\mrI=T^{t_1,1}[f_1,f_2]$.
We decompose $\mrII$ into
\begin{equation*}
    \begin{split}
        \mrII &= T^{t_1,2}[f_1,f_2] + \mrII_1 + \mrII_2,
    \end{split}
\end{equation*}
where
\begin{equation*}
    \begin{split}
        \mrII_1 &= \frac{1}{2\pi}\int_{\bbR^2\backslash B_1(0)}\alp \cdot \nabla_y  \Lambda^{t_1} f_1\left(\left(1+ (\Delta_\alp f_2)^2\right)^{-\frac32}-\left(1+ \left(\frac{\alp}{|\alp|}\cdot \nb_y f_2\right)^2\right)^{-\frac32} \right)\frac{d\alp}{|\alp|^3},\\
        \mrII_2 &= -\frac{1}{2\pi}\int_{\bbR^2}\alp \cdot \nabla_y  \Lambda^{t_1} f_1(y-\alp) \left(\left(1+ (\Delta_\alp f_2)^2\right)^{-\frac32}-\left(1+ \left(\frac{\alp}{|\alp|}\cdot \nb_y f_2\right)^2\right)^{-\frac32} \right)\frac{d\alp}{|\alp|^3}.
    \end{split}
\end{equation*}
For $\mrII_1$, we observe
\begin{equation*}
        \int_{\bbR^2\backslash B_1(0)}\frac{\alp \cdot \nabla_y  \Lambda^{t_1} f_1}{\left(1+ \left(\frac{\alp}{|\alp|}\cdot \nb_y f_2\right)^2\right)^{\frac32}} \frac{d\alp}{|\alp|^3}=0
\end{equation*}
by symmetry, so that 
\begin{equation*}
        \mrII_1 = \frac{1}{2\pi}\int_{\bbR^2\backslash B_1(0)} \frac{\alp \cdot \nabla_y  \Lambda^{t_1} f_1}{\left(1+ (\Delta_\alp f_2)^2\right)^{\frac32}} \frac{d\alp}{|\alp|^3}.
\end{equation*}
Moreover, making a change of variables $\alp \mapsto -\alp$, we see
\begin{equation*}
    \int_{\bbR^2\backslash B_1(0)} \frac{\alp \cdot \nabla_y  \Lambda^{t_1} f_1}{\left(1+ (\Delta_\alp f_2)^2\right)^{\frac32}} \frac{d\alp}{|\alp|^3}=-\int_{\bbR^2\backslash B_1(0)} \frac{\alp \cdot \nabla_y  \Lambda^{t_1} f_1}{\left(1+ (\Delta_{-\alp} f_2)^2\right)^{\frac32}} \frac{d\alp}{|\alp|^3},
\end{equation*}
which yields $\mrII_1=T^{t_1,3}[f_1,f_2]$.
For $\mrII_2$, recalling \eqref{change y to alpha derivative}, we integrate by parts to obtain
\begin{equation*}
    \mrII_2=\frac{1}{2\pi}\int_{\bbR^2}\dlt_\alp\Lambda^{t_1} f_1 \,\nb_\alp \cdot \left(\frac{\alp}{|\alp|^3} \left(\left(1+ (\Delta_\alp f_2)^2\right)^{-\frac32}-\left(1+ \left(\frac{\alp}{|\alp|}\cdot \nb_y f_2\right)^2\right)^{-\frac32} \right)\right)d\alp.
\end{equation*}
Computing with the aids of \eqref{derivatives of alpha functions-1} - \eqref{derivatives of alpha functions-3}, we arrive at
\begin{equation*}
    \mrII_2 = T^{t_1,4}[f_1,f_2]+T^{t_1,5}[f_1,f_2].
\end{equation*}
\end{proof}

Now we estimate $L^2$-norm of each $T^{t_1,j}[f,k^{Lin}_{s}]$ $(1\le j \le 7)$ over Lemma \ref{lem: T^1(klin)} - Lemma \ref{lem: T^7(klin)}, which proves Proposition \ref{proposition: key}.
\begin{lemma}\label{lem: T^1(klin)}
    For any  $t_1\in (3/2,2)$, there holds
    \begin{equation*}
\nrmb{T^{t_1,1}[f(\cdot),k^{Lin}_{s}(|\cdot|)]}_{L^2} \lesssim s^2\nrmb{f}_{\dot{H}^{t_1+1}}.
    \end{equation*}
\end{lemma}
\begin{proof}
To begin with, we claim that
\begin{equation}\label{est: 1- mean nabla klin}
    \Abs{\left(1+ \left(\frac{\alp}{|\alp|}\cdot \nb_y  k^{Lin}_{s}(|y|)\right)^2\right)^{-\frac32}-1} \lesssim s^2.
\end{equation}
Applying the mean value theorem to the function $F(x)=(1+x^2)^\frac32$, we have $\tau\in(0,1)$ such that
    \begin{equation*}
            \left|\frac{1}{F\left(\frac{\alp}{|\alp|}\cdot \nb_y  k^{Lin}_{s}(|y|)\right)}-1 \right|
            =
            \frac{\left|1-F\left(\frac{\alp}{|\alp|}\cdot \nb_y  k^{Lin}_{s}(|y|)\right)\right|}{F\left(\frac{\alp}{|\alp|}\cdot \nb_y  k^{Lin}_{s}(|y|)\right)} =\frac{\left|F'\left( \tau\frac{\alp}{|\alp|}\cdot \nb_y  k^{Lin}_{s}(|y|)\right)\right|\left|\frac{\alp}{|\alp|}\cdot \nb_y  k^{Lin}_{s}(|y|)\right|}{F\left(\frac{\alp}{|\alp|}\cdot \nb_y  k^{Lin}_{s}(|y|)\right)}.
    \end{equation*}
Since $\left|\frac{F'(\tau a)}{F(a)}\right| \lesssim |a|$
for $\tau \in (0,1)$ and $a\in\bbR$, we see
\begin{equation*}
    \left|\frac{1}{F\left(\frac{\alp}{|\alp|}\cdot \nb_y  k^{Lin}_{s}(|y|)\right)}-1 \right| \lesssim \left|\frac{\alp}{|\alp|}\cdot \nb_y  k^{Lin}_{s}(|y|)\right|^2 \lesssim s^2,
\end{equation*}
where in the last inequality, we used \eqref{klin nabla}.

Next, we observe that the factor
$\left(\frac{1}{\left(1+ \left(\frac{\alp}{|\alp|}\cdot \nb_y  k^{Lin}_{s}(|y|)\right)^2\right)^{\frac32}}-1\right)$ is independent of the length $|\alp|$, so that \eqref{est: 1- mean nabla klin} gives for $\alp=r\sigma$
\begin{equation*}
    |T^{t_1,1}[f(y),k^{Lin}_{s}(|y|)]| \lesssim s^2\int_{\bbS^1} \Abs{\int_{0}^\infty \Delta_{r\sigma} \Lambda^{t_1} f(y)+\Delta_{-r\sigma} \Lambda^{t_1} f(y)\frac{dr}{r}} d\sigma.
\end{equation*}
Thus using the Minkowski's inequality, we have
\begin{equation*}
    \nrmb{T^{t_1,1}[f(y),k^{Lin}_{s}(|y|)]}_{L^2}\lesssim s^2\int_{\bbS^1}\nrmb{\int_{0}^\infty \Delta_{r\sigma} \Lambda^{t_1} f(y)+\Delta_{-r\sigma} \Lambda^{t_1} f (y)\frac{dr}{r}}_{L^2_y}  d\sigma.
\end{equation*}
Since the Fourier transform of $(\Delta_{\alp}+\Delta_{-\alp})\Lambda^{t_1} f$ is
\begin{equation*}\label{Fourier transform of (Delta alp + Delta -alp)f}
    \frac{2-e^{-i\alp\cdot\xi}-e^{i\alp\cdot\xi}}{|\alp|}|\xi|^{t_1}\widehat{f}(\xi)=\frac{2-2\cos(\alp\cdot \xi)}{|\alp|}|\xi|^{t_1}\widehat{f}(\xi),
\end{equation*}
the Plancheral's theorem provides us with
\begin{equation*}
\begin{split}
    \nrmb{T^{t_1,1}[f(y),k^{Lin}_{s}(|y|)]}_{L^2}&\lesssim s^2\int_{\bbS^1}\normb{\int_{0}^\infty  \left|2-2\cos\left(r\sigma\cdot \xi\right)\right||\xi|^{t_1}\widehat{f} (\xi)\frac{dr}{r^2}} d\sigma \\
    &=s^2\int_{\bbS^1}\normb{|\xi|^{t_1}\widehat{f} (\xi)\left(\int_{0}^\frac{1}{|\sigma \cdot \xi|} \left|2-2\cos\left(r\sigma\cdot \xi\right)\right|\frac{dr}{r^2}+\int_\frac{1}{|\sigma \cdot \xi|}^\infty \left|2-2\cos\left(r\sigma\cdot \xi\right)\right|\frac{dr}{r^2}\right)} d\sigma.
\end{split}
\end{equation*}
Using
\begin{equation*}\label{est: fourier cosin}
    \left|2-2\cos\left(r\sigma\cdot \xi\right)\right|\le \min\left\{|r\sigma\cdot \xi|^2,2\right\},
\end{equation*}
we have
\begin{equation*}
    \int_{0}^\frac{1}{|\sigma \cdot \xi|} \left|2-2\cos\left(r\sigma\cdot \xi\right)\right|\frac{dr}{r^2}+\int_\frac{1}{|\sigma \cdot \xi|}^\infty \left|2-2\cos\left(r\sigma\cdot \xi\right)\right|\frac{dr}{r^2} \lesssim \int_{0}^\frac{1}{|\sigma \cdot \xi|} |r\sigma\cdot \xi|^2\frac{dr}{r^2}+\int_\frac{1}{|\sigma \cdot \xi|}^\infty\frac{dr}{r^2} \approx |\sigma \cdot \xi|\lesssim |\xi|
\end{equation*}
for every $\sigma \in \bbS^1$,
so that
\begin{equation*}
\nrmb{T^{t_1,1}[f(y),k^{Lin}_{s}(|y|)]}_{L^2}\lesssim s^2\int_{\bbS^1}\normb{|\xi|^{t_1+1}\widehat{f} (\xi)} d\sigma \lesssim s^2\nrmb{f}_{\dot{H}^{t_1+1}}.
\end{equation*}
\end{proof}

\begin{lemma}
     For any  $t_1\in (3/2,2)$, there holds
    \begin{equation*}
\nrmb{T^{t_1,2}[f(\cdot),k^{Lin}_{s}(|\cdot|)]}_{L^2} \lesssim s^2\nrmb{f}_{\dot{H}^{t_1+1}}.
    \end{equation*}
\end{lemma}
\begin{proof}
    To begin with, we show that for $p\in(2,\infty)$,
    \begin{equation}\label{est: diff-1-1}
        \left|\left(1+ (\Delta_\alp k^{Lin}_{s}(|y|))^2\right)^{-\frac32}-\left(1+ \left(\frac{\alp}{|\alp|}\cdot \nb_y  k^{Lin}_{s}(|y|)\right)^2\right)^{-\frac32} \right| \lesssim s^2|\alp|^{1-\frac2p}.
    \end{equation}
    Applying the mean value theorem to the function $F(x)=(1+x^2)^\frac32$, we have $\tau\in(0,1)$ such that
    \begin{equation*}
        \begin{split}
            &\left|\frac{1}{F\left(\Delta_\alp k^{Lin}_{s}(|y|))\right)}-\frac{1}{F\left(\frac{\alp}{|\alp|}\cdot \nb_y  k^{Lin}_{s}(|y|)\right)} \right| \\
            &\qquad=
            \frac{\left|F\left(\frac{\alp}{|\alp|}\cdot \nb_y  k^{Lin}_{s}(|y|)\right)-F\left(\Delta_\alp k^{Lin}_{s}(|y|))\right)\right|}{F\left(\Delta_\alp k^{Lin}_{s}(|y|))\right)F\left(\frac{\alp}{|\alp|}\cdot \nb_y  k^{Lin}_{s}(|y|)\right)} \\
            &\qquad=\frac{\left|F'\left(\tau \left(\frac{\alp}{|\alp|}\cdot \nb_y  k^{Lin}_{s}(|y|)\right) +(1-\tau)\left(\Delta_\alp k^{Lin}_{s}(|y|)\right)\right)\right|\left|\frac{\alp}{|\alp|}\cdot \nb_y  k^{Lin}_{s}(|y|)-\Delta_\alp k^{Lin}_{s}(|y|) \right|}{F\left(\Delta_\alp k^{Lin}_{s}(|y|))\right)F\left(\frac{\alp}{|\alp|}\cdot \nb_y  k^{Lin}_{s}(|y|)\right)}.
        \end{split}
    \end{equation*}
    Since 
\begin{equation*}\label{est: F-1}
 \left|\frac{F'(\tau a + (1-\tau)b)}{F(a)F(b)}\right| \lesssim |a|+|b|
\end{equation*}
for $\tau \in (0,1)$ and $a,b\in\bbR$, we obtain
\begin{equation}\label{est: F-2}
\begin{split}
    &\left|\frac{1}{F\left(\Delta_\alp k^{Lin}_{s}(|y|))\right)}-\frac{1}{F\left(\frac{\alp}{|\alp|}\cdot \nb_y  k^{Lin}_{s}(|y|)\right)} \right| \\
    &\qquad\lesssim \left|\Abs{\partial_r k^{Lin}_{s}(|y|)}+\Abs{\Delta_\alp k^{Lin}_{s}(|y|)} \right|\left|\frac{\alp}{|\alp|}\cdot \nb_y  k^{Lin}_{s}(|y|)-\Delta_\alp k^{Lin}_{s}(|y|) \right|.
    \end{split}
\end{equation}
Hence  \eqref{est: diff-1-1} follows from \eqref{est: klin-1'}, \eqref{klin nabla}, and 
\begin{equation}\label{est: nb klin for s}
        \Abs{\Delta_\alp k^{Lin}_{s}(|y|)} \lesssim s,
\end{equation}
which is guaranteed by \eqref{fundamental thm of cal} and  \eqref{klin nabla}.
The estimation of $\nrmb{T^{t_1,2}[f(y),k^{Lin}_{s}(|y|)]}_{ L^2}$ follows from \eqref{eq: reformulation- definitions} and  \eqref{est: diff-1-1}: 
\begin{equation*}
    \nrmb{T^{t_1,2}[f(y),k^{Lin}_{s}(|y|)]}_{ L^2} \lesssim s^2 \nrmb{\int_{B_1(0)}\Abs{\nb_y \Lambda^{t_1} f(y)}\frac{d\alp}{|\alp|^{1+\frac2p}}}_{L^2} \lesssim  s^2 \nrmb{\nb \Lambda^{t_1} f}_{L^2}.
\end{equation*}
\end{proof}

\begin{lemma}
     For any  $t_1\in (3/2,2)$, there holds
    \begin{equation*}
\nrmb{T^{t_1,3}[f(\cdot),k^{Lin}_{s}(|\cdot|)]}_{L^2} \lesssim s^2\nrmb{f}_{\dot{H}^{t_1+1}}.
    \end{equation*}
\end{lemma}
\begin{proof}
Applying the mean value theorem to the function $F(x)=(1+x)^\frac32$ with $x>0$, we have $\tau\in(0,1)$ such that
    \begin{equation*}
        \begin{split}
        &\left|\left(1+ (\Delta_\alp k^{Lin}_{s}(|y|))^2\right)^{-\frac32}-\left(1+ (\Delta_{-\alp} k^{Lin}_{s}(|y|))^2\right)^{-\frac32} \right|\\
            &\qquad=
            \frac{\left|F\left((\Delta_{-\alp} k^{Lin}_{s}(|y|))^2\right)-F\left((\Delta_\alp k^{Lin}_{s}(|y|))^2\right)\right|}{F\left((\Delta_\alp k^{Lin}_{s}(|y|))^2\right)F\left((\Delta_{-\alp} k^{Lin}_{s}(|y|))^2\right)} \\
            &\qquad=\frac{\left|F'\left(\tau \left(\Delta_{-\alp} k^{Lin}_{s}(|y|)\right)^2 +(1-\tau)\left(\Delta_\alp k^{Lin}_{s}(|y|)\right)^2\right)\right|\left|(\Delta_{\alp} k^{Lin}_{s}(|y|))^2-(\Delta_{-\alp} k^{Lin}_{s}(|y|))^2 \right|}{F\left((\Delta_\alp k^{Lin}_{s}(|y|))^2\right)F\left((\Delta_{-\alp} k^{Lin}_{s}(|y|))^2\right)}\\
            &\qquad \lesssim \left|(\Delta_{\alp} k^{Lin}_{s}(|y|))^2-(\Delta_{-\alp} k^{Lin}_{s}(|y|))^2 \right|,
        \end{split}
    \end{equation*}
    where on the last line, we used
    \begin{equation}\label{new F est abtau}
        \frac{\Abs{F'(\tau a + (1-\tau)b)}}{F(a)F(b)}\lesssim 1
    \end{equation}
    for any $a,b>0$ and $\tau\in(0,1)$.
    Then \eqref{eq: reformulation- definitions}, \eqref{est: klin-1} gives us
    \begin{equation*}
\begin{split}
    &\nrmb{T^{t_1,3}[f(y),k^{Lin}_{s}(|y|)]}_{L^2} \\
    &\quad\lesssim s^2 \nrmb{\Abs{\nb_y \Lambda^{t_1} f(y)}\int_{\bbR^2\backslash B_1(0)}1_{\left\{|\alp|\le \frac{|y|}{2}\right\}}\left(\frac{|y||\alp|}{|y|^2+|\alp|^2+1}+\frac{|y|}{|\alp|\sqrt{|y|^2+|\alp|^2+1}}\right)\frac{d\alp}{|\alp|^2}}_{L^2}\\
    &\qquad+ s^2 \nrmb{\Abs{\nb_y \Lambda^{t_1} f(y)}\int_{\bbR^2\backslash B_1(0)}1_{\left\{|\alp|\ge \frac{|y|}{2}\right\}}\frac{|y|}{\sqrt{|y|^2+|\alp|^2+1}}\frac{d\alp}{|\alp|^2}}_{L^2}.
    \end{split}
\end{equation*}
Note that for any $y\in\bbR^2\backslash \{0\}$,
\begin{equation*}
    \int_{\left\{|\alp|\ge \frac{|y|}{2}\right\}}\frac{|y|}{\sqrt{|y|^2+|\alp|^2+1}}\frac{d\alp}{|\alp|^2}=\frac{|y|}{\sqrt{|y|^2+1}}\log \left(\frac{2\sqrt{|y|^2+1}}{|y|}+\sqrt{\left(\frac{2\sqrt{|y|^2+1}}{|y|}\right)^2+1}\right)\lesssim 1,
\end{equation*}
and moreover, 
\begin{equation*}
    \int_{\bbR^2\backslash B_1(0)}1_{\left\{|\alp|\le \frac{|y|}{2}\right\}}\left(\frac{|y||\alp|}{|y|^2+|\alp|^2+1}+\frac{|y|}{|\alp|\sqrt{|y|^2+|\alp|^2+1}}\right)\frac{d\alp}{|\alp|^2}=0
\end{equation*}
for $|y|\le2$ while
\begin{equation*}
\begin{split}
    &\int_{\bbR^2\backslash B_1(0)}1_{\left\{|\alp|\le \frac{|y|}{2}\right\}}\left(\frac{|y||\alp|}{|y|^2+|\alp|^2+1}+\frac{|y|}{|\alp|\sqrt{|y|^2+|\alp|^2+1}}\right)\frac{d\alp}{|\alp|^2}\\
&\quad=\frac{|y|\left(\sqrt{|y|^2+1}\left(\arctan\left(\frac{|y|}{2\sqrt{|y|^2+1}}\right)-\arctan\left(\frac{1}{\sqrt{|y|^2+1}}\right)\right)+\sqrt{|y|^2+2}\right)-\sqrt{5|y|^2+4}}{|y|^2+1}\lesssim 1
    \end{split}
\end{equation*}
for $|y|>2$. Hence we arrive at
\begin{equation*}
    \nrmb{T^{t_1,3}[f(y),k^{Lin}_{s}(|y|)]}_{L^2}\lesssim s^2 \nrmb{\nb \Lambda^{t_1} k^{Lin}_{s}(|\cdot|)}_{L^2}.
\end{equation*}
\end{proof}

\begin{lemma}\label{lem: T^4(klin)}
For any  $t_1\in (3/2,2)$, there holds
    \begin{equation*}
\nrmb{T^{t_1,4}[f(\cdot),k^{Lin}_{s}(|\cdot|)]}_{L^2} \lesssim s^2\nrmb{f}_{\dot{H}^{t_1+\frac34}}.
    \end{equation*}
\end{lemma}
\begin{remark}
    Since $t_1\in (3/2,2)$, we can check $t_1+3/4 \in (2,t_1+1)$.
\end{remark}
\begin{proof}
Recalling \eqref{est: F-2}, we have
\begin{equation*}
    \Abs{T^{t_1,4}[f(y),k^{Lin}_{s}(|y|)]}\lesssim s \int_{\bbR^2} \Abs{\Delta_\alp \Lambda^{t_1} f(y)} \left|\frac{\alp}{|\alp|}\cdot \nb_y  k^{Lin}_{s}(|y|)-\Delta_\alp k^{Lin}_{s}(|y|) \right| \frac{d\alp}{|\alp|^2}.
\end{equation*}
Noticing \eqref{sobolev embedding-2}, \eqref{eq: prelemma-1}, \eqref{eq: prelemma-2}, and \eqref{klin fourier}, we estimate
\begin{equation*}
\begin{split}
    &\normb{T^{t_1,4}[f(y),k^{Lin}_{s}(|y|)]}\\
    &\quad\lesssim s \int_{\bbR^2} \nrmb{ \Lambda^{t_1} \dlt_\alp f(y)}_{\dot{H}^\frac12_y} \nrmb{\frac{\alp}{|\alp|}\cdot \nb_y  k^{Lin}_{s}(|y|)-\Delta_\alp k^{Lin}_{s}(|y|)}_{\dot{H}^\frac12_y} \frac{d\alp}{|\alp|^3} \\
    &\quad\lesssim s \left(\int_{\bbR^2}\nrmb{ \Lambda^{t_1} \dlt_\alp f(y)}^2_{\dot{H}^\frac12} \frac{d\alp}{|\alp|^\frac52}\right)^\frac12 \left(\int_{\bbR^2}\nrmb{\frac{\alp}{|\alp|}\cdot \nb_y  k^{Lin}_{s}(|y|)-\Delta_\alp k^{Lin}_{s}(|y|)}^2_{\dot{H}^\frac12} \frac{d\alp}{|\alp|^\frac72}\right)^\frac12 \\
    &\quad \lesssim s \nrmb{\Lambda^{t_1} f}_{\dot{H}^\frac34}\nrmb{k^{Lin}_{s}}_{\dot{H}^\frac94} \lesssim s^2\nrmb{f}_{\dot{H}^{t_1+\frac34}}.
\end{split}
\end{equation*}
\end{proof}

\begin{lemma}\label{lem: T^5(klin)}
For any  $t_1\in (3/2,2)$, there holds
    \begin{equation*}
\nrmb{T^{t_1,5}[f(\cdot),k^{Lin}_{s}(|\cdot|)]}_{L^2} \lesssim s^2\nrmb{f}_{\dot{H}^{t_1+\frac34}}.
    \end{equation*}
\end{lemma}
\begin{proof}
Using \eqref{est: nb klin for s},
we have
    \begin{equation*}
    \begin{split}
        \Abs{T^{t_1,5}[f(y),k^{Lin}_{s}(|y|)]} &\lesssim s \int_{\bbR^2} \Abs{\Delta_\alp \Lambda^{t_1} f(y)} \left|\frac{\alp}{|\alp|}\cdot \nb_y  k^{Lin}_{s}(|y|)-\Delta_\alp k^{Lin}_{s}(|y|) \right| \frac{d\alp}{|\alp|^2} \\
        &\quad + s\int_{\bbR^2} \Abs{\Delta_\alp \Lambda^{t_1} f(y)}\left|\nb_y \dlt_\alp k^{Lin}_{s}(|y|)\right|\frac{d\alp}{|\alp|^2}\\
        &=\mrI+\mrII.
    \end{split}
    \end{equation*}
    $\normb{\mrI}$ can be estimated with the same argument as the proof of last lemma. For $\normb{\mrII}$, we again use \eqref{sobolev embedding-2}, \eqref{eq: prelemma-1}, and \eqref{klin fourier} to obtain
\begin{equation*}
\begin{split}
    \normb{\mrII}
    &\lesssim  s\int_{\bbR^2} \nrmb{ \Lambda^{t_1} \dlt_\alp f(y)}_{\dot{H}^\frac12_y} \nrmb{\nb_y \dlt_\alp k^{Lin}_{s}(|y|)}_{\dot{H}^\frac12_y} \frac{d\alp}{|\alp|^3} \\
    &\lesssim  s\left(\int_{\bbR^2}\nrmb{ \Lambda^{t_1} \dlt_\alp f(y)}^2_{\dot{H}^\frac12} \frac{d\alp}{|\alp|^\frac52}\right)^\frac12 \left(\int_{\bbR^2}\nrmb{\nb_y \dlt_\alp k^{Lin}_{s}(|y|)}^2_{\dot{H}^\frac12} \frac{d\alp}{|\alp|^\frac72}\right)^\frac12 \\
    &\lesssim  s\nrmb{\Lambda^{t_1} f}_{\dot{H}^\frac34}\nrmb{k^{Lin}_{s}}_{\dot{H}^\frac94} \lesssim s^2\nrmb{f}_{\dot{H}^{t_1+\frac34}}.
\end{split}
\end{equation*}
\end{proof}

\begin{lemma}
  For any  $t_1\in (3/2,2)$, there holds
    \begin{equation*}
\nrmb{T^{t_1,6}[f(\cdot),k^{Lin}_{s}(|\cdot|)]}_{L^2} \lesssim s^2\nrmb{f}_{\dot{H}^{\frac{t_1}{2}+\frac54}\cap\dot{H}^{t_1+\frac34}}.
    \end{equation*}
\end{lemma}
\begin{remark}
    Since $t_1\in (3/2,2)$, we can check
    $t_1/2+5/4 \in(2,t_1+1)$.
\end{remark}
\begin{proof}
To begin with, we show
\begin{equation}\label{est: laplacian of klin in denominator}
    \nrmb{\left(1+ \left(\Delta_\alp  k_s^{Lin}(|y|)\right)^2\right)^{-\frac32}}_{\dot{H}^2_y} \lesssim \frac{s}{|\alp|}\left(\nrmb{\dlt_\alp\nabla_y  k^{Lin}_{s}(|y|)}_{\dot{H}^{\frac12}_y} + \nrmb{\dlt_\alp\Delta k^{Lin}_{s}(|y|)}_{L^2_y}\right).
\end{equation}
Recalling \eqref{est: nb klin for s}, we have
\begin{equation*}
    \begin{split}
    \Abs{\Delta \left(\left(1+ \left(\Delta_\alp  k_s^{Lin}(|y|)\right)^2\right)^{-\frac32}\right)}&\lesssim \Abs{\frac{\left(\Delta_\alp k_s^{Lin}(|y|)\right)^2 \Delta_\alp \nabla_y k_s^{Lin}(|y|)\cdot \Delta_\alp \nabla_y k_s^{Lin}(|y|) }{\left(1+ \left(\Delta_\alp k_s^{Lin}(|y|)\right)^2\right)^{\frac72}}}\\
    &\quad+\Abs{\frac{\Delta_\alp \nabla_y k_s^{Lin}(|y|)\cdot \Delta_\alp \nabla_y k_s^{Lin}(|y|) }{\left(1+ \left(\Delta_\alp k_s^{Lin}(|y|)\right)^2\right)^{\frac52}}} +\Abs{\frac{\Delta_\alp k_s^{Lin}(|y|)\Delta_\alp \Delta k_s^{Lin}(|y|) }{\left(1+ \left(\Delta_\alp k_s^{Lin}(|y|)\right)^2\right)^{\frac52}}}\\
        &\lesssim \Abs{\Delta_\alp \nb_y k^{Lin}_{s}(|y|)}^2+ s\Abs{\Delta_\alp \Delta k^{Lin}_{s}(|y|)},
    \end{split}
\end{equation*}
so that
\begin{equation*}
    \nrmb{\left(1+ \left(\Delta_\alp  k_s^{Lin}(|y|)\right)^2\right)^{-\frac32}}_{\dot{H}^2_y} \lesssim \frac{1}{|\alp|^2}\nrmb{\dlt_\alp\nabla_y   k^{Lin}_{s}(|y|)}^2_{L^4_y} + \frac{s}{|\alp|}\nrmb{\dlt_\alp\Delta k^{Lin}_{s}(|y|)}_{L^2_y}.
\end{equation*}
Note that \eqref{fundamental thm of cal} and \eqref{klin nabla^2} imply
\begin{equation*}
    \frac{\nrmb{ \dlt_\alp\nabla_y  k^{Lin}_{s}(|y|)}_{L^4_y}}{|\alp|}\lesssim \nrmb{\nb^2 k^{Lin}_{s}(|y|)}_{L^4_y} \lesssim s,
\end{equation*}
so that this together with the Sobolev embedding: $\dot{H}^{\frac12}(\bbR^2)\hookrightarrow L^4(\bbR^2)$ give us \eqref{est: laplacian of klin in denominator}.
Next, employing \eqref{est: laplacian of klin in denominator},  \eqref{sobolev embedding-2}, \eqref{eq: prelemma-1}, and \eqref{klin fourier}, we have
\begin{equation*}
\begin{split}
    &\normb{T^{t_1,6}[f(y),k^{Lin}_{s}(|y|)]}\\
    &\quad\lesssim  \int_{\bbR^2} \nrmb{ \nabla_y  \dlt_\alp f(y)}_{\dot{H}^{t_1-1}_y} \nrmb{\Lambda^{t_1}\left(\left(1+ \left(\Delta_\alp  k_s^{Lin}(|y|)\right)^2\right)^{-\frac32}\right)}_{\dot{H}^{2-t_1}_y} \frac{d\alp}{|\alp|^2} \\
    &\quad\lesssim  s\int_{\bbR^2} \nrmb{ \nabla_y   \dlt_\alp f(y)}_{\dot{H}^{t_1-1}} \nrmb{\nabla_y  \dlt_\alp k^{Lin}_{s}(|y|)}_{\dot{H}^\frac12} \frac{d\alp}{|\alp|^3} +s\int_{\bbR^2} \nrmb{ \nabla_y  \dlt_\alp f(y)}_{\dot{H}^{t_1-1}} \normb{\Delta \dlt_\alp  k^{Lin}_{s}(|y|)} \frac{d\alp}{|\alp|^3}.
\end{split}
\end{equation*}
Since we assumed $t_1\in (3/2,2)$, we have a positive constant $\eps_1:=t_1-\frac32>0$. Using \eqref{eq: prelemma-1} and \eqref{klin fourier}, we obtain
\begin{equation*}
\begin{split}
    &\normb{T^{t_1,6}[f(y),k^{Lin}_{s}(|y|)]}\\
    &\quad \lesssim s\left(\int_{\bbR^2}\nrmb{ \nabla_y \dlt_\alp f(y)}^2_{\dot{H}^{t_1-1}} \frac{d\alp}{|\alp|^{3-\eps_1}}\right)^\frac12 \left(\int_{\bbR^2}\nrmb{\nb_y \dlt_\alp  k^{Lin}_{s}(|y|)}_{\dot{H}^\frac12} \frac{d\alp}{|\alp|^{3+\eps_1}}\right)^\frac12\\
    &\qquad+  s\left(\int_{\bbR^2}\nrmb{ \nabla_y \dlt_\alp  f(y)}^2_{\dot{H}^{t_1-1}} \frac{d\alp}{|\alp|^\frac72}\right)^\frac12 \left(\int_{\bbR^2}\normb{\Delta \dlt_\alp   k^{Lin}_{s}(|y|)} \frac{d\alp}{|\alp|^\frac52}\right)^\frac12 \\
    &\quad \lesssim s\left(\nrmb{f}_{\dot{H}^{t_1+\frac{1-\eps_1}{2}}}\nrmb{k^{Lin}_{s}}_{\dot{H}^{2+\frac{\eps_1}{2}}}+ \nrmb{f}_{\dot{H}^{t_1+\frac{3}{4}}}\nrmb{k^{Lin}_{s}}_{\dot{H}^\frac94}\right) \lesssim s^2\left(\nrmb{f}_{\dot{H}^{\frac{t_1}{2}+\frac54}}+ \nrmb{f}_{\dot{H}^{t_1+\frac{3}{4}}}\right).
\end{split}
\end{equation*}

\end{proof}

\begin{lemma}\label{lem: T^7(klin)}
  For any  $t_1\in (3/2,2)$, there holds
    \begin{equation*}
\nrmb{T^{t_1,7}[f(\cdot),k^{Lin}_{s}(|\cdot|)]}_{L^2} \lesssim s^2 \nrmb{f}_{\dot{H}^{\frac{t_1+3}{2}}}
    \end{equation*}
\end{lemma}
\begin{remark}
    Since $t_1\in (3/2,2)$, we can check
    $t_1/2+3/2 \in(2,t_1+1)$.
\end{remark}
\begin{proof}
We observe for $\bt\in\bbR^2$,
\begin{equation}\label{alp bt est-1}
    \Abs{\dlt_{\beta} \left(\left(1+(\Delta_\alp k^{Lin}_{s}(|y|))^2\right)^{-\frac32}\right)}\lesssim s\Abs{\Dlt_\alp \dlt_\beta k^{Lin}(|y|)}.
\end{equation}
Indeed, applying the mean value theorem to the function $F(x)=(1+x)^\frac32$ with $x>0$, we have $\tau\in(0,1)$ such that
    \begin{equation*}
        \begin{split}
        &\Abs{\dlt_{\beta} \left(\left(1+(\Delta_\alp k^{Lin}_{s}(|y|))^2\right)^{-\frac32}\right)}\\
        &\qquad=\left|\left(1+ (\Delta_\alp k^{Lin}_{s}(|y|))^2\right)^{-\frac32}-\left(1+ (\Delta_{\alp} k^{Lin}_{s}(|y-\bt|))^2\right)^{-\frac32} \right|\\
            &\qquad=
            \frac{\left|F\left((\Delta_{\alp} k^{Lin}_{s}(|y|))^2\right)-F\left((\Delta_\alp k^{Lin}_{s}(|y-\bt|))^2\right)\right|}{F\left((\Delta_\alp k^{Lin}_{s}(|y|))^2\right)F\left((\Delta_{\alp} k^{Lin}_{s}(|y-\bt|))^2\right)} \\
            &\qquad=\frac{\left|F'\left(\tau \left(\Delta_{\alp} k^{Lin}_{s}(|y|)\right)^2 +(1-\tau)\left(\Delta_\alp k^{Lin}_{s}(|y-\bt|)\right)^2\right)\right|\left|(\Delta_{\alp} k^{Lin}_{s}(|y|))^2-(\Delta_{\alp} k^{Lin}_{s}(|y-\bt|))^2 \right|}{F\left((\Delta_\alp k^{Lin}_{s}(|y|))^2\right)F\left((\Delta_{\alp} k^{Lin}_{s}(|y-\bt|))^2\right)}\\
            &\qquad \lesssim \left|(\Delta_{\alp} k^{Lin}_{s}(|y|))^2-(\Delta_{\alp} k^{Lin}_{s}(|y-\bt|))^2 \right|\lesssim s\left|\Delta_{\alp} k^{Lin}_{s}(|y|)-\Delta_{\alp} k^{Lin}_{s}(|y-\bt|) \right|,
        \end{split}
    \end{equation*}
where we employed \eqref{new F est abtau} and \eqref{est: nb klin for s}.
Now using \eqref{commutator identity for fractional laplacian} and \eqref{alp bt est-1}, we have
\begin{equation*}
    \begin{split}
        \Abs{T^{t_1,7}[f(y),k^{Lin}_{s}(|y|)]} &\lesssim \int_{\bbR^2} \int_{\bbR^2} \Abs{\alp \cdot \nabla_y \Delta_\alp \dlt_\bt f(y)}\Abs{\dlt_{\beta} \left(\left(1+(\Delta_\alp k^{Lin}_{s}(|y|))^2\right)^{-\frac32}\right)}\frac{d\bt}{|\bt|^{2+t_1}}\frac{d\alp}{|\alp|^{2}} \\
        &\lesssim s \int_{\bbR^2} \int_{\bbR^2} \Abs{\nabla_y \delta_\alp \dlt_\bt f(y)}\Abs{\Dlt_\alp \dlt_\beta k^{Lin}(|y|)}\frac{d\bt}{|\bt|^{2+t_1}}\frac{d\alp}{|\alp|^{2}}.
    \end{split}
\end{equation*}
Thus using \eqref{sobolev embedding-2}, \eqref{eq: prelemma-1}, and \eqref{klin fourier}, we estimate
\begin{equation*}
    \begin{split}
       &\nrmb{T^{t_1,7}[f(y),k^{Lin}_{s}(|y|)]}_{L^2}  \\
       &\quad\lesssim  s\int_{\bbR^2}\int_{\bbR^2} \nrmb{\nabla_y \dlt_\alp \dlt_\bt f(y)}_{\dot{H}^\frac14} \nrmb{\dlt_\alp \dlt_\beta k^{Lin}_{s}(|y|)}_{\dot{H}^\frac34}  \frac{d\bt}{|\bt|^{2+t_1}}\frac{d\alp}{|\alp|^{3}} \\
       &\quad\lesssim s\left( \int_{\bbR^2}\int_{\bbR^2} \nrmb{\dlt_\alp \dlt_\bt f(y)}^2_{\dot{H}^\frac54}   \frac{d\bt}{|\bt|^{2+t_1}}\frac{d\alp}{|\alp|^{\frac{5}{2}}}\right)^\frac12
       \left(\int_{\bbR^2}\int_{\bbR^2} \nrmb{\dlt_\alp \dlt_\beta k^{Lin}_{s}(|y|)}^2_{\dot{H}^\frac34}  \frac{d\bt}{|\bt|^{2+t_1}}\frac{d\alp}{|\alp|^{\frac{7}{2}}}\right)^\frac12 \\
       &\quad\lesssim s \nrmb{f}_{\dot{H}^{\frac{t_1+3}{2}}}\nrmb{k^{Lin}_{s}}_{\dot{H}^{\frac{t_1+3}{2}}}\lesssim s^2 \nrmb{f}_{\dot{H}^{\frac{t_1+3}{2}}},
    \end{split}
\end{equation*}
where for the last inequality, we used $\frac{t_1+3}{2}>2$.
\end{proof}

\subsection{The linear operator $\calA[g]$}\label{sec: linear term}
In this subsection, we prove Lemma \ref{lem: l^2 estimate for T_1(g)}.  We recall \eqref{def: nonlinear operator} and \eqref{def: subTs} to compute
\begin{equation*}
\calT_1[g(|y|)]=\calT\left[g(|y|),k^{Lin}_{s}(|y|)\right]-\frac{3}{2\pi} \calQ[k^{Lin}_{s}(|y|),g(|y|),k^{Lin}_{s}(|y|),k^{Lin}_{s}(|y|)].
\end{equation*}
where $\calQ$ is the operator defined by
\begin{equation}\label{def: calQ}
       \calQ[f_1,f_2,f_3,f_4]=\int_{\bbR^2} \alp \cdot \nb_y \Dlt_\alp f_1 \frac{\Dlt_\alp f_2 \Dlt_\alp f_3 }{\left(1+\left(\Dlt_\alp f_4\right)^2\right)^\frac52} \frac{d\alp}{|\alp|^2}.
\end{equation}
Given $t_1\in (3/2,2)$, Proposition \ref{proposition: key} ensures the existence of $t^{**}\in(2,t_1+1)$ satisfying
\begin{equation}\label{est: calT_1-1}
\normb{\Lambda^{t_1}\calT\left[g(|y|),k^{Lin}_{s}(|y|)\right]}\lesssim s^2 \nrmb{g(|\cdot|)}_{\dot{H}^{t^{**}}\cap\dot{H}^{t_1+1}}.
\end{equation}
Hence, it suffices to show the following proposition:
\begin{proposition}\label{proposition: key for linear}
     For any  $t_1\in (3/2,2)$, there exist $t^{*}=t^*(t_1)\in(t_1,2)$ and $t^{**}=t^{**}(t_1)\in(2,t_1+1)$ such that
    \begin{equation}\label{key estimate-linear}
    \begin{split}
&\nrmb{ \calQ[f_1,f_2,f_3,f_4]}_{\dot{H}^{t_1}}\\
&\quad\lesssim \nrmb{f_1}_{\dot{H}^{t^{**}}\cap\dot{H}^{t_1+1}}\nrmb{f_2}_{\dot{W}^{1,\infty}\cap\dot{H}^{t^{*}}\cap\dot{H}^{t_1+1}}\nrmb{f_3}_{\dot{W}^{1,\infty}\cap\dot{H}^{t^{**}}\cap\dot{H}^{t_1+1}}\left( 1 +\nrmb{f_4}_{\dot{H}^{t^{**}}\cap\dot{H}^{t_1+1}}\right)
\end{split}
    \end{equation}
\end{proposition}
\begin{remark}\label{rmk: proposition-2}
    Inserting $f_1=f_3=f_4=k^{Lin}_{s}(|\cdot|)$ and $f_2=g(|\cdot|)$ in \eqref{key estimate-linear}, recalling \eqref{sobolev embedding-3(infty)} and \eqref{klin fourier}, and using $t^{**}>2$, we can prove 
    \begin{equation*}
        \nrmb{ \calQ[k^{Lin}_{s}(|y|),g(|y|),k^{Lin}_{s}(|y|),k^{Lin}_{s}(|y|)]}_{\dot{H}^{t_1}}\lesssim s^2 \nrmb{g(|\cdot|)}_{\dot{H}^{t^{**}}\cap\dot{H}^{t_1+1}}.
    \end{equation*}
    Hence combining this with \eqref{est: calT_1-1} and noticing \eqref{sobolev interpolation}, we arrive at Lemma \ref{lem: l^2 estimate for T_1(g)}.
\end{remark}
As the first step toward Proposition \ref{proposition: key for linear}, we reformulate $\Lmb^{t_1}\calQ$:
\begin{lemma}\label{lem: reformulation for linear A}
For any $t_1\in (3/2,2)$, there holds
\begin{equation*}
    \Lmb^{t_1}\calQ[f_1,f_2,f_3,f_4]=\sum_{j=1}^7 Q^{t_1,j}[f_1,f_2,f_3,f_4],
\end{equation*}
where
\begin{equation*}\label{splitting of Q_t_1}
    \begin{split}
        &Q^{t_1,1}[f_1,f_2,f_3,f_4]=-\frac{1}{2}\int_{\bbR^2}\left( \Delta_\alp \Lambda^{t_1} f_1+\Delta_{-\alp} \Lambda^{t_1} f_1\right)\frac{\frac{\alp}{|\alp|}\cdot \nb_y  f_2 \frac{\alp}{|\alp|}\cdot \nb_y  f_3}{\left(1+\left(\frac{\alp}{|\alp|}\cdot \nb_y f_4\right)^2\right)^\frac52}\frac{d\alp}{|\alp|^2},\\
        &Q^{t_1,2}[f_1,f_2,f_3,f_4]= \int_{B_1(0)}\alp \cdot \nabla_y \Lambda^{t_1} f_1 \left(\frac{\Dlt_\alp f_2 \Dlt_\alp f_3}{\left(1+\left(\Dlt_\alp f_4\right)^2\right)^\frac52} -\frac{\frac{\alp}{|\alp|}\cdot \nb_y  f_2 \frac{\alp}{|\alp|}\cdot \nb_y  f_3}{\left(1+\left(\frac{\alp}{|\alp|}\cdot \nb_y f_4\right)^2\right)^\frac52}\right)\frac{d\alp}{|\alp|^3},\\ 
        &Q^{t_1,3}[f_1,f_2,f_3,f_4]= \int_{\bbR^2\backslash B_1(0)}\alp \cdot \nabla_y \Lambda^{t_1} f_1 \frac{\Dlt_\alp f_2 \Dlt_\alp f_3}{\left(1+\left(\Dlt_\alp f_4\right)^2\right)^\frac52}\frac{d\alp}{|\alp|^3}, \\
        &Q^{t_1,4}[f_1,f_2,f_3,f_4]= -\int_{\bbR^2} \Delta_\alp \Lambda^{t_1} f_1\left(\frac{\Dlt_\alp f_2 \Dlt_\alp f_3}{\left(1+\left(\Dlt_\alp f_4\right)^2\right)^\frac52}-\frac{\frac{\alp}{|\alp|}\cdot \nb_y  f_2 \frac{\alp}{|\alp|}\cdot \nb_y  f_3}{\left(1+\left(\frac{\alp}{|\alp|}\cdot \nb_y f_4\right)^2\right)^\frac52}\right)\frac{d\alp}{|\alp|^2},\\ 
        &Q^{t_1,5}[f_1,f_2,f_3,f_4]= \int_{\bbR^2}\Dlt_\alp \Lmb^{t_1}f_1\alp \cdot \nb_\alp \left(\frac{\Dlt_\alp f_2 \Dlt_\alp f_3}{\left(1+\left(\Dlt_\alp f_4\right)^2\right)^\frac52} \right)\frac{d\alp}{|\alp|^2}.\\
        &Q^{t_1,6}[f_1,f_2,f_3,f_4]=\int_{\bbR^2} \alp \cdot \nb_y  \Dlt_\alp  f_1 \Lmb^{t_1}\left(\frac{\Dlt_\alp f_2 \Dlt_\alp f_3}{\left(1+\left(\Dlt_\alp f_4\right)^2\right)^\frac52} \right)\frac{d\alp}{|\alp|^2},\\
        &Q^{t_1,7}[f_1,f_2,f_3,f_4]=\int_{\bbR^2} \Lmb^{t_1}\left(\alp \cdot \nb_y \Dlt_\alp f_1 \frac{\Dlt_\alp f_2 \Dlt_\alp f_3}{\left(1+\left(\Dlt_\alp f_4\right)^2\right)^\frac52}\right) \\
        &\quad\qquad\quad\qquad\quad\qquad- \alp \cdot \nb_y \Dlt_\alp \Lambda^{t_1} f_1 \frac{\Dlt_\alp f_2 \Dlt_\alp f_3}{\left(1+\left(\Dlt_\alp f_4\right)^2\right)^\frac52} -\alp \cdot \nb_y  \Dlt_\alp  f_1 \Lmb^{t_1}\left(\frac{\Dlt_\alp f_2 \Dlt_\alp f_3}{\left(1+\left(\Dlt_\alp f_4\right)^2\right)^\frac52} \right)\frac{d\alp}{|\alp|^2}.
    \end{split}
\end{equation*}
\end{lemma}
\begin{proof}
    Before starting the proof, we note
    \begin{equation}\label{derivatives of alpha functions-for linear}
        \alp \cdot \nb_\alp  \left( \frac{\frac{\alp}{|\alp|}\cdot \nb_y  f_2 \frac{\alp}{|\alp|}\cdot \nb_y  f_3}{\left(1+\left(\frac{\alp}{|\alp|}\cdot \nb_y f_4\right)^2\right)^\frac52}\right) =0.
\end{equation}
    We split
\begin{equation*}
    \begin{split}
        \Lmb^{t_1}\calQ[f_1,f_2,f_3,f_4]&=\int_{\bbR^2} \alp \cdot \nabla_y \Delta_\alp \Lambda^{t_1} f_1\frac{\frac{\alp}{|\alp|}\cdot \nb_y  f_2 \frac{\alp}{|\alp|}\cdot \nb_y  f_3}{\left(1+\left(\frac{\alp}{|\alp|}\cdot \nb_y f_4\right)^2\right)^\frac52}\frac{d\alp}{|\alp|^2} \\
        &\quad+ \int_{\bbR^2}\alp \cdot \nabla_y \Dlt_\alp\Lambda^{t_1} f_1 \left(\frac{\Dlt_\alp f_2 \Dlt_\alp f_3}{\left(1+\left(\Dlt_\alp f_4\right)^2\right)^\frac52} -\frac{\frac{\alp}{|\alp|}\cdot \nb_y  f_2 \frac{\alp}{|\alp|}\cdot \nb_y  f_3}{\left(1+\left(\frac{\alp}{|\alp|}\cdot \nb_y f_4\right)^2\right)^\frac52}\right)\frac{d\alp}{|\alp|^3}\\
        &\quad+\sum_{j=6}^7 Q^{t_1,j}[f_1,f_2,f_3,f_4]\\
        &=\mrI+\mrII+\sum_{j=6}^7Q^{t_1,j}[f_1,f_2,f_3,f_4].
    \end{split}
\end{equation*}
For $\mrI$,  noticing
\begin{equation*}
       \int_{\bbR^2} \alp \cdot \nabla_y \Lambda^{t_1} f_1\frac{\frac{\alp}{|\alp|}\cdot \nb_y  f_2 \frac{\alp}{|\alp|}\cdot \nb_y  f_3}{\left(1+\left(\frac{\alp}{|\alp|}\cdot \nb_y f_4\right)^2\right)^\frac52}\frac{d\alp}{|\alp|^3} =0
\end{equation*}
by symmetry and applying \eqref{change y to alpha derivative}, we have
\begin{equation*}
    \mrI=-\int_{\bbR^2} \alp \cdot \nabla_\alp \dlt_\alp \Lambda^{t_1} f_1\frac{\frac{\alp}{|\alp|}\cdot \nb_y  f_2 \frac{\alp}{|\alp|}\cdot \nb_y  f_3}{\left(1+\left(\frac{\alp}{|\alp|}\cdot \nb_y f_4\right)^2\right)^\frac52}\frac{d\alp}{|\alp|^3}.
\end{equation*}
We integrate by parts and employ \eqref{derivatives of alpha functions-1}, \eqref{derivatives of alpha functions-for linear} to obtain
\begin{equation*}
        \mrI =  \int_{\bbR^2}\dlt_\alp\Lambda^{t_1} f_1 \,\nb_\alp \cdot \left(\frac{\alp}{|\alp|^3} \frac{\frac{\alp}{|\alp|}\cdot \nb_y  f_2 \frac{\alp}{|\alp|}\cdot \nb_y  f_3}{\left(1+\left(\frac{\alp}{|\alp|}\cdot \nb_y f_4\right)^2\right)^\frac52}\right)d\alp =- \int_{\bbR^2}\dlt_\alp\Lambda^{t_1} f_1 \frac{\frac{\alp}{|\alp|}\cdot \nb_y  f_2 \frac{\alp}{|\alp|}\cdot \nb_y  f_3}{\left(1+\left(\frac{\alp}{|\alp|}\cdot \nb_y f_4\right)^2\right)^\frac52}\frac{d\alp}{|\alp|^3}.
\end{equation*}
Making a change of variables $\alp \mapsto -\alp$, we observe
\begin{equation*}
\int_{\bbR^2}\dlt_\alp\Lambda^{t_1} f_1 \frac{\frac{\alp}{|\alp|}\cdot \nb_y  f_2 \frac{\alp}{|\alp|}\cdot \nb_y  f_3}{\left(1+\left(\frac{\alp}{|\alp|}\cdot \nb_y f_4\right)^2\right)^\frac52}\frac{d\alp}{|\alp|^3}=\int_{\bbR^2}\dlt_{-\alp}\Lambda^{t_1} f_1 \frac{\frac{\alp}{|\alp|}\cdot \nb_y  f_2 \frac{\alp}{|\alp|}\cdot \nb_y  f_3}{\left(1+\left(\frac{\alp}{|\alp|}\cdot \nb_y f_4\right)^2\right)^\frac52}\frac{d\alp}{|\alp|^3}
\end{equation*}
which implies $\mrI= Q^{t_1,1}[f_1,f_2,f_3,f_4]$.
We decompose $\mrII$ into
\begin{equation*}
    \begin{split}
        \mrII &= Q^{t_1,2}[f_1,f_2,f_3,f_4] + \mrII_1 + \mrII_2,
    \end{split}
\end{equation*}
where
\begin{equation*}
    \begin{split}
        \mrII_1 &= \int_{\bbR^2\backslash B_1(0)}\alp \cdot \nabla_y  \Lambda^{t_1} f_1\left(\frac{\Dlt_\alp f_2 \Dlt_\alp f_3}{\left(1+\left(\Dlt_\alp f_4\right)^2\right)^\frac52} -\frac{\frac{\alp}{|\alp|}\cdot \nb_y  f_2 \frac{\alp}{|\alp|}\cdot \nb_y  f_3}{\left(1+\left(\frac{\alp}{|\alp|}\cdot \nb_y f_4\right)^2\right)^\frac52}\right)\frac{d\alp}{|\alp|^3},\\
        \mrII_2 &= -\int_{\bbR^2}\alp \cdot \nabla_y  \Lambda^{t_1} f_1(y-\alp) \left(\frac{\Dlt_\alp f_2 \Dlt_\alp f_3}{\left(1+\left(\Dlt_\alp f_4\right)^2\right)^\frac52} -\frac{\frac{\alp}{|\alp|}\cdot \nb_y  f_2 \frac{\alp}{|\alp|}\cdot \nb_y  f_3}{\left(1+\left(\frac{\alp}{|\alp|}\cdot \nb_y f_4\right)^2\right)^\frac52}\right)\frac{d\alp}{|\alp|^3}.
    \end{split}
\end{equation*}
We see $\mrII_1=Q^{t_1,3}[f_1,f_2,f_3,f_4]$ since
\begin{equation*}
        \int_{\bbR^2\backslash B_1(0)}\alp \cdot \nabla_y \Lambda^{t_1} f_1\frac{\frac{\alp}{|\alp|}\cdot \nb_y  f_2 \frac{\alp}{|\alp|}\cdot \nb_y  f_3}{\left(1+\left(\frac{\alp}{|\alp|}\cdot \nb_y f_4\right)^2\right)^\frac52}\frac{d\alp}{|\alp|^3}=0
\end{equation*}
by symmetry.
For $\mrII_2$, we apply \eqref{change y to alpha derivative} and then integrate by parts to obtain
\begin{equation*}
    \mrII_2=\frac{1}{2\pi}\int_{\bbR^2}\dlt_\alp\Lambda^{t_1} f_1 \,\nb_\alp \cdot \left(\frac{\alp}{|\alp|^3} \left(\frac{\Dlt_\alp f_2 \Dlt_\alp f_3}{\left(1+\left(\Dlt_\alp f_4\right)^2\right)^\frac52} -\frac{\frac{\alp}{|\alp|}\cdot \nb_y  f_2 \frac{\alp}{|\alp|}\cdot \nb_y  f_3}{\left(1+\left(\frac{\alp}{|\alp|}\cdot \nb_y f_4\right)^2\right)^\frac52} \right)\right)d\alp.
\end{equation*}
Computing with the aids of \eqref{derivatives of alpha functions-1} and \eqref{derivatives of alpha functions-for linear}, we arrive at
\begin{equation*}
    \mrII_2 = Q^{t_1,4}[f_1,f_2,f_3,f_4]+Q^{t_1,5}[f_1,f_2,f_3,f_4].
\end{equation*}
\end{proof}

The estimates of $Q^{t_1,j}$ $(1\le j \le7)$ range over Lemma \ref{lem: Q^1-5} - Lemma \ref{lem: Q^7}, which leads to Proposition \ref{proposition: key for linear} by choosing the smallest $t^*$ and $t^{**}$ among those in Lemma \ref{lem: Q^1-5} - Lemma \ref{lem: Q^7} due to \eqref{sobolev interpolation}.
\begin{lemma}\label{lem: Q^1-5}
     For any  $t_1\in (3/2,2)$, there exist $t^{*}=t^*(t_1)\in(t_1,2)$ and $t^{**}=t^{**}(t_1)\in(2,t_1+1)$ such that
  \begin{equation*}
  \begin{split}
    \sum_{j=1}^5&\normb{Q^{t_1,j}[f_1,f_2,f_3,f_4]}\\
    &\lesssim \nrmb{f_1}_{\dot{H}^{t^{**}}\cap\dot{H}^{t_1+1}}\nrmb{f_2}_{\dot{W}^{1,\infty}\cap\dot{H}^{t^{*}} \cap \dot{H}^{t_1+1} }\nrmb{f_3}_{\dot{W}^{1,\infty}\cap\dot{H}^{t^{**}}\cap\dot{H}^{t_1+1}}\left(1+\nrmb{f_4}_{\dot{H}^{t^{**}}\cap\dot{H}^{t_1+1}}\right).
    \end{split}
\end{equation*}
\end{lemma}
\begin{proof}
To estimate $Q^{t_1,1}[f_1,f_2,f_3,f_4]$, we observe that the factor
$\frac{\frac{\alp}{|\alp|}\cdot \nb_y  f_2 \frac{\alp}{|\alp|}\cdot \nb_y  f_3}{\left(1+\left(\frac{\alp}{|\alp|}\cdot \nb_y f_4\right)^2\right)^\frac52}$ is independent of the length $|\alp|$, so that for $\alp=r\sigma$
\begin{equation*}
    |Q^{t_1,1}[f_1,f_2,f_3,f_4]| \lesssim |\nb_y f_2| |\nb_y f_3|\int_{\bbS^1} \Abs{\int_{0}^\infty \Delta_{r\sigma} \Lambda^{t_1} f_1(y)+\Delta_{-r\sigma} \Lambda^{t_1} f_1(y)\frac{dr}{r}} d\sigma.
\end{equation*}
Thus using the Minkowski's inequality, we have
\begin{equation*}
    \nrmb{Q^{t_1,1}[f_1,f_2,f_3,f_4]}_{L^2_y}\lesssim \nrmb{\nb_y f_2}_{L^{\infty}_y}\nrmb{\nb_y f_3}_{L^{\infty}_y}\int_{\bbS^1}\nrmb{\int_{0}^\infty \Delta_{r\sigma} \Lambda^{t_1} f_1(y)+\Delta_{-r\sigma} \Lambda^{t_1} f_1 (y)\frac{dr}{r}}_{L^2_y}  d\sigma.
\end{equation*}
Now applying the same argument with the proof of Lemma \ref{lem: T^1(klin)}, we obtain
\begin{equation}\label{key est: for T^1(klin) to make t_1+1}
    \int_{\bbS^1}\nrmb{\int_{0}^\infty \Delta_{r\sigma} \Lambda^{t_1} f_1(y)+\Delta_{-r\sigma} \Lambda^{t_1} f_1 (y)\frac{dr}{r}}_{L^2_y}  d\sigma \lesssim \nrmb{f_1}_{\dot{H}^{t_1+1}},
\end{equation}
so that
\begin{equation*}
    \nrmb{Q^{t_1,1}[f_1,f_2,f_3,f_4]}_{L^2_y}\lesssim \nrmb{f_1}_{\dot{H}^{t_1+1}} \nrmb{ f_2}_{\dot{W}^{1,\infty}}\nrmb{f_3}_{\dot{W}^{1,\infty}}.
\end{equation*}

For $Q^{t_1,2}[f_1,f_2,f_3,f_4]$, applying the mean value theorem to the function $F(x)=(1+x^2)^{-\frac52}$, we can see 
    \begin{equation*}
        \left|\left(1+ (\Delta_\alp f_4)^2\right)^{-\frac52}-\left(1+ \left(\frac{\alp}{|\alp|}\cdot \nb_y  f_4\right)^2\right)^{-\frac52} \right| \lesssim \Abs{\Delta_\alp f_4-\frac{\alp}{|\alp|}\cdot \nb_y  f_4}.
    \end{equation*}
Based on this inequality and an elementary formula
\begin{equation}\label{formular: difference of three product}
\begin{split}
    4(a_+b_+c_+ -\,a_-b_-c_-)&=(a_+ -\, a_-)(b_+ +\, b_-)(c_+ +\, c_-) + (a_+ + \,a_-)(b_+ - \,b_-)(c_+ + \,c_-)\\
    &\quad + (a_+ +\, a_-)(b_+ +\, b_-)(c_+ -\, c_-) + (a_+\, - a_-)(b_+ - \,b_-)(c_+ -\, c_-),
    \end{split}
\end{equation}
we deduce
\begin{equation}\label{spliting f_2f_3/(1+w)}
\begin{split}
    \Abs{\frac{\Dlt_\alp f_2 \Dlt_\alp f_3 }{\left(1+\left(\Dlt_\alp f_4\right)^2\right)^\frac52} -\frac{\frac{\alp}{|\alp|}\cdot \nb_y  f_2 \frac{\alp}{|\alp|}\cdot \nb_y  f_3}{\left(1+\left(\frac{\alp}{|\alp|}\cdot \nb_y f_4\right)^2\right)^\frac52}}
    &\lesssim \Abs{\Delta_\alp f_2-\frac{\alp}{|\alp|}\cdot \nb_y  f_2}\Abs{\Delta_\alp f_3+\frac{\alp}{|\alp|}\cdot \nb_y  f_3} \\
    &\quad+\Abs{\Delta_\alp f_2+\frac{\alp}{|\alp|}\cdot \nb_y  f_2}\Abs{\Delta_\alp f_3-\frac{\alp}{|\alp|}\cdot \nb_y  f_3} \\
    &\quad+\Abs{\Delta_\alp f_2+\frac{\alp}{|\alp|}\cdot \nb_y  f_2}\Abs{\Delta_\alp f_3+\frac{\alp}{|\alp|}\cdot \nb_y  f_3}\Abs{\Delta_\alp f_4-\frac{\alp}{|\alp|}\cdot \nb_y  f_4}\\
    &\quad+\Abs{\Delta_\alp f_2-\frac{\alp}{|\alp|}\cdot \nb_y  f_2}\Abs{\Delta_\alp f_3-\frac{\alp}{|\alp|}\cdot \nb_y  f_3}\Abs{\Delta_\alp f_4-\frac{\alp}{|\alp|}\cdot \nb_y  f_4}.
    \end{split}
\end{equation}
Since \eqref{fundamental thm of cal} and \eqref{App of Morrey} toghether with \eqref{sobolev embedding-1} yield
\begin{equation*}
    \begin{split}
        \Abs{\Delta_\alp f_2-\frac{\alp}{|\alp|}\cdot \nb_y  f_2}\Abs{\Delta_\alp f_3+\frac{\alp}{|\alp|}\cdot \nb_y  f_3} &\lesssim |\alp|^{t_1-1}\nrmb{\nb^2 f_2}_{\dot{H}^{t_1-1}}\nrmb{\nb f_3}_{L^{\infty}}, \\
        \Abs{\Delta_\alp f_2+\frac{\alp}{|\alp|}\cdot \nb_y  f_2}\Abs{\Delta_\alp f_3-\frac{\alp}{|\alp|}\cdot \nb_y  f_3} &\lesssim |\alp|^{t_1-1}\nrmb{\nb f_2}_{L^{\infty}}\nrmb{\nb^2 f_3}_{\dot{H}^{t_1-1}},\\
       \Abs{\Delta_\alp f_2+\frac{\alp}{|\alp|}\cdot \nb_y  f_2}\Abs{\Delta_\alp f_3+\frac{\alp}{|\alp|}\cdot \nb_y  f_3}\Abs{\Delta_\alp f_4-\frac{\alp}{|\alp|}\cdot \nb_y  f_4}&\lesssim |\alp|^{t_1-1}\nrmb{\nb f_2}_{L^{\infty}}\nrmb{\nb f_3}_{L^{\infty}}\nrmb{\nb^2 f_4}_{\dot{H}^{t_1-1}},\\
       \Abs{\Delta_\alp f_2-\frac{\alp}{|\alp|}\cdot \nb_y  f_2}\Abs{\Delta_\alp f_3-\frac{\alp}{|\alp|}\cdot \nb_y  f_3}\Abs{\Delta_\alp f_4-\frac{\alp}{|\alp|}\cdot \nb_y  f_4}&\lesssim |\alp|^{t_1-1}\nrmb{\nb f_2}_{L^{\infty}}\nrmb{\nb f_3}_{L^{\infty}}\nrmb{\nb^2 f_4}_{\dot{H}^{t_1-1}},
    \end{split}
\end{equation*}
we obtain
\begin{equation*}
\begin{split}
    \nrmb{Q^{t_1,2}[f_1,f_2,f_3,f_4]}_{L^2_y}&\lesssim \normb{\nabla_y \Lambda^{t_1} f_1}\nrmb{f_2}_{\dot{W}^{1,\infty}\cap\dot{H}^{t_1+1}}\nrmb{f_3}_{\dot{W}^{1,\infty}\cap\dot{H}^{t_1+1}}\left(1+\nrmb{f_4}_{\dot{H}^{t_1+1}}\right)\left(\int_{B_1(0)} \frac{d\alp}{|\alp|^{3-t_1}}\right) \\
    &\lesssim \nrmb{f_1}_{\dot{H}^{t_1+1}}\nrmb{f_2}_{\dot{W}^{1,\infty}\cap\dot{H}^{t_1+1}}\nrmb{f_3}_{\dot{W}^{1,\infty}\cap\dot{H}^{t_1+1}}\left(1+\nrmb{f_4}_{\dot{H}^{t_1+1}}\right),
    \end{split}
\end{equation*}
where we used $t_1\in(3/2,2)$ in the last inequality.

For $Q^{t_1,3}[f_1,f_2,f_3,f_4]$, we use \eqref{fundamental thm of cal} and \eqref{Morrey} to observe
\begin{equation*}
    \Abs{\frac{\Dlt_\alp f_2 \Dlt_\alp f_3}{\left(1+\left(\Dlt_\alp f_4\right)^2\right)^\frac52}}\lesssim |\alp|^{\frac{t_1}{2}-1}\nrmb{\nb_y f_2}_{\dot{H}^{\frac{t_1}{2}}}\nrmb{\nb_y f_3}_{L^{\infty}},
\end{equation*}
which gives
\begin{equation*}
\begin{split}
    \nrmb{Q^{t_1,3}[f_1,f_2,f_3,f_4]}_{L^2_y}&\lesssim \nrmb{f_1}_{\dot{H}^{t_1+1}}\nrmb{f_2}_{\dot{H}^{\frac{t_1}{2}+1}}\nrmb{f_3}_{\dot{W}^{1,\infty}}\left(\int_{\bbR^2\backslash B_1(0)} \frac{d\alp}{|\alp|^{3-\frac{t_1}{2}}}\right)\\
    &\lesssim \nrmb{f_1}_{\dot{H}^{t_1+1}}\nrmb{f_2}_{\dot{H}^{\frac{t_1}{2}+1}}\nrmb{f_3}_{\dot{W}^{1,\infty}},
    \end{split}
\end{equation*}
where we used $t_1\in(3/2,2)$ in the last inequality.

For $Q^{t_1,4}[f_1,f_2,f_3,f_4]$, we use \eqref{fundamental thm of cal} and \eqref{spliting f_2f_3/(1+w)} to obtain
\begin{equation*}
    \begin{split}
        \Abs{Q^{t_1,4}[f_1,f_2,f_3,f_4]}&\lesssim \nrmb{\nb_y f_3}_{L^\infty} \int_{\bbR^2} \Abs{\Delta_\alp \Lambda^{t_1} f_1(y)} \left|\frac{\alp}{|\alp|}\cdot \nb_y f_2(y)-\Delta_\alp f_2(y) \right| \frac{d\alp}{|\alp|^2}\\
        &\quad + \nrmb{\nb_y f_2}_{L^\infty} \int_{\bbR^2} \Abs{\Delta_\alp \Lambda^{t_1} f_1(y)} \left|\frac{\alp}{|\alp|}\cdot \nb_y f_3(y)-\Delta_\alp f_3(y) \right| \frac{d\alp}{|\alp|^2}\\
        &\quad + \nrmb{\nb_y f_2}_{L^\infty} \nrmb{\nb_y f_3}_{L^\infty}\int_{\bbR^2} \Abs{\Delta_\alp \Lambda^{t_1} f_1(y)} \left|\frac{\alp}{|\alp|}\cdot \nb_y f_4(y)-\Delta_\alp f_4(y) \right| \frac{d\alp}{|\alp|^2}.
    \end{split}
\end{equation*}
Now we proceed similarly to the proof of Lemma \ref{lem: T^4(klin)} to have
\begin{equation*}
    \normb{Q^{t_1,4}[f_1,f_2,f_3,f_4]}\lesssim \nrmb{f_1}_{\dot{H}^{t_1+\frac34}}\nrmb{f_2}_{\dot{W}^{1,\infty}\cap\dot{H}^{\frac94}}\nrmb{f_3}_{\dot{W}^{1,\infty}\cap\dot{H}^{\frac94}}\left(1+\nrmb{f_4}_{\dot{H}^{\frac94}}\right).
\end{equation*}

For $Q^{t_1,5}[f_1,f_2,f_3,f_4]$, we compute
\begin{equation*}
    \begin{split}
        \alp \cdot \nb_\alp \left(\frac{\Dlt_\alp f_2 \Dlt_\alp f_3}{\left(1+\left(\Dlt_\alp f_4\right)^2\right)^\frac52} \right)
    &\quad=-\frac{ \Dlt_\alp f_3}{\left(1+\left(\Dlt_\alp f_4\right)^2\right)^\frac52}\left(\Delta_\alp f_2-\frac{\alp \cdot \nb_y f_2}{|\alp|} + \alp \cdot \nb_y \Delta_\alp f_2\right) \\
    &\qquad-\frac{ \Dlt_\alp f_2}{\left(1+\left(\Dlt_\alp f_4\right)^2\right)^\frac52}\left(\Delta_\alp f_3-\frac{\alp \cdot \nb_y f_3}{|\alp|} + \alp \cdot \nb_y \Delta_\alp f_3\right)\\
    &\qquad+\frac{5\Dlt_\alp f_2\Dlt_\alp f_3\Delta_\alp f_4}{\left(1+ (\Delta_\alp f_4)^2\right)^{\frac72}}\left(\Delta_\alp f_4-\frac{\alp \cdot \nb_y f_4}{|\alp|} + \alp \cdot \nb_y \Delta_\alp f_4 \right),
    \end{split}
\end{equation*}
so that
\eqref{fundamental thm of cal} yields
\begin{equation*}
    \begin{split}
        \Abs{\alp \cdot \nb_\alp \left(\frac{\Dlt_\alp f_2 \Dlt_\alp f_3}{\left(1+\left(\Dlt_\alp f_4\right)^2\right)^\frac52} \right)}
    &\lesssim \nrmb{\nb f_3}_{L^\infty}\Abs{\Delta_\alp f_2-\frac{\alp \cdot \nb_y f_2}{|\alp|} + \alp \cdot \nb_y \Delta_\alp f_2} \\
    &\quad\nrmb{\nb f_2}_{L^\infty}\Abs{\Delta_\alp f_3-\frac{\alp \cdot \nb_y f_3}{|\alp|} + \alp \cdot \nb_y \Delta_\alp f_3}\\
    &\quad+\nrmb{\nb f_2}_{L^\infty}\nrmb{\nb f_3}_{L^\infty}\Abs{\Delta_\alp f_4-\frac{\alp \cdot \nb_y f_4}{|\alp|} + \alp \cdot \nb_y \Delta_\alp f_4 }.
    \end{split}
\end{equation*}
Thus we proceed similarly to the proof of Lemma \ref{lem: T^5(klin)} to have
\begin{equation*}
    \normb{Q^{t_1,5}[f_1,f_2,f_3,f_4]}\lesssim \nrmb{f_1}_{\dot{H}^{t_1+\frac34}}\nrmb{f_2}_{\dot{W}^{1,\infty}\cap\dot{H}^{\frac94}}\nrmb{f_3}_{\dot{W}^{1,\infty}\cap\dot{H}^{\frac94}}\left(1+\nrmb{f_4}_{\dot{H}^{\frac94}}\right).
\end{equation*}
Note that since $t_1\in(3/2,2)$, the exponents of Sobolev spaces appeared above estimates satisfy
\begin{equation}\label{sobolev exponet for Q^1-Q^5}
    t_1+\frac34,\;\,\frac94 \in (2,t_1+1),\qquad \frac{t_1}{2}+1\in(t_1,2).
\end{equation}
Therefore, combining all, we are done.
\end{proof}

\begin{lemma}\label{lem: Q^6}
    For any  $t_1\in (3/2,2)$, there exist $t^{*}=t^*(t_1)\in(t_1,2)$ and $t^{**}=t^{**}(t_1)\in(2,t_1+1)$ such that
    \begin{equation*}
    \begin{split}
&\nrmb{Q^{t_1,6}[f_1,f_2,f_3,f_4]}_{L^2} \\
&\qquad\lesssim \nrmb{f_1}_{\dot{H}^{t^{**}}\cap\dot{H}^{t_1+1}}\nrmb{f_2}_{\dot{H}^{t^{*}}\cap\dot{H}^{t_1+1}}\nrmb{f_3}_{\dot{W}^{1,\infty}\cap\dot{H}^{t^{**}}\cap\dot{H}^{t_1+1}}\left(1 +\nrmb{f_4}_{\dot{H}^{t^{**}}\cap\dot{H}^{t_1+1}}\right).
\end{split}
    \end{equation*}
\end{lemma}
\begin{proof}
    To begin with, we show
    \begin{equation}\label{est: laplacian for T_1}
        \begin{split}
            &\nrmb{\frac{\Dlt_\alp f_2 \Dlt_\alp f_3}{\left(1+\left(\Dlt_\alp f_4\right)^2\right)^{\frac52}}}_{\dot{H}^2_y}\\
            &\quad\lesssim \frac{1}{|\alp|^2}\left(\nrmb{  \dlt_\alp f_3}_{\dot{H}^{3-\frac{t_1}{2}}_y}\nrmb{\dlt_\alp f_2}_{\dot{H}^{\frac{t_1}{2}}_y} + \nrmb{ \dlt_\alp f_3}_{\dot{H}^{\frac{3}{2}}_y}\nrmb{ \dlt_\alp f_2}_{\dot{H}^{\frac{3}{2}}_y} + \nrmb{f_3}_{\dot{W}^{1,\infty}}\nrmb{  \dlt_\alp f_4}_{\dot{H}^{\frac{3}{2}}_y}\nrmb{ \dlt_\alp f_2}_{\dot{H}^{\frac{3}{2}}_y} \right.\\
&\qquad\left.+\nrmb{f_3}_{\dot{W}^{1,\infty}}\nrmb{\dlt_\alp f_2}_{\dot{H}^{\frac{t_1}{2}}_y}\nrmb{ \dlt_\alp f_4}_{\dot{H}^{3-\frac{t_1}{2}}_y} \right) + \frac{\nrmb{ f_3}_{\dot{W}^{1,\infty}}}{|\alp|}\nrmb{ \dlt_\alp f_2}_{\dot{H}^2_y} \\
&\qquad+\frac{\nrmb{ f_3}_{\dot{W}^{1,\infty}\cap\dot{H}^{3-\frac{t_1}{2}}}\nrmb{ f_4}_{\dot{H}^{3-\frac{t_1}{2}} \cap \dot{H}^{\frac{5}{2}}}}{|\alp|^\frac32}\nrmb{\dlt_\alp f_2}_{\dot{H}^{\frac{t_1}{2}}_y}.\\
        \end{split}
    \end{equation}
    Using the identity $\Dlt(fg)=f\Dlt g + 2\nb f \cdot \nb g + g\Dlt f$, we split
    \begin{equation*}
        \begin{split}
            \Dlt\left(\frac{\Dlt_\alp f_2 \Dlt_\alp f_3}{\left(1+\left(\Dlt_\alp f_4\right)^2\right)^{\frac52}}\right)&=\left(1+\left(\Dlt_\alp f_4\right)^2\right)^{-\frac52}\Dlt\left(\Dlt_\alp f_2 \Dlt_\alp f_3\right)\\
            &\quad+2\nb_y \left(\left(1+\left(\Dlt_\alp f_4\right)^2\right)^{-\frac52}\right)\cdot\nb_y \left(\Dlt_\alp f_2 \Dlt_\alp f_3\right) \\
            &\quad+ \Dlt_\alp f_2 \Dlt_\alp f_3 \Dlt \left(\left(1+\left(\Dlt_\alp f_4\right)^2\right)^{-\frac52}\right) \\
            &=\mrI_1+\mrI_2+\mrI_3.
        \end{split}
    \end{equation*}
    We recall \eqref{fundamental thm of cal}, \eqref{Morrey}, and \eqref{sobolev embedding-1} to obtain
    \begin{equation}\label{Morrey for T_1}
        \Abs{ \dlt_\alp f_3}\lesssim |\alp|\nrmb{\nb f_3}_{L^\infty}, \quad\;  \quad \Abs{\nb_y \dlt_\alp f_4}\lesssim  |\alp|^{\frac12} \nrmb{\nb_y^2 f_4 }_{L^4}\lesssim |\alp|^{\frac12} \nrmb{f_4 }_{\dot{H}^{\frac{5}{2}}},
    \end{equation}
    which yield
    \begin{equation}\label{three cases of laplacian - linear}
    \begin{split}
        \Abs{\mrI_1}&\lesssim \frac{1}{|\alp|^2} \left(\Abs{ \Dlt \dlt_\alp f_3}\Abs{\dlt_\alp f_2} + \Abs{ \nb_y \dlt_\alp f_3}\Abs{\nb_y \dlt_\alp f_2} + \Abs{ \dlt_\alp f_3}\Abs{ \Dlt \dlt_\alp f_2}\right) \\
        &\lesssim \frac{1}{|\alp|^2} \left(\Abs{ \Dlt \dlt_\alp f_3}\Abs{\dlt_\alp f_2} + \Abs{ \nb_y \dlt_\alp f_3}\Abs{\nb_y \dlt_\alp f_2}\right) +  \frac{\nrmb{\nb f_3}_{L^\infty}}{|\alp|}\Abs{ \Dlt \dlt_\alp f_2},\\
        \Abs{\mrI_2}&\lesssim \frac{\Abs{ \nb_y \dlt_\alp f_4}}{|\alp|^3}\left(\Abs{ \nb_y \dlt_\alp f_3}\Abs{\dlt_\alp f_2} +\Abs{ \dlt_\alp f_3}\Abs{\nb_y \dlt_\alp f_2}\right) \\
            &\lesssim \frac{\nrmb{f_4 }_{\dot{H}^{\frac{5}{2}}}}{|\alp|^\frac52}\Abs{ \nb_y \dlt_\alp f_3}\Abs{\dlt_\alp f_2} + \frac{\nrmb{\nb f_3}_{L^\infty}}{|\alp|^2}\Abs{ \nb_y \dlt_\alp f_4}\Abs{\nb_y \dlt_\alp f_2},\\
            \Abs{\mrI_3}&\lesssim \frac{\Abs{ \dlt_\alp f_3}\Abs{ \dlt_\alp f_2}}{|\alp|^4}\Abs{ \nb_y \dlt_\alp f_4}^2  +\frac{\Abs{ \dlt_\alp f_3}\Abs{ \dlt_\alp f_2}}{|\alp|^3}\Abs{ \Dlt \dlt_\alp f_4}\\
            &\lesssim \frac{\nrmb{\nb f_3}_{L^\infty}\nrmb{f_4 }_{\dot{H}^{\frac{5}{2}}}}{|\alp|^\frac52}\Abs{ \nb_y \dlt_\alp f_4}\Abs{\dlt_\alp f_2} + \frac{\nrmb{\nb f_3}_{L^\infty}}{|\alp|^2}\Abs{ \dlt_\alp f_2}\Abs{ \Dlt \dlt_\alp f_4}.
        \end{split}
    \end{equation}
    Hence using \eqref{sobolev embedding-2} and H\"older's inequality, we have
    \begin{equation*}
        \begin{split}
            &\nrmb{\frac{\Dlt_\alp f_2 \Dlt_\alp f_3}{\left(1+\left(\Dlt_\alp f_4\right)^2\right)^{\frac52}}}_{\dot{H}^2_y}\\
            &\quad\lesssim \frac{1}{|\alp|^2}\left(\nrmb{ \Dlt \dlt_\alp f_3}_{\dot{H}^{1-\frac{t_1}{2}}_y}\nrmb{\dlt_\alp f_2}_{\dot{H}^{\frac{t_1}{2}}_y} + \nrmb{ \nb_y \dlt_\alp f_3}_{\dot{H}^{\frac{1}{2}}_y}\nrmb{\nb_y \dlt_\alp f_2}_{\dot{H}^{\frac{1}{2}}_y} + \nrmb{\nb f_3}_{L^\infty} \nrmb{ \nb_y \dlt_\alp f_4}_{\dot{H}^{\frac{1}{2}}_y}\nrmb{\nb_y \dlt_\alp f_2}_{\dot{H}^{\frac{1}{2}}_y} \right.\\
            &\qquad\left.+\nrmb{\nb f_3}_{L^\infty}\nrmb{\dlt_\alp f_2}_{\dot{H}^{\frac{t_1}{2}}_y}\nrmb{ \Dlt \dlt_\alp f_4}_{\dot{H}^{1-\frac{t_1}{2}}_y} \right) + \frac{\nrmb{\nb f_3}_{L^\infty}}{|\alp|}\nrmb{ \Dlt \dlt_\alp f_2}_{L^2_y} \\
&\qquad+\frac{\nrmb{f_4}_{\dot{H}^\frac52}}{|\alp|^\frac52}\nrmb{ \nb_y \dlt_\alp f_3}_{L^{\frac{4}{t_1}}_y}\nrmb{\dlt_\alp f_2}_{L^{\frac{4}{2-t_1}}_y} +\frac{\nrmb{\nb f_3}_{L^\infty}\nrmb{f_4}_{\dot{H}^\frac52}}{|\alp|^\frac52}\nrmb{ \nb_y \dlt_\alp f_4}_{L^{\frac{4}{t_1}}_y}\nrmb{\dlt_\alp f_2}_{L^{\frac{4}{2-t_1}}_y} 
        \end{split}
    \end{equation*}
    Employing \eqref{fundamental thm of cal} together with \eqref{sobolev embedding-1}, we can bound
    \begin{equation*}
        \nrmb{ \nb_y \dlt_\alp f_j}_{L^{\frac{4}{t_1}}_y} \lesssim |\alp|\nrmb{\nb^2_y f_j}_{L^{\frac{4}{t_1}}}\lesssim |\alp|\nrmb{\nb^2_y f_j}_{\dot{H}^{1-\frac{t_1}{2}}} \;\,(j=3,4), \quad\; \nrmb{\dlt_\alp f_2}_{L^{\frac{4}{2-t_1}}_y} \lesssim \nrmb{\dlt_\alp f_2}_{\dot{H}^{\frac{t_1}{2}}_y},
    \end{equation*}
    so that we arrive at \eqref{est: laplacian for T_1}.
    
    Now using the Sobolev embeddings: $\dot{H}^{t_1-1}(\bbR^2)\hookrightarrow L^\frac{2}{2-t_1}(\bbR^2)$, $\dot{H}^{2-t_1}(\bbR^2)\hookrightarrow L^\frac{2}{t_1-1}(\bbR^2)$ and \eqref{est: laplacian for T_1}, we have
    \begin{equation*}
\begin{split}
    \nrmb{Q^{t_1,6}[f_1,f_2,f_3,f_4]}_{L^2}&\lesssim  \int_{\bbR^2} \nrmb{ \nabla_y \dlt_\alp  f_1}_{L^\frac{2}{2-t_1}_y} \nrmb{\Lmb^{t_1}\left(\frac{\Dlt_\alp f_2 \Dlt_\alp f_3}{\left(1+\left(\Dlt_\alp f_4\right)^2\right)^\frac52} \right)}_{\dot{H}^{2-t_1}_y} \frac{d\alp}{|\alp|^2} \\
    &\lesssim \int_{\bbR^2} \nrmb{ \nabla_y   \dlt_\alp f_1}_{L^\frac{2}{2-t_1}} \nrmb{  \dlt_\alp f_3}_{\dot{H}^{3-\frac{t_1}{2}}_y}\nrmb{\dlt_\alp f_2}_{\dot{H}^{\frac{t_1}{2}}_y}\frac{d\alp}{|\alp|^4} \\
    &\quad+ \int_{\bbR^2} \nrmb{ \nabla_y   \dlt_\alp f_1}_{L^\frac{2}{2-t_1}}  \nrmb{ \dlt_\alp f_3}_{\dot{H}^{\frac{3}{2}}_y}\nrmb{ \dlt_\alp f_2}_{\dot{H}^{\frac{3}{2}}_y}\frac{d\alp}{|\alp|^4}\\
    &\quad+ \nrmb{f_3}_{\dot{W}^{1,\infty}}\int_{\bbR^2} \nrmb{ \nabla_y   \dlt_\alp f_1}_{L^\frac{2}{2-t_1}}  \nrmb{ \dlt_\alp f_4}_{\dot{H}^{\frac{3}{2}}_y}\nrmb{ \dlt_\alp f_2}_{\dot{H}^{\frac{3}{2}}_y}\frac{d\alp}{|\alp|^4}\\
    &\quad+ \nrmb{f_3}_{\dot{W}^{1,\infty}}\int_{\bbR^2} \nrmb{ \nabla_y   \dlt_\alp f_1}_{L^\frac{2}{2-t_1}} \nrmb{\dlt_\alp f_2}_{\dot{H}^{\frac{t_1}{2}}_y}\nrmb{ \dlt_\alp f_4}_{\dot{H}^{3-\frac{t_1}{2}}_y}\frac{d\alp}{|\alp|^4}\\
    &\quad+ \nrmb{ f_3}_{\dot{W}^{1,\infty}}\int_{\bbR^2} \nrmb{ \nabla_y   \dlt_\alp f_1}_{\dot{H}^{t_1-1}} \nrmb{ \dlt_\alp f_2}_{\dot{H}^{2}_y} \frac{d\alp}{|\alp|^3}\\
    &\quad+ \nrmb{ f_3}_{\dot{W}^{1,\infty}\cap\dot{H}^{3-\frac{t_1}{2}}}\nrmb{ f_4}_{\dot{H}^{3-\frac{t_1}{2}} \cap \dot{H}^{\frac{5}{2}}}\int_{\bbR^2} \nrmb{ \nabla_y   \dlt_\alp f_1}_{\dot{H}^{t_1-1}} \nrmb{\dlt_\alp f_2}_{\dot{H}^{\frac{t_1}{2}}_y} \frac{d\alp}{|\alp|^\frac72} \\
    &= \mrII_1+\mrII_2+\mrII_3+\mrII_4+\mrII_5+\mrII_6.
\end{split}
\end{equation*}
Noticing \eqref{fundamental thm of cal} together with \eqref{sobolev embedding-1}, we can bound
    \begin{equation*}\label{nabla^2 for s - complex exponent}
        \nrmb{ \nabla_y   \dlt_\alp f_1}_{L^\frac{2}{2-t_1}_y} \lesssim |\alp|\nrmb{\nb^2 f_1}_{L^\frac{2}{2-t_1}_y}\lesssim |\alp|\nrmb{f_1}_{\dot{H}^{t_1+1}}.
    \end{equation*}
Using this inequality and \eqref{eq: prelemma-1}, we estimate
\begin{equation*}
    \begin{split}
        \mrII_1+&\mrII_2+\mrII_3+\mrII_4 \\
        &\lesssim \nrmb{f_1}_{\dot{H}^{t_1+1}}\int_{\bbR^2} \nrmb{  \dlt_\alp f_3}_{\dot{H}^{3-\frac{t_1}{2}}_y}\nrmb{\dlt_\alp f_2}_{\dot{H}^{\frac{t_1}{2}}_y} \frac{d\alp}{|\alp|^3} \\
        &\quad+ \nrmb{f_1}_{\dot{H}^{t_1+1}}\int_{\bbR^2} \nrmb{ \dlt_\alp f_3}_{\dot{H}^{\frac{3}{2}}_y}\nrmb{ \dlt_\alp f_2}_{\dot{H}^{\frac{3}{2}}_y} \frac{d\alp}{|\alp|^3}\\
        &\quad+ \nrmb{f_1}_{\dot{H}^{t_1+1}}\nrmb{f_3}_{\dot{W}^{1,\infty}}\int_{\bbR^2} \nrmb{ \dlt_\alp f_4}_{\dot{H}^{\frac{3}{2}}_y}\nrmb{ \dlt_\alp f_2}_{\dot{H}^{\frac{3}{2}}_y}\frac{d\alp}{|\alp|^3}\\
        &\quad+ \nrmb{f_1}_{\dot{H}^{t_1+1}}\nrmb{f_3}_{\dot{W}^{1,\infty}}\int_{\bbR^2} \nrmb{\dlt_\alp f_2}_{\dot{H}^{\frac{t_1}{2}}_y}\nrmb{ \dlt_\alp f_4}_{\dot{H}^{3-\frac{t_1}{2}}_y} \frac{d\alp}{|\alp|^3}\\
        &\lesssim \nrmb{f_1}_{\dot{H}^{t_1+1}}\left(\int_{\bbR^2}\nrmb{ \dlt_\alp f_3}^2_{\dot{H}^{3-\frac{t_1}{2}}_y}\frac{d\alp}{|\alp|^{3-\frac{t_1}{2}}}\right)^\frac12 \left(\int_{\bbR^2}\nrmb{\dlt_\alp f_2}^2_{\dot{H}^{\frac{t_1}{2}}_y}\frac{d\alp}{|\alp|^{3+\frac{t_1}{2}}}\right)^\frac12 \\
        &\quad +\nrmb{f_1}_{\dot{H}^{t_1+1}}\left(\int_{\bbR^2}\nrmb{ \dlt_\alp f_3}^2_{\dot{H}^{\frac{3}{2}}_y}\frac{d\alp}{|\alp|^{5-t_1}}\right)^\frac12 \left(\int_{\bbR^2}\nrmb{\dlt_\alp f_2}_{\dot{H}^{\frac{3}{2}}_y}^2\frac{d\alp}{|\alp|^{1+t_1}}\right)^\frac12 \\
        &\quad +\nrmb{f_1}_{\dot{H}^{t_1+1}}\nrmb{f_3}_{\dot{W}^{1,\infty}} \left(\int_{\bbR^2}\nrmb{ \dlt_\alp f_4}^2_{\dot{H}^{\frac{3}{2}}_y}\frac{d\alp}{|\alp|^{5-t_1}}\right)^\frac12 \left(\int_{\bbR^2}\nrmb{\dlt_\alp f_2}_{\dot{H}^{\frac{3}{2}}_y}^2\frac{d\alp}{|\alp|^{1+t_1}}\right)^\frac12 \\
        &\quad+ \nrmb{f_1}_{\dot{H}^{t_1+1}}\nrmb{f_3}_{\dot{W}^{1,\infty}}\left(\int_{\bbR^2}\nrmb{ \dlt_\alp f_4}^2_{\dot{H}^{3-\frac{t_1}{2}}_y}\frac{d\alp}{|\alp|^{3-\frac{t_1}{2}}}\right)^\frac12 \left(\int_{\bbR^2}\nrmb{\dlt_\alp f_2}^2_{\dot{H}^{\frac{t_1}{2}}_y}\frac{d\alp}{|\alp|^{3+\frac{t_1}{2}}}\right)^\frac12 \\
        &\lesssim \nrmb{f_1}_{\dot{H}^{t_1+1}}\nrmb{f_2}_{\dot{H}^{\frac{3t_1}{4}+\frac12}\cap \dot{H}^{\frac{t_1}{2}+1}}\nrmb{f_3}_{\dot{W}^{1,\infty}\cap\dot{H}^{\frac72-\frac{3t_1}{4}} \cap \dot{H}^{3-\frac{t_1}{2}}}\left(1+ +\nrmb{f_4}_{\dot{H}^{\frac72-\frac{3t_1}{4}} \cap \dot{H}^{3-\frac{t_1}{2}}}\right).
    \end{split}
\end{equation*}
For $\mrII_5$, we have
\begin{equation*}
    \begin{split}
        \mrII_5 &\lesssim \nrmb{ f_3}_{\dot{W}^{1,\infty}}\int_{\bbR^2} \nrmb{ \dlt_\alp f_1}_{\dot{H}^{t_1}_y}\nrmb{\dlt_\alp f_2}_{\dot{H}^{2}_y} \frac{d\alp}{|\alp|^3}\\
        &\lesssim \nrmb{ f_3}_{\dot{W}^{1,\infty}} \left(\int_{\bbR^2}\nrmb{ \dlt_\alp f_1}^2_{\dot{H}^{t_1}_y}\frac{d\alp}{|\alp|^{5-t_1}}\right)^\frac12 \left(\int_{\bbR^2}\nrmb{\dlt_\alp f_2}_{\dot{H}^{2}_y}^2\frac{d\alp}{|\alp|^{1+t_1}}\right)^\frac12 \\
        &\lesssim \nrmb{ f_3}_{\dot{W}^{1,\infty}} \nrmb{f_1}_{\dot{H}^{\frac{t_1}{2}+\frac32}} \nrmb{f_2}_{\dot{H}^{\frac{t_1}{2}+\frac32}}.
    \end{split}
\end{equation*}
Finally we estimate $\mrII_6$:
\begin{equation*}
    \begin{split}
        \mrII_6 &\lesssim \nrmb{ f_3}_{\dot{W}^{1,\infty}\cap\dot{H}^{3-\frac{t_1}{2}}}\nrmb{ f_4}_{\dot{H}^{3-\frac{t_1}{2}} \cap \dot{H}^{\frac{5}{2}}}\int_{\bbR^2} \nrmb{ \dlt_\alp f_1}_{\dot{H}^{t_1}_y}\nrmb{\dlt_\alp f_2}_{\dot{H}^{\frac{t_1}{2}}_y} \frac{d\alp}{|\alp|^\frac72}\\
        &\lesssim \nrmb{ f_3}_{\dot{W}^{1,\infty}\cap\dot{H}^{3-\frac{t_1}{2}}}\nrmb{ f_4}_{\dot{H}^{3-\frac{t_1}{2}} \cap \dot{H}^{\frac{5}{2}}} \left(\int_{\bbR^2}\nrmb{ \dlt_\alp f_1}^2_{\dot{H}^{t_1}_y}\frac{d\alp}{|\alp|^{4-\frac{t_1}{2}}}\right)^\frac12 \left(\int_{\bbR^2}\nrmb{\dlt_\alp f_2}_{\dot{H}^{\frac{t_1}{2}}_y}^2\frac{d\alp}{|\alp|^{3+\frac{t_1}{2}}}\right)^\frac12 \\
        &\lesssim \nrmb{ f_3}_{\dot{W}^{1,\infty}\cap\dot{H}^{3-\frac{t_1}{2}}}\nrmb{ f_4}_{\dot{H}^{3-\frac{t_1}{2}} \cap \dot{H}^{\frac{5}{2}}}\nrmb{f_1}_{\dot{H}^{\frac{3t_1}{4}+1}} \nrmb{f_2}_{\dot{H}^{\frac{3t_1}{4}+\frac12}}.
    \end{split}
\end{equation*}
Note that since $t_1\in(3/2,2)$, the exponents of Sobolev spaces appeared above estimates satisfy
\begin{equation*}
    \frac{5}{2},\;\,\frac{t_1}{2}+\frac{3}{2},\;\,\frac{3t_1}{4}+1,\;\,\frac72-\frac{3t_1}{4},\;\,3-\frac{t_1}{2} \in (2,t_1+1),\quad \frac{t_1}{2}+1,\;\, \frac{3t_1}{4}+\frac{1}{2}\in(t_1,2).
\end{equation*}
Therefore, combining all, we are done.

\end{proof}

\begin{lemma}\label{lem: Q^7}
    For any  $t_1\in (3/2,2)$, there exists $t^{**}=t^{**}(t_1)\in(2,t_1+1)$ such that
    \begin{equation*}
    \begin{split}
&\nrmb{Q^{t_1,7}[f_1,f_2,f_3,f_4]}_{L^2} \\
&\quad\lesssim \nrmb{f_1}_{\dot{H}^{t^{**}}\cap\dot{H}^{t_1+1}}\nrmb{f_2}_{\dot{W}^{1,\infty}\cap \dot{H}^{t^{**}}\cap\dot{H}^{t_1+1}}\nrmb{f_3}_{\dot{W}^{1,\infty}\cap\dot{H}^{t^{**}}\cap\dot{H}^{t_1+1}} \left(1 +\nrmb{f_4}_{\dot{H}^{t^{**}}\cap\dot{H}^{t_1+1}}\right).
\end{split}
    \end{equation*}
\end{lemma}
\begin{proof}
Using the mean value theorem to the function $F(x)=(1+x^2)^{-\frac52}$, we can show
\begin{equation*}
    \Abs{\dlt_{\beta} \left(\left(1+(\Delta_\alp f_4)^2\right)^{-\frac52}\right)}\lesssim \Abs{\Dlt_\alp \dlt_\beta f_4}
\end{equation*}
for $\bt\in\bbR^2$. Using this inequality and \eqref{fundamental thm of cal}, and recalling \eqref{formular: difference of three product},
we deduce
\begin{equation}\label{splitting for beta- linear}
    \Abs{\dlt_{\beta} \left(\frac{\Dlt_\alp f_2 \Dlt_\alp f_3}{\left(1+\left(\Dlt_\alp f_4\right)^2\right)^{\frac52}}\right)}\lesssim \nrmb{\nb f_3}_{L^\infty}\Abs{\Dlt_\alp \dlt_\beta f_2} +\nrmb{\nb f_2}_{L^\infty}\Abs{\Dlt_\alp \dlt_\beta f_3} + \nrmb{\nb f_2}_{L^\infty}\nrmb{\nb f_3}_{L^\infty}\Abs{\Dlt_\alp \dlt_\beta f_4}.
\end{equation}
Thus with the aid of \eqref{commutator identity for fractional laplacian}, we have
\begin{equation*}
    \begin{split}
        \Abs{Q^{t_1,7}[f_1,f_2,f_3,f_4]} &\lesssim \int_{\bbR^2} \int_{\bbR^2} \Abs{\alp \cdot \nabla_y \Delta_\alp \dlt_\bt f_1 }\Abs{\dlt_{\beta} \left(\frac{\Dlt_\alp k^{Lin}_{s}(|y|) \Dlt_\alp g(|y|)}{\left(1+\left(\Dlt_\alp k^{Lin}_{s}(|y|)\right)^2\right)^{\frac52}}\right)}\frac{d\bt}{|\bt|^{2+t_1}}\frac{d\alp}{|\alp|^2} \\
        &\lesssim  \nrmb{\nb f_3}_{L^\infty}\int_{\bbR^2} \int_{\bbR^2} \Abs{\nabla_y \delta_\alp \dlt_\bt f_1 }\Abs{\dlt_\alp \dlt_\beta f_2}\frac{d\bt}{|\bt|^{2+t_1}}\frac{d\alp}{|\alp|^3} \\
        &\quad+\nrmb{\nb f_2}_{L^\infty}\int_{\bbR^2} \int_{\bbR^2} \Abs{\nabla_y \delta_\alp \dlt_\bt f_1 }\Abs{\dlt_\alp \dlt_\beta f_3}\frac{d\bt}{|\bt|^{2+t_1}}\frac{d\alp}{|\alp|^3} \\
        &\quad+\nrmb{\nb f_2}_{L^\infty}\nrmb{\nb f_3}_{L^\infty}\int_{\bbR^2} \int_{\bbR^2} \Abs{\nabla_y \delta_\alp \dlt_\bt f_1 }\Abs{\dlt_\alp \dlt_\beta f_4}\frac{d\bt}{|\bt|^{2+t_1}}\frac{d\alp}{|\alp|^3} \\
        &= \mrI + \mrII+ \mrIII.
    \end{split}
\end{equation*}
For $\mrI$, using \eqref{sobolev embedding-2} and \eqref{eq: prelemma-1}, we estimate
\begin{equation*}
    \begin{split}
       \normb{\mrI}&\lesssim  \nrmb{\nb f_3}_{L^\infty}\int_{\bbR^2}\int_{\bbR^2} \nrmb{\nabla_y \dlt_\alp \dlt_\bt f_1}_{\dot{H}^\frac14} \nrmb{\dlt_\alp \dlt_\beta f_2}_{\dot{H}^\frac34}  \frac{d\bt}{|\bt|^{2+t_1}}\frac{d\alp}{|\alp|^{3}} \\
       &\lesssim \nrmb{\nb f_3}_{L^\infty}\left( \int_{\bbR^2}\int_{\bbR^2} \nrmb{\dlt_\alp \dlt_\bt f_1}^2_{\dot{H}^\frac54}   \frac{d\bt}{|\bt|^{2+t_1}}\frac{d\alp}{|\alp|^{\frac{5}{2}}}\right)^\frac12
       \left(\int_{\bbR^2}\int_{\bbR^2} \nrmb{\dlt_\alp \dlt_\beta f_2}^2_{\dot{H}^\frac34}  \frac{d\bt}{|\bt|^{2+t_1}}\frac{d\alp}{|\alp|^{\frac{7}{2}}}\right)^\frac12 \\
       &\lesssim \nrmb{\nb f_3}_{L^\infty} \nrmb{f_1}_{\dot{H}^{\frac{t_1+3}{2}}}\nrmb{f_2}_{\dot{H}^{\frac{t_1+3}{2}}}.
    \end{split}
\end{equation*}
With the same argument, we can show
\begin{equation*}
        \normb{\mrII}\lesssim \nrmb{\nb f_2}_{L^\infty} \nrmb{f_1}_{\dot{H}^{\frac{t_1+3}{2}}}\nrmb{f_3}_{\dot{H}^{\frac{t_1+3}{2}}}, \quad \normb{\mrIII}\lesssim \nrmb{\nb f_2}_{L^\infty} \nrmb{\nb f_3}_{L^\infty}\nrmb{f_1}_{\dot{H}^{\frac{t_1+3}{2}}}\nrmb{f_4}_{\dot{H}^{\frac{t_1+3}{2}}}.
\end{equation*}
Here, noticing $t_1\in(3/2,2)$, we see
\begin{equation*}
    \frac{t_1+3}{2}\in(2,t_1+1).
\end{equation*}
\end{proof}

\subsection{The nonlinear operator $\calN[g]$}\label{sec: nonlinear term}
In this subsection, we prove Lemma \ref{lem: l^2 estimate for T_>2(g)}. 
Recalling \eqref{def: subTs}, we can observe
\begin{equation}\label{taylor for nonlinear}
\begin{split}
    \calT_{\ge 2}[g_2]-\calT_{\ge 2}[g_1]&= \int_{\tau_1=0}^1 \left.\frac{d^2}{d\tau_1 d\tau_2} \calT[k^{Lin}_{s}+\tau_1g_1+\tau_2(g_2-g_1)]\right|_{\tau_2=0} d\tau_1 \\
    &\quad+ \int_{\tau_2=0}^1 (1-\tau_2)\left.\frac{d^2}{ d\tau_2^2} \calT[k^{Lin}_{s}+\tau_1g_1+\tau_2(g_2-g_1)]\right|_{\tau_1=1} d\tau_2.
    \end{split}
\end{equation}
by Taylor's Theorem. Hence we aim to estimate of $\dot{H}^{t_1}$-norms ($3/2<t_1<2$) of integrands in the above equation. Hereafter, we denote
\begin{equation}\label{h=klin+g_1+g_2-g_1}
\omg=\omg(k^{Lin}_{s},g_1,g_2,\tau_1,\tau_2)=k^{Lin}_{s}+\tau_1g_1+\tau_2(g_2-g_1)
\end{equation}
for simplicity. Then recalling \eqref{def: calQ}, we compute
\begin{equation}\label{computation with R}
\begin{split}
\frac{d^2}{d\tau_1 d\tau_2} \calT[\omg]&=-\frac{3}{2\pi}\left(\calQ[g_2-g_1,g_1,\omg,\omg]+\calQ[g_1,g_2-g_1,\omg,\omg]+\calQ[\omg,g_1,g_2-g_1\omg]-5\calR[\omg,g_1,g_2-g_1,\omg]\right),\\
\frac{d^2}{d\tau_2^2} \calT[\omg]&=-\frac{3}{2\pi}\left(2\calQ[g_2-g_1,g_2-g_1,\omg,\omg]+\calQ[\omg,g_2-g_1,g_2-g_1,\omg]-5\calR[\omg,g_2-g_1,g_2-g_1,\omg]\right),
\end{split}
\end{equation}
where 
\begin{equation}\label{def: calR}
       \calR[f_1,f_2,f_3,f_4]=\int_{\bbR^2} \alp \cdot \nb_y \Dlt_\alp f_1 \frac{\Dlt_\alp f_2 \Dlt_\alp f_3 \left(\Dlt_\alp f_4\right)^2 }{\left(1+\left(\Dlt_\alp f_4\right)^2\right)^\frac72} \frac{d\alp}{|\alp|^2}.
\end{equation}

Thus, Lemma \ref{lem: l^2 estimate for T_>2(g)} follows from Proposition \ref{proposition: key for linear} and Proposition \ref{proposition: key for nonlinear}:

\begin{proposition}\label{proposition: key for nonlinear}
     For any  $t_1\in (3/2,2)$, there exist $t^{*}=t^*(t_1)\in(t_1,2)$ and $t^{**}=t^{**}(t_1)\in(2,t_1+1)$ such that
    \begin{equation*}\label{key estimate-nonlinear}
    \begin{split}
&\nrmb{ \calR[f_1,f_2,f_3,f_4]}_{\dot{H}^{t_1}}\\
&\quad\lesssim \nrmb{f_1}_{\dot{H}^{t^{**}}\cap\dot{H}^{t_1+1}}\nrmb{f_2}_{\dot{W}^{1,\infty}\cap\dot{H}^{t^{*}}\cap\dot{H}^{t_1+1}}\nrmb{f_3}_{\dot{W}^{1,\infty}\cap\dot{H}^{t^{**}}\cap\dot{H}^{t_1+1}}\left( 1 +\nrmb{f_4}_{\dot{H}^{t^{**}}\cap\dot{H}^{t_1+1}}\right)
\end{split}
    \end{equation*}
\end{proposition}
\begin{remark}\label{rmk: proposition-3}
Proposition \ref{proposition: key for linear} and Proposition \ref{proposition: key for nonlinear} together with \eqref{h=klin+g_1+g_2-g_1}, \eqref{computation with R},  \eqref{sobolev embedding-3(infty)}, and $t^{**}>2$ imply that for any $\tau_1,\tau_2\in[0,1]$,
\begin{equation*}
    \begin{split}
        \nrmb{\frac{d^2}{d\tau_1 d\tau_2} \calT[\omg]}_{\dot{H}^{t_1}}&\lesssim \left(s^2+\nrmb{g_1}_{\dot{H}^{t^*} \cap \dot{H}^{t_1+1}}^2+\nrmb{g_2}_{\dot{H}^{t^*} \cap \dot{H}^{t_1+1}}^2\right)\nrmb{g_1-g_2}_{\dot{H}^{t^*} \cap \dot{H}^{t_1+1}},\\
        \nrmb{\frac{d^2}{d\tau_2^2} \calT[\omg]}_{\dot{H}^{t_1}}&\lesssim \left(s^2+\nrmb{g_1}_{\dot{H}^{t^*} \cap \dot{H}^{t_1+1}}^2+\nrmb{g_2}_{\dot{H}^{t^*} \cap \dot{H}^{t_1+1}}^2\right)\nrmb{g_1-g_2}_{\dot{H}^{t^*} \cap \dot{H}^{t_1+1}}
    \end{split}
\end{equation*}
under the assumption \eqref{ineq: g_1+g_2 < 1}. Hence these estimates with \eqref{taylor for nonlinear} lead us to Lemma \ref{lem: l^2 estimate for T_>2(g)}.
\end{remark}
As the first step toward Proposition \ref{proposition: key for nonlinear}, we reformulate $\Lmb^{t_1}\calR$:
\begin{lemma}\label{lem: reformulation for nonlinear N}
For any $t_1\in (3/2,2)$, there holds
\begin{equation*}
    \Lmb^{t_1}\calR[f_1,f_2,f_3,f_4]=\sum_{j=1}^7 R^{t_1,j}[f_1,f_2,f_3,f_4],
\end{equation*}
where
\begin{equation*}\label{splitting of R_t_1}
    \begin{split}
        &R^{t_1,1}[f_1,f_2,f_3,f_4]=-\frac{1}{2}\int_{\bbR^2}\left( \Delta_\alp \Lambda^{t_1} f_1+\Delta_{-\alp} \Lambda^{t_1} f_1\right)\frac{\frac{\alp}{|\alp|}\cdot \nb_y  f_2 \frac{\alp}{|\alp|}\cdot \nb_y  f_3 \left(\frac{\alp}{|\alp|}\cdot \nb_y  f_4\right)^2}{\left(1+\left(\frac{\alp}{|\alp|}\cdot \nb_y f_4\right)^2\right)^\frac72}\frac{d\alp}{|\alp|^2},\\
        &R^{t_1,2}[f_1,f_2,f_3,f_4]= \int_{B_1(0)}\alp \cdot \nabla_y \Lambda^{t_1} f_1 \left(\frac{\Dlt_\alp f_2 \Dlt_\alp f_3 \left(\Dlt_\alp f_4\right)^2}{\left(1+\left(\Dlt_\alp f_4\right)^2\right)^\frac72} -\frac{\frac{\alp}{|\alp|}\cdot \nb_y  f_2 \frac{\alp}{|\alp|}\cdot \nb_y  f_3 \left(\frac{\alp}{|\alp|}\cdot \nb_y  f_4\right)^2}{\left(1+\left(\frac{\alp}{|\alp|}\cdot \nb_y f_4\right)^2\right)^\frac72}\right)\frac{d\alp}{|\alp|^3},\\ 
        &R^{t_1,3}[f_1,f_2,f_3,f_4]= \int_{\bbR^2\backslash B_1(0)}\alp \cdot \nabla_y \Lambda^{t_1} f_1 \frac{\Dlt_\alp f_2 \Dlt_\alp f_3 \left(\Dlt_\alp f_4\right)^2}{\left(1+\left(\Dlt_\alp f_4\right)^2\right)^\frac72}\frac{d\alp}{|\alp|^3}, \\
        &R^{t_1,4}[f_1,f_2,f_3,f_4]= -\int_{\bbR^2} \Delta_\alp \Lambda^{t_1} f_1\left(\frac{\Dlt_\alp f_2 \Dlt_\alp f_3 \left(\Dlt_\alp f_4\right)^2}{\left(1+\left(\Dlt_\alp f_4\right)^2\right)^\frac72}-\frac{\frac{\alp}{|\alp|}\cdot \nb_y  f_2 \frac{\alp}{|\alp|}\cdot \nb_y  f_3 \left(\frac{\alp}{|\alp|}\cdot \nb_y  f_4\right)^2}{\left(1+\left(\frac{\alp}{|\alp|}\cdot \nb_y f_4\right)^2\right)^\frac72}\right)\frac{d\alp}{|\alp|^2},\\ 
        &R^{t_1,5}[f_1,f_2,f_3,f_4]= \int_{\bbR^2}\Dlt_\alp \Lmb^{t_1}f_1\alp \cdot \nb_\alp \left(\frac{\Dlt_\alp f_2 \Dlt_\alp f_3 \left(\Dlt_\alp f_4\right)^2}{\left(1+\left(\Dlt_\alp f_4\right)^2\right)^\frac72} \right)\frac{d\alp}{|\alp|^2}.\\
        &R^{t_1,6}[f_1,f_2,f_3,f_4]=\int_{\bbR^2} \alp \cdot \nb_y  \Dlt_\alp  f_1 \Lmb^{t_1}\left(\frac{\Dlt_\alp f_2 \Dlt_\alp f_3 \left(\Dlt_\alp f_4\right)^2}{\left(1+\left(\Dlt_\alp f_4\right)^2\right)^\frac72} \right)\frac{d\alp}{|\alp|^2},\\
        &R^{t_1,7}[f_1,f_2,f_3,f_4]=\int_{\bbR^2} \Lmb^{t_1}\left(\alp \cdot \nb_y \Dlt_\alp f_1 \frac{\Dlt_\alp f_2 \Dlt_\alp f_3 \left(\Dlt_\alp f_4\right)^2}{\left(1+\left(\Dlt_\alp f_4\right)^2\right)^\frac72}\right) \\
        &\quad\qquad\quad\qquad\qquad- \alp \cdot \nb_y \Dlt_\alp \Lambda^{t_1} f_1 \frac{\Dlt_\alp f_2 \Dlt_\alp f_3 \left(\Dlt_\alp f_4\right)^2}{\left(1+\left(\Dlt_\alp f_4\right)^2\right)^\frac72} -\alp \cdot \nb_y  \Dlt_\alp  f_1 \Lmb^{t_1}\left(\frac{\Dlt_\alp f_2 \Dlt_\alp f_3 \left(\Dlt_\alp f_4\right)^2}{\left(1+\left(\Dlt_\alp f_4\right)^2\right)^\frac72} \right)\frac{d\alp}{|\alp|^2}.
    \end{split}
\end{equation*}
\end{lemma}
\begin{proof}
    We split
\begin{equation*}
    \begin{split}
        \Lmb^{t_1}\calR[f_1,f_2,f_3,f_4]&=\int_{\bbR^2} \alp \cdot \nabla_y \Delta_\alp \Lambda^{t_1} f_1\frac{\frac{\alp}{|\alp|}\cdot \nb_y  f_2 \frac{\alp}{|\alp|}\cdot \nb_y  f_3 \left(\frac{\alp}{|\alp|}\cdot \nb_y  f_4\right)^2}{\left(1+\left(\frac{\alp}{|\alp|}\cdot \nb_y f_4\right)^2\right)^\frac72}\frac{d\alp}{|\alp|^2} \\
        &\quad+ \int_{\bbR^2}\alp \cdot \nabla_y \Dlt_\alp\Lambda^{t_1} f_1 \left(\frac{\Dlt_\alp f_2 \Dlt_\alp f_3 \left(\Dlt_\alp f_4\right)^2}{\left(1+\left(\Dlt_\alp f_4\right)^2\right)^\frac72} -\frac{\frac{\alp}{|\alp|}\cdot \nb_y  f_2 \frac{\alp}{|\alp|}\cdot \nb_y  f_3 \left(\frac{\alp}{|\alp|}\cdot \nb_y  f_4\right)^2}{\left(1+\left(\frac{\alp}{|\alp|}\cdot \nb_y f_4\right)^2\right)^\frac72}\right)\frac{d\alp}{|\alp|^3}\\
        &\quad+\sum_{j=6}^7 R^{t_1,j}[f_1,f_2,f_3,f_4]\\
        &=\mrI+\mrII+\sum_{j=6}^7R^{t_1,j}[f_1,f_2,f_3,f_4].
    \end{split}
\end{equation*}
For $\mrI$,  noticing
\begin{equation*}
       \int_{\bbR^2} \alp \cdot \nabla_y \Lambda^{t_1} f_1\frac{\frac{\alp}{|\alp|}\cdot \nb_y  f_2 \frac{\alp}{|\alp|}\cdot \nb_y  f_3 \left(\frac{\alp}{|\alp|}\cdot \nb_y  f_4\right)^2}{\left(1+\left(\frac{\alp}{|\alp|}\cdot \nb_y f_4\right)^2\right)^\frac72}\frac{d\alp}{|\alp|^3} =0
\end{equation*}
by symmetry and applying \eqref{change y to alpha derivative}, we have
\begin{equation*}
    \mrI=-\int_{\bbR^2} \alp \cdot \nabla_\alp \dlt_\alp \Lambda^{t_1} f_1\frac{\frac{\alp}{|\alp|}\cdot \nb_y  f_2 \frac{\alp}{|\alp|}\cdot \nb_y  f_3 \left(\frac{\alp}{|\alp|}\cdot \nb_y  f_4\right)^2}{\left(1+\left(\frac{\alp}{|\alp|}\cdot \nb_y f_4\right)^2\right)^\frac72}\frac{d\alp}{|\alp|^3}.
\end{equation*}
We integrate by parts and employ 
    \begin{equation}\label{derivatives of alpha functions-nonlinear}
        \alp \cdot \nb_\alp  \left( \frac{\frac{\alp}{|\alp|}\cdot \nb_y  f_2 \frac{\alp}{|\alp|}\cdot \nb_y  f_3 \left(\frac{\alp}{|\alp|}\cdot \nb_y  f_4\right)^2}{\left(1+\left(\frac{\alp}{|\alp|}\cdot \nb_y f_4\right)^2\right)^\frac72}\right) =0.
\end{equation}
to obtain
\begin{equation*}
\begin{split}
        \mrI &=  \int_{\bbR^2}\dlt_\alp\Lambda^{t_1} f_1 \,\nb_\alp \cdot \left(\frac{\alp}{|\alp|^3} \frac{\frac{\alp}{|\alp|}\cdot \nb_y  f_2 \frac{\alp}{|\alp|}\cdot \nb_y  f_3 \left(\frac{\alp}{|\alp|}\cdot \nb_y  f_4\right)^2}{\left(1+\left(\frac{\alp}{|\alp|}\cdot \nb_y f_4\right)^2\right)^\frac72}\right)d\alp \\
        &=- \int_{\bbR^2}\dlt_\alp\Lambda^{t_1} f_1 \frac{\frac{\alp}{|\alp|}\cdot \nb_y  f_2 \frac{\alp}{|\alp|}\cdot \nb_y  f_3\left(\frac{\alp}{|\alp|}\cdot \nb_y  f_4\right)^2}{\left(1+\left(\frac{\alp}{|\alp|}\cdot \nb_y f_4\right)^2\right)^\frac72}\frac{d\alp}{|\alp|^3}.
        \end{split}
\end{equation*}
Making a change of variables $\alp \mapsto -\alp$, we observe
\begin{equation*}
\int_{\bbR^2}\dlt_\alp\Lambda^{t_1} f_1 \frac{\frac{\alp}{|\alp|}\cdot \nb_y  f_2 \frac{\alp}{|\alp|}\cdot \nb_y  f_3 \left(\frac{\alp}{|\alp|}\cdot \nb_y  f_4\right)^2}{\left(1+\left(\frac{\alp}{|\alp|}\cdot \nb_y f_4\right)^2\right)^\frac72}\frac{d\alp}{|\alp|^3}=\int_{\bbR^2}\dlt_{-\alp}\Lambda^{t_1} f_1 \frac{\frac{\alp}{|\alp|}\cdot \nb_y  f_2 \frac{\alp}{|\alp|}\cdot \nb_y  f_3 \left(\frac{\alp}{|\alp|}\cdot \nb_y  f_4\right)^2}{\left(1+\left(\frac{\alp}{|\alp|}\cdot \nb_y f_4\right)^2\right)^\frac72}\frac{d\alp}{|\alp|^3}
\end{equation*}
which implies $\mrI= R^{t_1,1}[f_1,f_2,f_3,f_4]$.
We decompose $\mrII$ into
\begin{equation*}
    \begin{split}
        \mrII &= R^{t_1,2}[f_1,f_2,f_3,f_4] + \mrII_1 + \mrII_2,
    \end{split}
\end{equation*}
where
\begin{equation*}
    \begin{split}
        \mrII_1 &= \int_{\bbR^2\backslash B_1(0)}\alp \cdot \nabla_y  \Lambda^{t_1} f_1\left(\frac{\Dlt_\alp f_2 \Dlt_\alp f_3 \left(\Dlt_\alp f_4\right)^2}{\left(1+\left(\Dlt_\alp f_4\right)^2\right)^\frac72} -\frac{\frac{\alp}{|\alp|}\cdot \nb_y  f_2 \frac{\alp}{|\alp|}\cdot \nb_y  f_3 \left(\frac{\alp}{|\alp|}\cdot \nb_y  f_4\right)^2}{\left(1+\left(\frac{\alp}{|\alp|}\cdot \nb_y f_4\right)^2\right)^\frac72}\right)\frac{d\alp}{|\alp|^3},\\
        \mrII_2 &= -\int_{\bbR^2}\alp \cdot \nabla_y  \Lambda^{t_1} f_1(y-\alp) \left(\frac{\Dlt_\alp f_2 \Dlt_\alp f_3 \left(\Dlt_\alp f_4\right)^2}{\left(1+\left(\Dlt_\alp f_4\right)^2\right)^\frac72} -\frac{\frac{\alp}{|\alp|}\cdot \nb_y  f_2 \frac{\alp}{|\alp|}\cdot \nb_y  f_3 \left(\frac{\alp}{|\alp|}\cdot \nb_y  f_4\right)^2}{\left(1+\left(\frac{\alp}{|\alp|}\cdot \nb_y f_4\right)^2\right)^\frac72}\right)\frac{d\alp}{|\alp|^3}.
    \end{split}
\end{equation*}
We see $\mrII_1=R^{t_1,3}[f_1,f_2,f_3,f_4]$ since
\begin{equation*}
        \int_{\bbR^2\backslash B_1(0)}\alp \cdot \nabla_y \Lambda^{t_1} f_1\frac{\frac{\alp}{|\alp|}\cdot \nb_y  f_2 \frac{\alp}{|\alp|}\cdot \nb_y  f_3 \left(\frac{\alp}{|\alp|}\cdot \nb_y  f_4\right)^2}{\left(1+\left(\frac{\alp}{|\alp|}\cdot \nb_y f_4\right)^2\right)^\frac72}\frac{d\alp}{|\alp|^3}=0
\end{equation*}
by symmetry.
For $\mrII_2$, we apply \eqref{change y to alpha derivative} and then integrate by parts to obtain
\begin{equation*}
    \mrII_2=\frac{1}{2\pi}\int_{\bbR^2}\dlt_\alp\Lambda^{t_1} f_1 \,\nb_\alp \cdot \left(\frac{\alp}{|\alp|^3} \left(\frac{\Dlt_\alp f_2 \Dlt_\alp f_3 \left(\Dlt_\alp f_4\right)^2}{\left(1+\left(\Dlt_\alp f_4\right)^2\right)^\frac72} -\frac{\frac{\alp}{|\alp|}\cdot \nb_y  f_2 \frac{\alp}{|\alp|}\cdot \nb_y  f_3 \left(\frac{\alp}{|\alp|}\cdot \nb_y  f_4\right)^2}{\left(1+\left(\frac{\alp}{|\alp|}\cdot \nb_y f_4\right)^2\right)^\frac72} \right)\right)d\alp.
\end{equation*}
Computing with the aids of \eqref{derivatives of alpha functions-1} and \eqref{derivatives of alpha functions-nonlinear}, we arrive at
\begin{equation*}
    \mrII_2 = R^{t_1,4}[f_1,f_2,f_3,f_4]+R^{t_1,5}[f_1,f_2,f_3,f_4].
\end{equation*}
\end{proof}

The estimates of $R^{t_1,j}$ $(1\le j \le7)$ range over Lemma \ref{lem: R^1-5} - Lemma \ref{lem: R^7}, which leads to Proposition \ref{proposition: key for nonlinear} by choosing the smallest $t^*$ and $t^{**}$ among those in Lemma \ref{lem: R^1-5} - Lemma \ref{lem: R^7} due to \eqref{sobolev interpolation}.
\begin{lemma}\label{lem: R^1-5}
  For any  $t_1\in (3/2,2)$, there exist $t^{*}=t^*(t_1)\in(t_1,2)$ and $t^{**}=t^{**}(t_1)\in(2,t_1+1)$ such that
  \begin{equation*}
  \begin{split}
    \sum_{j=1}^5&\normb{R^{t_1,j}[f_1,f_2,f_3,f_4]}\\
    &\lesssim \nrmb{f_1}_{\dot{H}^{t^{**}}\cap\dot{H}^{t_1+1}}\nrmb{f_2}_{\dot{W}^{1,\infty}\cap\dot{H}^{t^{*}} \cap \dot{H}^{t_1+1} }\nrmb{f_3}_{\dot{W}^{1,\infty}\cap\dot{H}^{t^{**}}\cap\dot{H}^{t_1+1}}\left(1+\nrmb{f_4}_{\dot{H}^{t^{**}}\cap\dot{H}^{t_1+1}}\right).
    \end{split}
\end{equation*}
\end{lemma}
\begin{proof}
To estimate $R^{t_1,1}[f_1,f_2,f_3,f_4]$, we observe that the factor
$\frac{\frac{\alp}{|\alp|}\cdot \nb_y  f_2 \frac{\alp}{|\alp|}\cdot \nb_y  f_3 \left(\frac{\alp}{|\alp|}\cdot \nb_y  f_4\right)^2}{\left(1+\left(\frac{\alp}{|\alp|}\cdot \nb_y f_4\right)^2\right)^\frac72}$ is independent of the length $|\alp|$, so that for $\alp=r\sigma$
\begin{equation*}
    |R^{t_1,1}[f_1,f_2,f_3,f_4]| \lesssim |\nb_y f_2| |\nb_y f_3|\int_{\bbS^1} \Abs{\int_{0}^\infty \Delta_{r\sigma} \Lambda^{t_1} f_1(y)+\Delta_{-r\sigma} \Lambda^{t_1} f_1(y)\frac{dr}{r}} d\sigma.
\end{equation*}
Thus using the Minkowski's inequality, we have
\begin{equation*}
    \nrmb{R^{t_1,1}[f_1,f_2,f_3,f_4]}_{L^2_y}\lesssim \nrmb{\nb_y f_2}_{L^{\infty}_y}\nrmb{\nb_y f_3}_{L^{\infty}_y}\int_{\bbS^1}\nrmb{\int_{0}^\infty \Delta_{r\sigma} \Lambda^{t_1} f_1(y)+\Delta_{-r\sigma} \Lambda^{t_1} f_1 (y)\frac{dr}{r}}_{L^2_y}  d\sigma.
\end{equation*}
Now using \eqref{key est: for T^1(klin) to make t_1+1}, we obtain
\begin{equation*}
    \nrmb{R^{t_1,1}[f_1,f_2,f_3,f_4]}_{L^2_y}\lesssim \nrmb{f_1}_{\dot{H}^{t_1+1}} \nrmb{ f_2}_{\dot{W}^{1,\infty}}\nrmb{f_3}_{\dot{W}^{1,\infty}}.
\end{equation*}

For $R^{t_1,2}[f_1,f_2,f_3,f_4]$, applying the mean value theorem to the function $F(x)=x^2(1+x^2)^{-\frac72}$, we can see 
    \begin{equation*}
        \left|(\Delta_\alp f_4)^2\left(1+ (\Delta_\alp f_4)^2\right)^{-\frac72}-\left(\frac{\alp}{|\alp|}\cdot \nb_y  f_4\right)^2\left(1+ \left(\frac{\alp}{|\alp|}\cdot \nb_y  f_4\right)^2\right)^{-\frac72} \right| \lesssim \Abs{\Delta_\alp f_4-\frac{\alp}{|\alp|}\cdot \nb_y  f_4}.
    \end{equation*}
Based on this inequality and \eqref{formular: difference of three product},
we deduce
\begin{equation}\label{spliting f_2f_3(w)^2/(1+w)}
\begin{split}
    &\Abs{\frac{\Dlt_\alp f_2 \Dlt_\alp f_3 \left(\Dlt_\alp f_4\right)^2}{\left(1+\left(\Dlt_\alp f_4\right)^2\right)^\frac72} -\frac{\frac{\alp}{|\alp|}\cdot \nb_y  f_2 \frac{\alp}{|\alp|}\cdot \nb_y  f_3 \left(\frac{\alp}{|\alp|}\cdot \nb_y  f_4\right)^2}{\left(1+\left(\frac{\alp}{|\alp|}\cdot \nb_y f_4\right)^2\right)^\frac72}}\\
    &\quad\lesssim \Abs{\Delta_\alp f_2-\frac{\alp}{|\alp|}\cdot \nb_y  f_2}\Abs{\Delta_\alp f_3+\frac{\alp}{|\alp|}\cdot \nb_y  f_3} +\Abs{\Delta_\alp f_2+\frac{\alp}{|\alp|}\cdot \nb_y  f_2}\Abs{\Delta_\alp f_3-\frac{\alp}{|\alp|}\cdot \nb_y  f_3} \\
    &\qquad+\Abs{\Delta_\alp f_2+\frac{\alp}{|\alp|}\cdot \nb_y  f_2}\Abs{\Delta_\alp f_3+\frac{\alp}{|\alp|}\cdot \nb_y  f_3}\Abs{\Delta_\alp f_4-\frac{\alp}{|\alp|}\cdot \nb_y  f_4}\\
    &\qquad+\Abs{\Delta_\alp f_2-\frac{\alp}{|\alp|}\cdot \nb_y  f_2}\Abs{\Delta_\alp f_3-\frac{\alp}{|\alp|}\cdot \nb_y  f_3}\Abs{\Delta_\alp f_4-\frac{\alp}{|\alp|}\cdot \nb_y  f_4}.
    \end{split}
\end{equation}
Hence applying the same argument with the estimation of $\nrmb{Q^{t_1,2}[f_1,f_2,f_3,f_4]}_{L^2_y}$ in Lemma \ref{lem: Q^1-5}, we have
\begin{equation*}
    \nrmb{R^{t_1,2}[f_1,f_2,f_3,f_4]}_{L^2_y}\lesssim \nrmb{f_1}_{\dot{H}^{t_1+1}}\nrmb{f_2}_{\dot{W}^{1,\infty}\cap\dot{H}^{t_1+1}}\nrmb{f_3}_{\dot{W}^{1,\infty}\cap\dot{H}^{t_1+1}}\left(1+\nrmb{f_4}_{\dot{H}^{t_1+1}}\right).
\end{equation*}

For $R^{t_1,3}[f_1,f_2,f_3,f_4]$, we use \eqref{fundamental thm of cal} and \eqref{Morrey} to observe
\begin{equation*}
    \Abs{\frac{\Dlt_\alp f_2 \Dlt_\alp f_3 \left(\Dlt_\alp f_4\right)^2}{\left(1+\left(\Dlt_\alp f_4\right)^2\right)^\frac72}}\lesssim |\alp|^{\frac{t_1}{2}-1}\nrmb{\nb f_2}_{\dot{H}^{\frac{t_1}{2}}}\nrmb{\nb_y f_3}_{L^{\infty}},
\end{equation*}
which gives
\begin{equation*}
    \nrmb{R^{t_1,3}[f_1,f_2,f_3,f_4]}_{L^2_y}\lesssim \nrmb{f_1}_{\dot{H}^{t_1+1}}\nrmb{f_2}_{\dot{H}^{\frac{t_1}{2}+1}}\nrmb{f_3}_{\dot{W}^{1,\infty}}.
\end{equation*}

For $R^{t_1,4}[f_1,f_2,f_3,f_4]$, we use \eqref{fundamental thm of cal} and \eqref{spliting f_2f_3(w)^2/(1+w)} and proceed similarly to the the estimation of $\nrmb{Q^{t_1,4}[f_1,f_2,f_3,f_4]}_{L^2_y}$ in Lemma \ref{lem: Q^1-5}, we obtain
\begin{equation*}
    \normb{R^{t_1,4}[f_1,f_2,f_3,f_4]}\lesssim \nrmb{f_1}_{\dot{H}^{t_1+\frac34}}\nrmb{f_2}_{\dot{W}^{1,\infty}\cap\dot{H}^{\frac94}}\nrmb{f_3}_{\dot{W}^{1,\infty}\cap\dot{H}^{\frac94}}\left(1+\nrmb{f_4}_{\dot{H}^{\frac94}}\right).
\end{equation*}

For $R^{t_1,5}[f_1,f_2,f_3,f_4]$, we compute
\begin{equation*}
    \begin{split}
        \alp \cdot \nb_\alp \left(\frac{\Dlt_\alp f_2 \Dlt_\alp f_3 \left(\Dlt_\alp f_4\right)^2}{\left(1+\left(\Dlt_\alp f_4\right)^2\right)^\frac72} \right)
    &\quad=-\frac{ \Dlt_\alp f_3 \left(\Dlt_\alp f_4\right)^2}{\left(1+\left(\Dlt_\alp f_4\right)^2\right)^\frac72}\left(\Delta_\alp f_2-\frac{\alp \cdot \nb_y f_2}{|\alp|} + \alp \cdot \nb_y \Delta_\alp f_2\right) \\
    &\qquad-\frac{ \Dlt_\alp f_2 \left(\Dlt_\alp f_4\right)^2}{\left(1+\left(\Dlt_\alp f_4\right)^2\right)^\frac72}\left(\Delta_\alp f_3-\frac{\alp \cdot \nb_y f_3}{|\alp|} + \alp \cdot \nb_y \Delta_\alp f_3\right)\\
    &\qquad-\frac{2 \Dlt_\alp f_2 \Dlt_\alp f_3 \Dlt_\alp f_4}{\left(1+\left(\Dlt_\alp f_4\right)^2\right)^\frac72}\left(\Delta_\alp f_4-\frac{\alp \cdot \nb_y f_4}{|\alp|} + \alp \cdot \nb_y \Delta_\alp f_4\right)\\
    &\qquad+\frac{7\Dlt_\alp f_2\Dlt_\alp f_3\left(\Delta_\alp f_4\right)^3}{\left(1+ (\Delta_\alp f_4)^2\right)^{\frac92}}\left(\Delta_\alp f_4-\frac{\alp \cdot \nb_y f_4}{|\alp|} + \alp \cdot \nb_y \Delta_\alp f_4 \right),
    \end{split}
\end{equation*}
so that
\eqref{fundamental thm of cal} yields
\begin{equation*}
    \begin{split}
        \Abs{\alp \cdot \nb_\alp \left(\frac{\Dlt_\alp f_2 \Dlt_\alp f_3}{\left(1+\left(\Dlt_\alp f_4\right)^2\right)^\frac72} \right)}
    &\lesssim \nrmb{\nb f_3}_{L^\infty}\Abs{\Delta_\alp f_2-\frac{\alp \cdot \nb_y f_2}{|\alp|} + \alp \cdot \nb_y \Delta_\alp f_2} \\
    &\quad\nrmb{\nb f_2}_{L^\infty}\Abs{\Delta_\alp f_3-\frac{\alp \cdot \nb_y f_3}{|\alp|} + \alp \cdot \nb_y \Delta_\alp f_3}\\
    &\quad+\nrmb{\nb f_2}_{L^\infty}\nrmb{\nb f_3}_{L^\infty}\Abs{\Delta_\alp f_4-\frac{\alp \cdot \nb_y f_4}{|\alp|} + \alp \cdot \nb_y \Delta_\alp f_4 }.
    \end{split}
\end{equation*}
Thus we proceed similarly  to the the estimation of $\nrmb{Q^{t_1,5}[f_1,f_2,f_3,f_4]}_{L^2_y}$ in Lemma \ref{lem: Q^1-5} to obtain
\begin{equation*}
    \normb{R^{t_1,5}[f_1,f_2,f_3,f_4]}\lesssim \nrmb{f_1}_{\dot{H}^{t_1+\frac34}}\nrmb{f_2}_{\dot{W}^{1,\infty}\cap\dot{H}^{\frac94}}\nrmb{f_3}_{\dot{W}^{1,\infty}\cap\dot{H}^{\frac94}}\left(1+\nrmb{f_4}_{\dot{H}^{\frac94}}\right).
\end{equation*}
Note that since $t_1\in(3/2,2)$, the exponents of Sobolev spaces appeared above estimates satisfy \eqref{sobolev exponet for Q^1-Q^5}, which gives the desired result.
\end{proof}

\begin{lemma}\label{lem: R^6}
    For any  $t_1\in (3/2,2)$, there exist $t^{*}=t^*(t_1)\in(t_1,2)$ and $t^{**}=t^{**}(t_1)\in(2,t_1+1)$ such that
    \begin{equation*}
    \begin{split}
&\nrmb{R^{t_1,6}[f_1,f_2,f_3,f_4]}_{L^2} \\
&\qquad\lesssim \nrmb{f_1}_{\dot{H}^{t^{**}}\cap\dot{H}^{t_1+1}}\nrmb{f_2}_{\dot{H}^{t^{*}}\cap\dot{H}^{t_1+1}}\nrmb{f_3}_{\dot{W}^{1,\infty}\cap\dot{H}^{t^{**}}\cap\dot{H}^{t_1+1}}\left(1 +\nrmb{f_4}_{\dot{H}^{t^{**}}\cap\dot{H}^{t_1+1}}\right).
\end{split}
    \end{equation*}
\end{lemma}
\begin{proof}
    Using the identity $\Dlt(fg)=f\Dlt g + 2\nb f \cdot \nb g + g\Dlt f$, we compute
    \begin{equation*}
        \begin{split}
            \Dlt\left(\frac{\Dlt_\alp f_2 \Dlt_\alp f_3 \left(\Dlt_\alp f_4\right)^2}{\left(1+\left(\Dlt_\alp f_4\right)^2\right)^{\frac72}}\right)&=\left(\Dlt_\alp f_4\right)^2\left(1+\left(\Dlt_\alp f_4\right)^2\right)^{-\frac72}\Dlt\left(\Dlt_\alp f_2 \Dlt_\alp f_3\right)\\
            &\quad+2\nb_y \left(\left(\Dlt_\alp f_4\right)^2\left(1+\left(\Dlt_\alp f_4\right)^2\right)^{-\frac72}\right)\cdot\nb_y \left(\Dlt_\alp f_2 \Dlt_\alp f_3\right) \\
            &\quad+ \Dlt_\alp f_2 \Dlt_\alp f_3 \Dlt \left(\left(\Dlt_\alp f_4\right)^2\left(1+\left(\Dlt_\alp f_4\right)^2\right)^{-\frac72}\right) \\
            &=\mrI_1+\mrI_2+\mrI_3.
        \end{split}
    \end{equation*}
Using \eqref{Morrey for T_1}, we can see that $\Abs{\mrI_1}$, $\Abs{\mrI_2}$, and $\Abs{\mrI_3}$ satisfy \eqref{three cases of laplacian - linear} with constant $C$'s in the inequalities `$\lesssim=\;\le C$' adjusted if necessary.
Thus we apply the same argument with the proof of  Lemma \ref{lem: Q^6} to obtain the desired result.
\end{proof}

\begin{lemma}\label{lem: R^7}
    For any  $t_1\in (3/2,2)$, there exists $t^{**}=t^{**}(t_1)\in(2,t_1+1)$ such that
    \begin{equation*}
    \begin{split}
&\nrmb{R^{t_1,7}[f_1,f_2,f_3,f_4]}_{L^2} \\
&\quad\lesssim \nrmb{f_1}_{\dot{H}^{t^{**}}\cap\dot{H}^{t_1+1}}\nrmb{f_2}_{\dot{W}^{1,\infty}\cap \dot{H}^{t^{**}}\cap\dot{H}^{t_1+1}}\nrmb{f_3}_{\dot{W}^{1,\infty}\cap\dot{H}^{t^{**}}\cap\dot{H}^{t_1+1}} \left(1 +\nrmb{f_4}_{\dot{H}^{t^{**}}\cap\dot{H}^{t_1+1}}\right).
\end{split}
    \end{equation*}
\end{lemma}
\begin{proof}
Using the mean value theorem to the function $F(x)=x^2(1+x^2)^{-\frac72}$, we can show
\begin{equation*}
    \Abs{\dlt_{\beta} \left((\Delta_\alp f_4)^2\left(1+(\Delta_\alp f_4)^2\right)^{-\frac72}\right)}\lesssim \Abs{\Dlt_\alp \dlt_\beta f_4}
\end{equation*}
for $\bt\in\bbR^2$. Using this inequality and \eqref{fundamental thm of cal}, and recalling \eqref{formular: difference of three product},
we can check that
$\Abs{\dlt_{\beta} \left(\frac{\Dlt_\alp f_2 \Dlt_\alp f_3 \left(\Dlt_\alp f_4\right)^2}{\left(1+\left(\Dlt_\alp f_4\right)^2\right)^{\frac72}}\right)}$ satisfies \eqref{splitting for beta- linear} with constant $C$ in the inequality `$\lesssim=\;\le C$' adjusted if necessary. Therefore we apply the same argument with the proof of  Lemma \ref{lem: Q^7} to obtain the desired result.
\end{proof}

\bibliographystyle{amsplain}
\bibliography{Muskat_revised_1}

\end{document}